\documentclass[]{amsart}

\makeatletter
\@addtoreset{equation}{section}

\makeatother

\newtheorem{theorem}{Theorem}[section]
\newtheorem{proposition}[theorem]{Proposition}
\newtheorem{corollary}[theorem]{Corollary}
\newtheorem{lemma}[theorem]{Lemma}
\theoremstyle{definition}
\newtheorem{definition}[theorem]{Definition}

\theoremstyle{remark}
\newtheorem{remark}[theorem]{Remark}
\newtheorem{example}[theorem]{Example}

\DeclareMathOperator{\sech}{sech}
\DeclareMathOperator{\cn}{cn}
\DeclareMathOperator{\dn}{dn}
\DeclareMathOperator{\sn}{sn}
\DeclareMathOperator{\am}{am}

\usepackage{braket} % \Set{} is available
\usepackage{amssymb} % \varkappa is available 
\usepackage{color}
\usepackage[pdftex]{graphicx}
\usepackage{mathrsfs} % \mathscr{} is available 
\usepackage{url}

\newcommand{\lm}{\lambda}
\newcommand{\R}{\mathbf{R}}
\newcommand{\Z}{\mathbf{Z}}

\newcommand{\N}{\mathbf{N}}
\newcommand{\T}{\mathbf{T}}

\newcommand{\vp}{\varphi}

\newcommand{\vc}{\gamma}

\newcommand{\B}{\mathcal{B}}
\newcommand{\K}{\mathrm{K}}
\newcommand{\E}{\mathrm{E}}

\newcommand{\F}{\mathrm{F}}  %p-elastica
\newcommand{\amcn}{\am_{1,p}}
\newcommand{\amdn}{\am_{2,p}}
\newcommand{\Ecn}{\mathrm{E}_{1,p}}
\newcommand{\Fcn}{\mathrm{F}_{1,p}}
\newcommand{\Kcn}{\mathrm{K}_{1,p}}

\begin{document}

\title{Pinned planar $p$-elasticae}

\author{Tatsuya Miura}
\address[T.~Miura]{Department of Mathematics, Tokyo Institute of Technology, 2-12-1 Ookayama, Meguro-ku, Tokyo 152-8551, Japan}
\email{miura@math.titech.ac.jp}

\author{Kensuke Yoshizawa}
\address[K.~Yoshizawa]{Institute of Mathematics for Industry, Kyushu University, 744 Motooka, Nishi-ku, Fukuoka 819-0395, Japan}
\email{k-yoshizawa@imi.kyushu-u.ac.jp}


\date{\today}

\begin{abstract}
\if0
Building on our previous work, we classify all planar $p$-elasticae under the pinned boundary condition.
We then obtain uniqueness and geometric properties of global minimizers.
Applications are given to a Li--Yau type inequality for the $p$-bending energy and also to minimal $p$-elastic networks.
Among other results, we discover a new unique exponent $p \simeq 1.5728$ such that our Li--Yau type inequality is optimal for every multiplicity.
\fi

Building on our previous work, we classify all planar $p$-elasticae under the pinned boundary condition, and then obtain uniqueness and geometric properties of global minimizers.
As an application we establish a Li--Yau type inequality for the $p$-bending energy, and in particular discover a unique exponent $p \simeq 1.5728$ for full optimality.
We also prove existence of minimal $p$-elastic networks, extending a recent result of Dall'Acqua--Novaga--Pluda.
\end{abstract}

\maketitle

\setcounter{tocdepth}{1}
\tableofcontents

\section{Introduction}

This paper is a continuation of our study of $p$-elasticae \cite{MYarXiv2203}, wherein we have classified all planar $p$-elasticae and obtained their explicit parameterizations as well as optimal regularity.
In this paper, we turn to a boundary value problem and classify all planar $p$-elasticae subject to the pinned boundary condition, with some applications to a Li--Yau type inequality and minimal networks.

\subsection{Classification of pinned planar $p$-elasticae}

For $p\in(1,\infty)$ and immersed curves $\gamma$ in the Euclidean plane $\R^2$, the \emph{$p$-bending energy} is defined by 
\[
\mathcal{B}_p[\gamma]:=\int_{\gamma}|k|^p\,ds,
\]
where $k$ denotes the signed curvature and $s$ denotes the arclength parameter of $\gamma$, respectively.
A critical point of the $p$-bending energy under the fixed-length constraint is called $p$-elastica; in other words, its signed curvature is a (weak) solution to the Euler--Lagrange equation formally given by
\begin{align}\notag%\label{Teq:sw20-1.2}
p(p-1)|k|^{p-2} \partial_s^2 k +p(p-1)(p-2) |k|^{p-4}k (\partial_s k)^2 +(p-1)|k|^pk - \lambda k =0,
\end{align}
where $\partial_s$ denotes the arclength derivative and $\lambda\in \R$ denotes a multiplier due to the fixed-length constraint.
Classification of planar elasticae ($p=2$) is classic and given by Euler in the 18th century.
Planar $p$-elasticae are just recently classified by the authors \cite{MYarXiv2203}, with optimal regularity and closed formulae in terms of newly introduced $p$-elliptic functions.
In particular, in the degenerate case $p>2$ the obtained family includes Watanabe's flat-core solutions \cite{nabe14} (see also \cite{SW20}), which are of qualitatively novel type compared to $p=2$.
See \cite{MYarXiv2203} and references therein for details.

Boundary value problems for ($p$-)elasticae are in general more advanced because we need to detect the precise values of the multiplier or other geometric parameters for compatibility with the given boundary data, and thus need to solve the corresponding system of transcendental equations.

As a first step we focus on the so-called pinned boundary condition, which prescribes the endpoints up to zeroth order.
More precisely, for given $P_0,P_1\in\R^2$ and $L>0$ such that $|P_0-P_1|<L$, we define the admissible space by
\begin{equation}\label{eq:admissiblespace}
    \mathcal{A}_{P_0,P_1,L}:=\Set{
\gamma \in W^{2,p}_{\rm imm}(0,1; \R^2) | \gamma(0)=P_0, \ \ \gamma(1)=P_1, \ \ \mathcal{L}[\gamma]=L },
\end{equation}
where $\mathcal{L}$ denotes the length functional $\mathcal{L}[\gamma]:=\int_\gamma\,ds$, 
and $W^{2,p}_{\rm imm}(0,1;\R^2)$ denotes the set of immersed $W^{2,p}$-curves: 
\[
W^{2,p}_{\rm imm}(0,1;\R^2):=
\Set{
\gamma \in W^{2,p}(0,1; \R^2) |\  |\gamma'(t)|\neq0 \text{ for all } t\in[0,1]}.
\]
%(Note that $W^{2,p}(0,1)\subset C^1([0,1])$ by Sobolev embedding.)
We call a critical point of $\mathcal{B}_p$ in $\mathcal{A}_{P_0,P_1,L}$ \emph{pinned $p$-elastica} (cf.\ Definition \ref{critical_point}).

In the case of elastica ($p=2$) the pinned boundary value problem has been completely solved on the level of critical points and global minimizers \cite{Ydcds, Miura_LiYau} (see also \cite{Love, Lin98, AGP20}).
Roughly speaking, if $P_0=P_1$, then all pinned elasticae are part of ``figure-eight'', while if $P_0\neq P_1$, then the curves are ``arcs'' and ``loops'' (like Figure \ref{fig:arc}).
In any case, for given boundary data there are countably many critical points but the global minimizer is uniquely given by a ``single arc'' up to invariances.

Our primary result here extends the known classification of pinned planar elasticae for $p=2$ to all $p\in(1,\infty)$.
In particular, in the degenerate case $p>2$, we find a critical distance of the endpoints at which the number of pinned $p$-elasticae changes from countable to uncountable.
%This is ascribed to the transition between wavelike and flat-core types, cf.\ Proposition \ref{prop:MY2203-thm1.1}.
Following terminology in Definition~\ref{def:arcloop}, our result is summarized as follows (see also Figures~\ref{fig:arc} and \ref{fig:flat}).

\begin{theorem}[Classification of pinned $p$-elasticae]\label{thm:classify-pinned}
Let $p\in(1,\infty)$.
Let $P_0,P_1\in\R^2$ and $L>0$ such that
\begin{equation}\nonumber
    r:=\frac{|P_0-P_1|}{L}\in[0,1).
\end{equation}
Suppose that $\gamma \in \mathcal{A}_{P_0,P_1,L}$ is a critical point of $\mathcal{B}_p$ in $\mathcal{A}_{P_0,P_1,L}$.
Then the following assertions hold true.
\begin{itemize}
    \item[(i)] If $r=0$, then $\gamma$ is an $\frac{n}{2}$-fold figure-eight $p$-elastica for some $n\in \N$.
    \item[(ii)] If $r\in(0,\frac{1}{p-1})$, then $\gamma$ is either a $(p,r,n)$-arc for some $n\in \N$, or a $(p,r,n)$-loop for some $n\in \N$.
    \item[(iii)] If $r\in[\frac{1}{p-1},1)$, then $\gamma$ is either a $(p,r,n)$-arc for some $n\in \N$, or a $(p, r, n, \boldsymbol{\sigma}, \boldsymbol{L})$-flat-core for some $n\in \N$, $\boldsymbol{\sigma}=(\sigma_1, \ldots, \sigma_n) \in \{+,-\}^n$, and $\boldsymbol{L}=(L_1, \ldots, L_{n+1}) \in [0,\infty)^{n+1}$ subject to relation \eqref{eq:sum-flatparts}.
\end{itemize}
\end{theorem}

\begin{figure}[htbp]
      \includegraphics[scale=0.21]{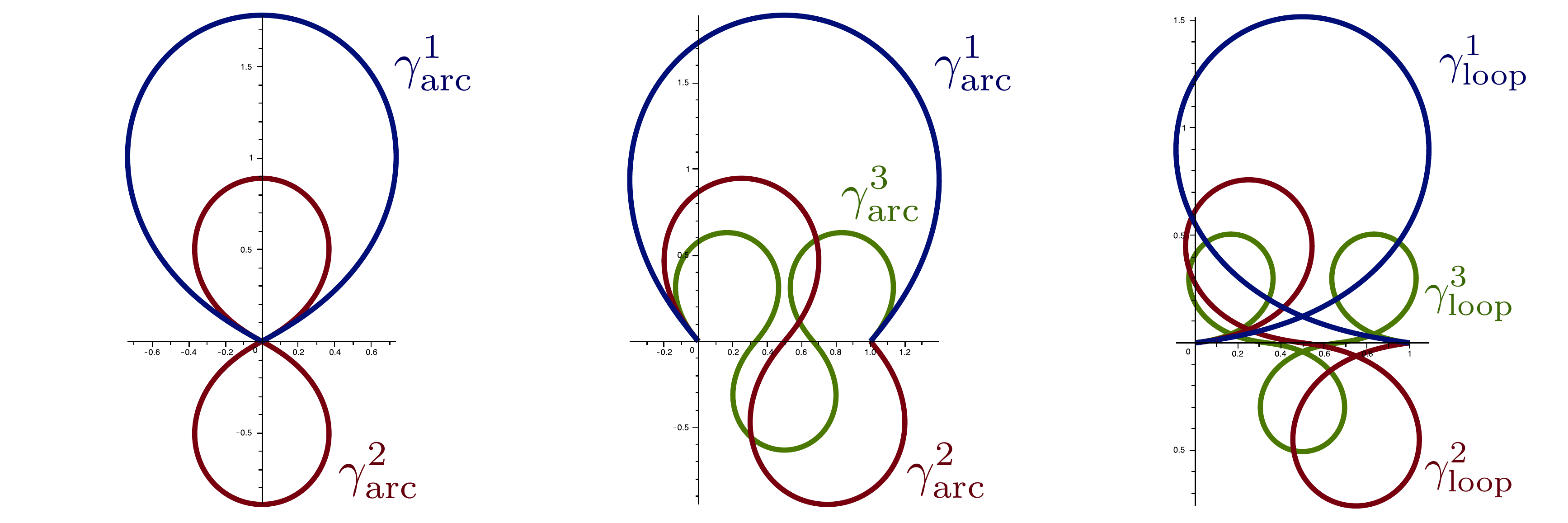}
  \caption{The left $\gamma^n_{\rm arc}$ represents an $\frac{n}{2}$-fold figure-eight $p$-elastica, the middle $\gamma^n_{\rm arc}$ a $(p,r,n)$-arc, and the right
  $\gamma^n_{\rm loop}$ a $(p,r,n)$-loop, where $p=4$ and $r=\tfrac{1}{5}$.}
  \label{fig:arc}
\end{figure}
  
\begin{figure}[htbp]
      \includegraphics[scale=0.175]{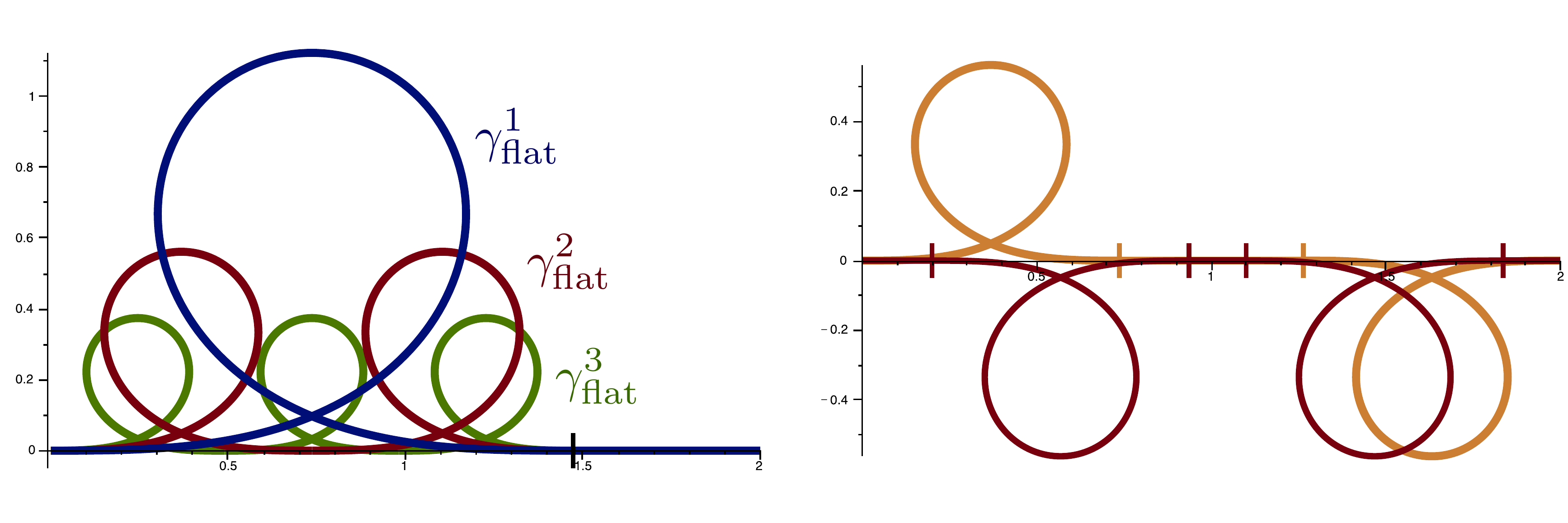}
  \caption{
  The left $\gamma^n_{\rm flat}$ represents a $(p,r,n, \boldsymbol{\sigma}, \boldsymbol{L})$-flat-core with $p=4$ and $ r=\frac{2}{5}$.
  The right two curves are also $(4,\frac{2}{5},2, \boldsymbol{\sigma}, \boldsymbol{L})$-flat-cores, with different choices of the direction of the loops $\boldsymbol{\sigma}$ and the length of the connecting segments $\boldsymbol{L}$.}
  \label{fig:flat}
\end{figure}

\begin{remark}[Countability-uncountability transition]
If $p\leq2$, then $\frac{1}{p-1}\geq1$ so the only possible cases are (i) and (ii), and hence the classification is qualitatively parallel to the classical case $p=2$.
Case (iii) is of novel type and occurs if and only if $p>2$ and the endpoints are sufficiently distant.
In particular, if $p>2$ and $r>\frac{1}{p-1}$, then there are uncountably many critical points, and otherwise countably many (up to isometry and reparameterization).
The uncountability is due to the freedom of the length $\boldsymbol{L}$ of the flat parts as in Figure \ref{fig:flat}.
\end{remark}

\begin{remark}[Loss of regularity]
Critical points in Theorem \ref{thm:classify-pinned} are always of class $C^2$ but may not $C^\infty$ in general, particularly at the points where the curvature vanishes.
More precisely, if $\frac{1}{p-1}$ is an odd integer (i.e., $p\in\{2,\frac{4}{3},\frac{6}{5},\frac{8}{7},\dots\}$), then any arclength-parameterized critical point is always of class $C^\infty$, but otherwise may not. 
For example, if $\frac{1}{p-1}$ is not an integer, then any arclength-parameterized $(p,r,n)$-arc and $(p,r,n)$-loop are not of class $C^{m_p+2}$, where $m_p:=\lceil \frac{1}{p-1} \rceil$.
We emphasize that this loss of regularity occurs even for $n=1$, since at the endpoints the curvature vanishes.
In addition, if $p>2$ and if $r>\frac{1}{p-1}$, then any arclength-parameterized $(p,r,n, \boldsymbol{\sigma}, \boldsymbol{L})$-flat-core is not of class $C^{M_p+2}$, where $M_p:=\lceil \frac{2}{p-2} \rceil$.
For more details as well as optimal regularity in terms of Sobolev class, see \cite{MYarXiv2203}.
\end{remark}

Now we focus on global minimizers.
Existence follows by a standard direct method.
In addition, comparing the $p$-bending energy of all the above critical points, we obtain uniqueness of global minimizers.

%%%%%%%%%%%%%%%%%%%%%%%%%%%%%%%%%%%%%%
\begin{theorem}[Unique existence of global minimizers]\label{thm:uniqueness}
Let $p\in(1,\infty)$, $P_0, P_1 \in \R^2$, and $L>0$ such that $|P_0-P_1|<L$.
Then there exists a global minimizer of $\mathcal{B}_p$ in $\mathcal{A}_{P_0, P_1,L}$, and in addition any global minimizer is uniquely given by a $(p,r,1)$-arc (up to isometry and reparameterization).
\end{theorem}
%%%%%%%%%%%%%%%%%%%%%%%%%%%%%%%%%%%%%%

A particularly useful case is where the endpoints agree $P_0=P_1$.
In this case, by Theorems~\ref{thm:classify-pinned} and \ref{thm:uniqueness} we deduce that a half-fold (i.e., $\frac{1}{2}$-fold) figure-eight $p$-elastica is a unique global minimizer, extending \cite[Proposition 2.6]{Miura_LiYau} from $p=2$ to $p\in(1,\infty)$ in the planar case.
Here we apply this fact to obtain a Li--Yau type inequality and existence of minimal networks in the same spirit as \cite{Miura_LiYau}.
In view of these applications it would be informative if we could know the precise geometric properties of the figure-eight $p$-elastica, in particular how the crossing angle depends on $p$.
Fortunately we succeed in ensuring the following key monotonicity, cf.\ Figure~\ref{fig:phi^*}: 

%%%%%%%%%%%%%%%%%%%%%%%%%%%%%%%%%%%%%%
\begin{theorem}[Monotonicity of the crossing angle]\label{thm:phi*-decrease}
For $p\in(1,\infty)$, let $2\phi^*(p)$ be the angle between the tangent vectors at the two endpoints of a half-fold figure-eight $p$-elastica, cf.\ \eqref{eq:phi*}. 
Then the function $p\mapsto\phi^*(p)$ is continuous and strictly decreasing from $\pi/2$ to $0$.
\end{theorem}
%%%%%%%%%%%%%%%%%%%%%%%%%%%%%%%%%%%%%%

\begin{figure}[htbp]
    \centering
    \includegraphics[scale=0.12]{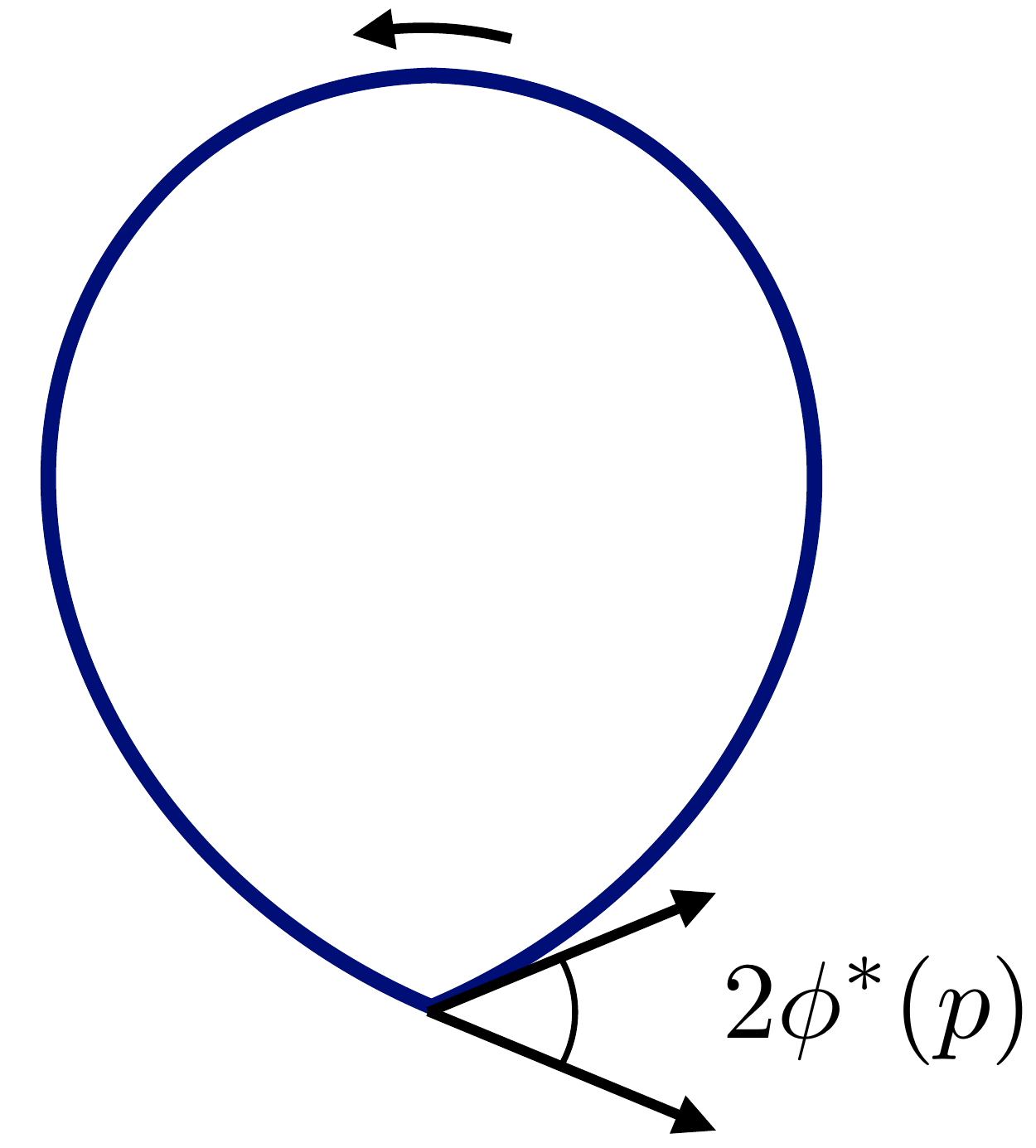}
    \hspace{20pt}
    \includegraphics[scale=0.155]{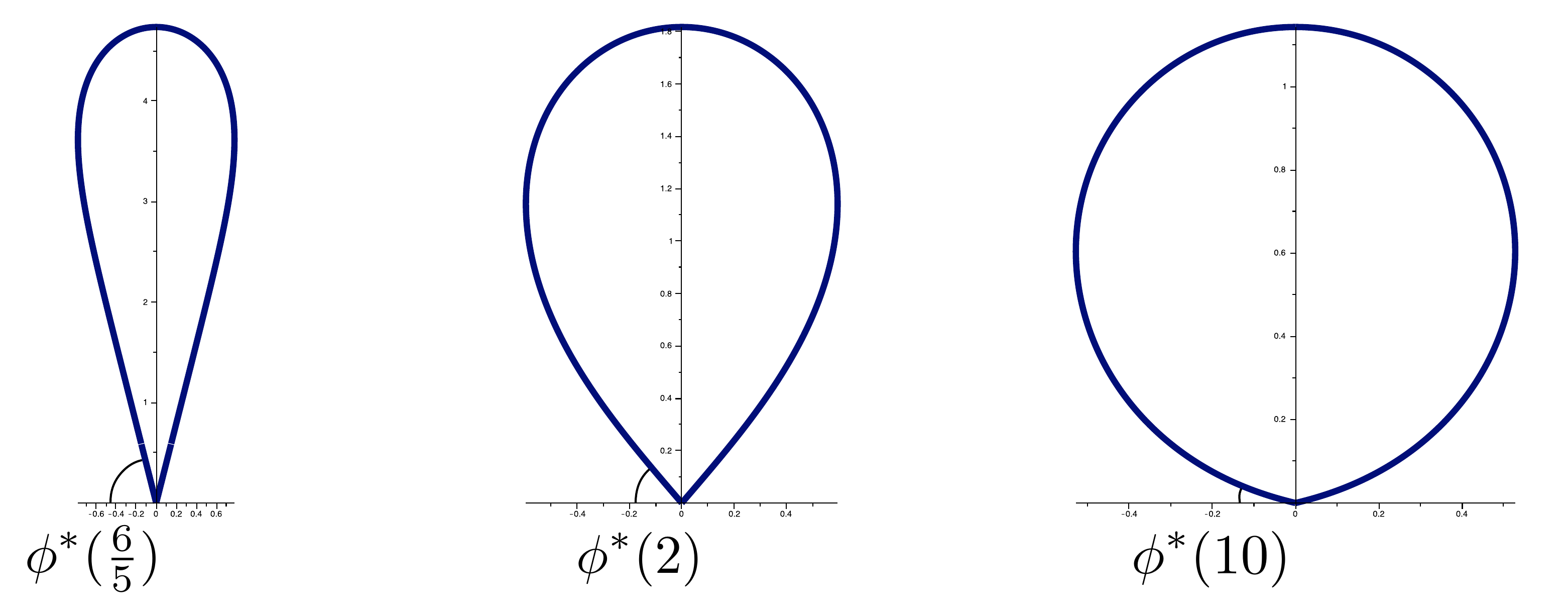}
    \caption{The angle $2\phi^*(p)$, and three examples of the half-fold figure-eight $p$-elasticae with $p=\frac{6}{5}$, $2$, $10$ (from left to right).}
    \label{fig:phi^*}
\end{figure}

This monotonicity is easy to infer from the numerical pictures but its analytic verification seems not straightforward.
Our proof involves an indirect argument.

In what follows we discuss the aforementioned applications in detail.
We also expect that our classification would be useful in other contexts, e.g.\ to detect the limit curves of gradient flows for the $p$-bending energy under the natural boundary condition.
For recent developments of such flows, see e.g.\ \cite{NP20, OPW20, BHV, BVH, OW23}.

\subsection{Li--Yau type inequality}
We first apply our result to deduce a Li--Yau type inequality involving the \emph{normalized $p$-bending energy} $\overline{\mathcal{B}}_p$, defined as the $p$-bending energy normalized by the length to be scale-invariant:
\begin{align}\label{eq:def-normalizedBp}
    \overline{\mathcal{B}}_p[\gamma]:=\mathcal{L}[\gamma]^{p-1}\mathcal{B}_p[\gamma].
\end{align}
In their celebrated study, Li and Yau obtained the sharp inequality that the Willmore energy of a closed surface is bounded below by $4\pi$ times multiplicity \cite{LiYau}.
A one-dimensional counterpart has also been studied by several authors \cite{PolPhD, vdM98, Glen13, MR23}, and recently (almost) optimized in \cite{Miura_LiYau}: If an immersed closed $W^{2,2}$-curve $\gamma$ in $\R^n$ ($n\geq2$) has a point of multiplicity $m\geq2$, then
\begin{align}\label{eq:Li-YauB2}
    \overline{\mathcal{B}}_2[\gamma] \geq \varpi^* m^2.
\end{align}
Here $\varpi^*>0$ denotes the normalized bending energy $\overline{\mathcal{B}}_2$ of a half-fold figure-eight elastica, and we say that a curve $\gamma$ has a point $P\in \R^n$ of \emph{multiplicity} $m$ if the preimage $\gamma^{-1}(P)$ contains at least $m$ distinct points.

A typical characteristic of the 1D Li--Yau type inequality (compared to 2D) is the presence of a new algebraic obstruction for optimality of the inequality.
In fact, it is proven in \cite{Miura_LiYau} that equality in \eqref{eq:Li-YauB2} is attained if and only if $n\geq3$ or $m$ is even, and also $\gamma$ is an $m$-leafed elastica.
In particular, for $n=2$ and any odd multiplicity $m\geq3$, the inequality is not optimal because of the irrationality of $\phi^*(2)/\pi$.

In this paper we first extend inequality \eqref{eq:Li-YauB2} to all $p\in(1,\infty)$ in the plane, and then reveal some new phenomena on its optimality arising from the generality of $p$.
Let $\mathbf{T}^1:=\R/\Z$ and $W^{2,p}_{\rm imm}(\mathbf{T}^1;\R^2)$ denote the set of immersed closed $W^{2,p}$-curves.
Let $\varpi^*_p>0$ denote the normalized $p$-bending energy $\overline{\mathcal{B}}_p$ of a half-fold figure-eight $p$-elastica, which will be given by formula \eqref{eq:def-varpi_p}; in particular, $\varpi^*_2=\varpi^*$.
Finally, we say that $\gamma$ is an $m$-leafed $p$-elastica if the curve $\gamma$ consists of $m$ half-fold figure-eight $p$-elasticae of same length 
%, cf.\ Figure~\ref{fig:35-leaf} 
(see Definition~\ref{def:m-leaf} for details).
Then, applying Theorem \ref{thm:uniqueness}, we establish the following

%%%%%%%%%%%%%%%%%%%%%%%%%%%%%%%%%%%%%%
\begin{theorem}[Li--Yau type inequality and rigidity]\label{thm:LiYau-Bp}
Let $p\in(1,\infty)$, and $m\geq2$ an integer.
If a closed curve $\gamma\in W^{2,p}_{\rm imm}(\mathbf{T}^1;\R^2)$ has a point of multiplicity $m$, then
\begin{align}\label{eq:LiYau-Bp}
    \overline{\mathcal{B}}_p[\gamma] \geq \varpi^*_p m^p.
\end{align}
Equality in \eqref{eq:LiYau-Bp} is attained if and only if $\gamma$ is a closed $m$-leafed $p$-elastica.
\end{theorem}

In addition, as in the case of $p=2$ \cite{Miura_LiYau}, it is easy to see that if the multiplicity $m$ is even, then our inequality is always optimal.

\begin{theorem}[Optimality for even multiplicity]\label{thm:LiYau-optimal-even}
Let $p\in(1,\infty)$.
If $m\geq2$ is an even integer, then there exists a closed curve $\gamma \in W^{2,p}_{\rm imm}(\mathbf{T}^1;\R^2)$ with a point of multiplicity $m$ such that
\begin{align}\notag%\label{eq:LiYau-Bp-eq}
    \overline{\mathcal{B}}_p[\gamma] = \varpi^*_p m^p.
\end{align}
In particular, we can take $\gamma$ to be an $\frac{m}{2}$-fold figure-eight $p$-elastica.
\end{theorem}

As for odd multiplicity, the optimality is much more delicate (and does not hold for $p=2$ as discussed).
However, for $p\neq2$ we can apply Theorem \ref{thm:phi*-decrease} to find many new phenomena that recover optimality, even for closed planar curves with odd multiplicity.
For example, there is a dense set $S\subset(1,\infty)$ such that each exponent $p\in S$ retrieves the optimality for all but a finite number of odd multiplicities (Theorem~\ref{thm:LiYau-optimal-odd}).
In particular, as a very peculiar example, we discover a unique exponent for which inequality \eqref{eq:LiYau-Bp} becomes fully optimal.

\begin{theorem}[Unique exponent for full optimality]\label{thm:LiYau-full-optimal}
There exists a unique exponent $p\in(1,\infty)$ with the following property: For every integer $m\geq2$ there exists a closed curve $\gamma\in W^{2,p}_{\rm imm}(\mathbf{T}^1;\R^2)$ with a point of multiplicity $m$ such that 
\begin{align*}
    \overline{\mathcal{B}}_{p}[\gamma] = \varpi^*_p m^{p}.
\end{align*}
The unique exponent is given by $p=p_3$ defined as
\begin{align}\label{eq:def-p3}
p_3:=(\phi^*)^{-1}(\tfrac{\pi}{3})\simeq 1.5728.
\end{align}
\end{theorem}

Moreover, for an odd $m\geq3$, if we define $p_m:=(\phi^*)^{-1}(\tfrac{\pi}{m})$,
then the corresponding closed planar $m$-leafed $p_m$-elastica is unique (up to invariances) and has $m$-fold rotational symmetry, cf.\ Figure \ref{fig:35-leaf}.
These exponents can be used to ensure a ``thresholding'' phenomenon; for any given odd integer $2\ell+1\geq3$ there is an exponent $p=p_{2\ell+1}$ for which optimality holds if and only if the multiplicity $m$ is none of $3,5,\dots,2\ell-1$ (see Theorem~\ref{thm:thresholding}).

%%%%%%%%%%%%%%%%%%%%%%%%%%%%%% 
\begin{center}
    \begin{figure}[htbp]
      \includegraphics[scale=0.2]{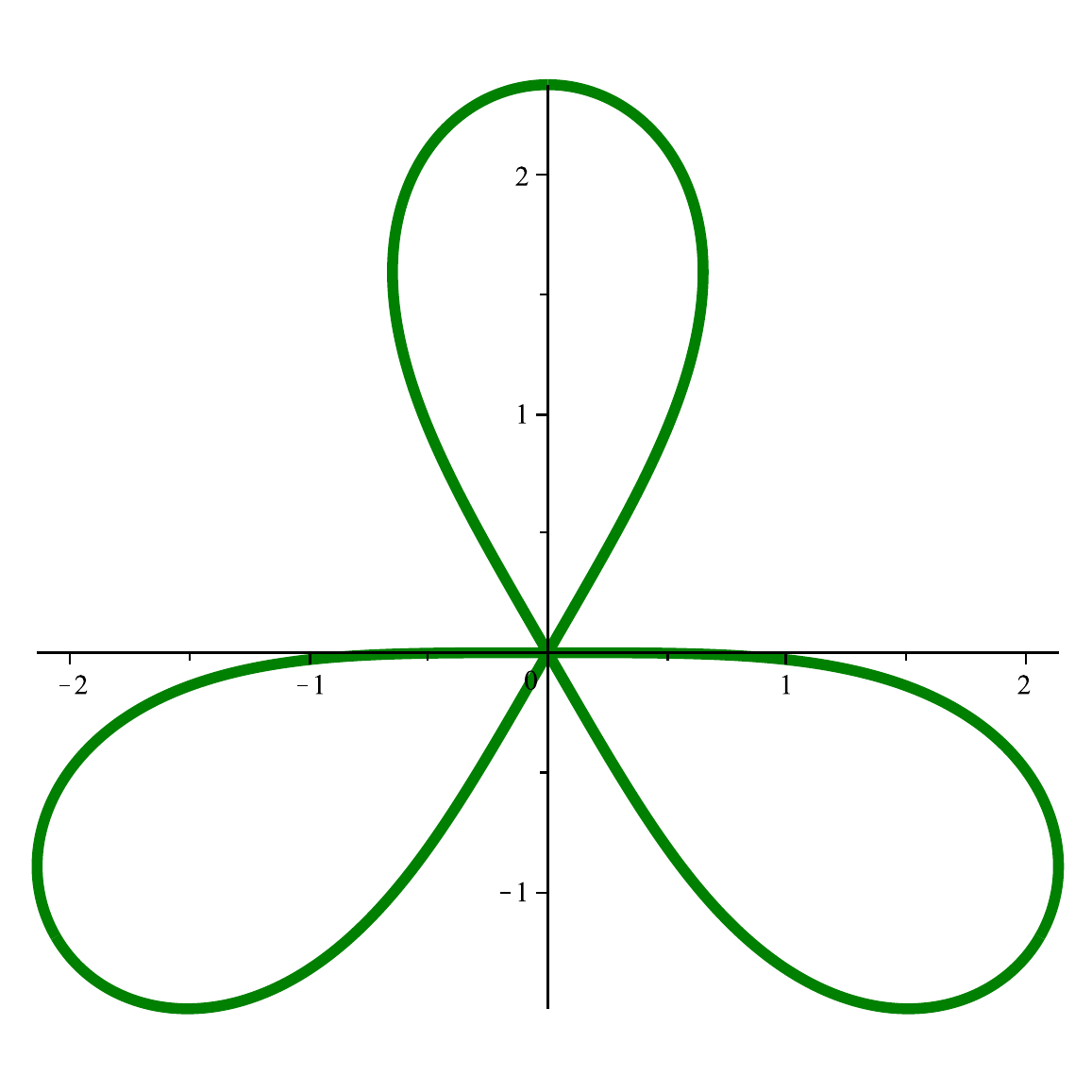}
  \hspace{30pt}
      \includegraphics[scale=0.2]{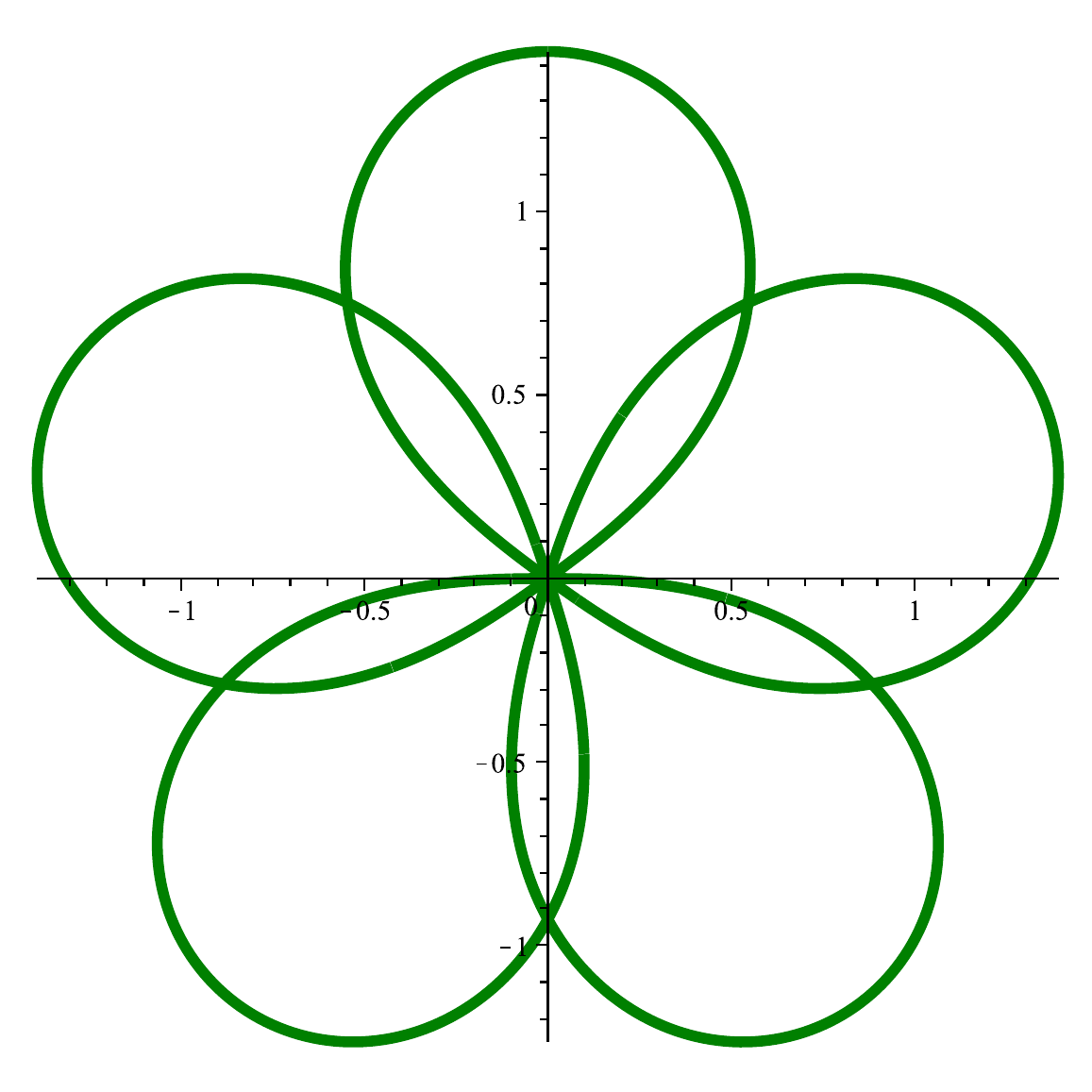}
  \caption{The $3$-leafed $p_3$-elastica and $5$-leafed $p_5$-elastica.}
  \label{fig:35-leaf}
  \end{figure}
\end{center}
%%%%%%%%%%%%%%%%%%%%%%%%%%%%%% 

\subsection{Minimal $p$-elastic networks}
Another application is about existence of planar networks composed by three curves minimizing the normalized $p$-bending energy.
Such a network structure is of recent interest particularly in geometric variational problems or flows as a model case with possible structural changes.
See \cite{BGN12, DP17, DLP19, GMP19, DNP20, GMP20, NP20, Miura_LiYau} for some instances of static and dynamical studies of elastic networks.

In this paper, for given $\alpha\in(0,\pi)$, a triplet of immersed curves $(\gamma_1, \gamma_2, \gamma_3) \in W^{2,p}_{\rm imm}(0,1; \R^2)^3$ is called \emph{$\Theta$-network with angles $(\alpha, \alpha, 2\pi-2\alpha)$} if $\gamma_1(0)=\gamma_2(0)=\gamma_3(0)$ and $\gamma_1(1)=\gamma_2(1)=\gamma_3(1)$, and if in addition the curves satisfy the angle condition
\begin{align*}
\left\langle \frac{\gamma_j'(0)}{|\gamma_j'(0)|}, \frac{\gamma_{j+1}'(0)}{|\gamma_{j+1}'(0)|} \right\rangle = \left\langle \frac{\gamma_j'(1)}{|\gamma_j'(1)|}, \frac{\gamma_{j+1}'(1)}{|\gamma_{j+1}'(1)|} \right\rangle 
=
\begin{cases}
\cos{\alpha} \quad &\text{if} \ \ j=1,2,  \\
\cos{(2\pi-2\alpha)} &\text{if} \ \ j=3, 
\end{cases}
\end{align*}
where we interpret $\gamma_4:=\gamma_1$.
Let $\Theta(p,\alpha)$ denote the class of all $\Theta$-networks with angles $(\alpha, \alpha, 2\pi-2\alpha)$, 
and for $\Gamma=(\gamma_1, \gamma_2, \gamma_3)\in \Theta(p,\alpha)$ we define
\[
\overline{\mathcal{B}}_p[\Gamma]:= \mathcal{L}[\Gamma]^{p-1}\mathcal{B}_p[\Gamma], 
\] 
where $\mathcal{L}[\Gamma]:=\sum_{i=1}^3 \mathcal{L}[\gamma_i]$ and $\mathcal{B}_p[\Gamma]:=\sum_{i=1}^3 \mathcal{B}_p[\gamma_i]$.
Applying Theorems~\ref{thm:uniqueness} and \ref{thm:phi*-decrease}, we can obtain the existence of minimal $\Theta$-networks in a certain range of the angle $\alpha$: 
%%%%%%%%%%%%%%%%%%%%%%%%%%%%%%%%%%%%%%
\begin{theorem}[Existence of minimal $p$-elastic $\Theta$-networks]\label{thm:inf-network-alpha}
Let $p\in(1,\infty)$ and $\alpha\in(0,\pi)$. Let $\phi^*(p)\in(0,\pi/2)$ be the angle defined in \eqref{eq:phi*}.
Suppose that
\begin{align}\label{eq:angle-network}
0 < \alpha < \pi -\phi^*(p).
\end{align}
Then there exists a network $\bar{\Gamma} \in \Theta(p,\alpha)$ such that 
\begin{align}\label{eq:inf-network}
\overline{\mathcal{B}}_p[\bar{\Gamma}] = \inf_{\Gamma\in\Theta(p,\alpha) }\overline{\mathcal{B}}_p[\Gamma].
\end{align}
\end{theorem}

In particular, focusing on the special ``homogeneous'' angle condition $\alpha={2\pi}/{3}$, we may rephrase Theorem \ref{thm:inf-network-alpha} in terms of $p_3$ defined by \eqref{eq:def-p3}.
%%%%%%%%%%%%%%%%%%%%%%%%%%%%%%%%%%%%%%
\begin{corollary}\label{cor:inf-network-sym}
If $p>p_3$,
then there exists a network $\bar{\Gamma} \in \Theta(p,\frac{2\pi}{3})$ such that 
$$
\overline{\mathcal{B}}_p[\bar{\Gamma}] = \inf_{\Gamma\in\Theta(p,\frac{2\pi}{3}) }\overline{\mathcal{B}}_p[\Gamma].
$$
\end{corollary}
%%%%%%%%%%%%%%%%%%%%%%%%%%%%%%%%%%%%%%

Corollary~\ref{cor:inf-network-sym} directly extends Dall'Acqua--Novaga--Pluda's existence result for $p=2$ \cite{DNP20} (see also \cite{Miura_LiYau}) to $p>p_3$.
Recall that numerically $p_3\simeq 1.5728$.
Analytically we can ensure at least $p_3<2$ thanks to Theorem~\ref{thm:phi*-decrease} with the fact that $\phi^*(2) <{\pi}/{3}=\phi^*(p_3)$ (cf.\ \cite[Lemma 3.11]{Miura_LiYau}).
%It is not clear whether the assumption $p>p_3$ is optimal, but the same existence result seems not to hold for $p$ close to $1$.

In general, the existence of minimal elastic networks does not follow from a standard direct method due to the lack of compactness, namely the possibility that a component-curve degenerates into a point (see Section~\ref{sect:network} for details).
Our proof of Theorem \ref{thm:inf-network-alpha} follows the general strategy of \cite{Miura_LiYau} strongly inspired by \cite{DNP20}.
Although the proof in \cite{DNP20} contains a numerical part, the one in \cite{Miura_LiYau} gives a different analytic proof.
Our proof here is also fully analytic.
Along the way we establish and use new monotonicity results involving $p$-elliptic integrals.

We close the introduction by mentioning a few of many open problems. 
Concerning uniqueness of global minimizers, here we addressed the pinned boundary condition (Theorem~\ref{thm:uniqueness}), but as for the clamped boundary condition, only a few results are available even for $p=2$ (see e.g.\ \cite{Miura20} for the straightened case and \cite{MMR23} for the cuspidal case) and in particular it is widely open for $p\neq2$ (except for the obvious closed case).
Concerning the Li--Yau type inequality, the shape of a minimizer of $\overline{\mathcal{B}}_p$ among $W^{2,p}$-immersed closed planar curves with an odd multiplicity $m\geq3$ is totally open unless equality in \eqref{eq:LiYau-Bp} is attained; in particular, if $p=2$, it is open for all odd $m\geq3$.
Finally, concerning minimal $p$-elastic networks, it is not clear whether the assumption $p>p_3$ in Corollary~\ref{cor:inf-network-sym} (or \eqref{eq:angle-network} in Theorem~\ref{thm:inf-network-alpha}) is optimal.
We expect that the same existence result would not hold at least for $p$ sufficiently close to $1$, but even this is remained open.

\subsection{Organization of the paper}
This paper is organized as follows: 
In Section~\ref{sect:preliminary} we prepare notation and known results.
In Section~\ref{sect:classification} we prove Theorem~\ref{thm:classify-pinned}. 
In Section~\ref{sect:min-unique} we complete the proof of Theorems~\ref{thm:uniqueness} and \ref{thm:phi*-decrease}.
Sections~\ref{sect:Li-Yau} and \ref{sect:network} are about applications to Li--Yau type inequalities and networks, respectively.

\subsection*{Acknowledgments}
The first author is supported by JSPS KAKENHI Grant Numbers 18H03670, 20K14341, and 21H00990, and by Grant for Basic Science Research Projects from The Sumitomo Foundation.
The second author is supported by JSPS KAKENHI Grant Number 22K20339.

\section{Preliminary}\label{sect:preliminary}

In this section we first recall the definitions and fundamental properties of $p$-elliptic integrals and functions introduced in \cite{MYarXiv2203}, which we use throughout this paper.
Then we rigorously define pinned $p$-elasticae, and recall some known facts for $p$-elasticae.

Hereafter, we always let $p\in(1,\infty)$ denote an arbitrary exponent unless an additional condition is specified.

\subsection{$p$-Elliptic integrals}\label{sect:p-Jacobi}

%%%%%%%%%%%%%%%%%%%%%%%%%%%%%%%%%%%%%%
\begin{definition}[$p$-Elliptic integrals of the first kind] \label{def:K_p}
The incomplete $p$-elliptic integrals of the first kind $\mathrm{F}_{1,p}(x,q)$ and $\mathrm{F}_{2,p}(x,q)$ of modulus $q\in[0,1)$, where $x\in\R$, are defined by 
\begin{align*}
&\mathrm{F}_{1,p}(x, q):=\int_0^{x} \frac{|\cos\theta|^{1-\frac{2}{p}} }{ \sqrt[]{1-q^2\sin^2\theta} }\,d\theta, \quad
 \mathrm{F}_{2,p}(x,q):=\int_0^{x} \frac{1}{ \sqrt[p]{1-q^2\sin^2\theta} }\,d\theta,
\end{align*}
and also the corresponding complete $p$-elliptic integrals $\K_{1,p}(q)$ and $\K_{2,p}(q)$ by 
\begin{align*}
    \K_{1,p}(q):=\mathrm{F}_{1,p}(\pi/2, q), \quad \K_{2,p}(q):=\mathrm{F}_{2,p}(\pi/2, q).
\end{align*}
In addition, for $q=1$,
\[
\mathrm{F}_{1,p}(x,1)=\mathrm{F}_{2,p}(x,1):=\displaystyle \int_0^{x} \frac{d\theta}{(\cos \theta)^{\frac{2}{p} } }, \quad \text{where}\ 
\begin{cases}
x\in(-\frac{\pi}{2},\frac{\pi}{2}) \quad &\text{if} \ \ 1< p \leq  2,  \\
x\in\R &\text{if} \ \ p>2,
\end{cases}
\]
and
\begin{align} \label{eq:K_p}
\K_{1,p}(1)=\K_{2,p}(1)=\K_{p} (1) := 
\begin{cases}
\infty \quad &\text{if} \ \ 1< p \leq  2,  \\
\displaystyle \int_0^{\frac{\pi}{2}} \frac{d\theta}{(\cos \theta)^{\frac{2}{p} } } < \infty &\text{if} \ \ p>2.
%=\int_0^{\frac{\pi}{2}} \frac{d\theta}{(\sin \theta)^{\frac{2}{p} } }
\end{cases}
\end{align}
\end{definition}
%%%%%%%%%%%%%%%%%%%%%%%%%%%%%%%%%%%%%%

%%%%%%%%%%%%%%%%%%%%%%%%%%%%%%%%%%%%%%
\begin{definition}[$p$-Elliptic integrals of the second kind]\label{def:E_p}
The incomplete $p$-elliptic integrals of the second kind $\E_{1,p}(x,q)$ and $\E_{2,p}(x,q)$ 
of modulus $q\in[0,1]$, where $x\in\R$, are defined by 
\[
\E_{1,p}(x,q):=\int_0^{x} \sqrt{1-q^2 \sin^2 \theta}\, |\cos \theta|^{1-\frac{2}{p}}\,d\theta, 
\quad
\E_{2,p}(x,q):=\int_0^{x} \sqrt[p]{1-q^2 \sin^2 \theta} \,d\theta,
\]
and also the corresponding complete $p$-elliptic integrals $\E_{1,p}(q)$ and $\E_{2,p}(q)$ by 
\begin{align*}
    \E_{1,p}(q):=\E_{1,p}(\pi/2, q), \quad \E_{2,p}(q):=\E_{2,p}(\pi/2, q).
\end{align*}
\end{definition}
%%%%%%%%%%%%%%%%%%%%%%%%%%%%%%%%%%%%%%

\begin{remark}
Since $1-\frac{2}{p} >-1$, both $\Fcn(x,q)$ and $\Ecn(x,q)$ are well defined for each $x\in\R$ and $q\in[0,1)$.
By definition and periodicity of the integrand we deduce the following quasiperiodicity:
\begin{align}\label{eq:period_Ecn}
\begin{split}
\E_{i,p}(x+ n\pi , q) &=  \E_{i,p}(x , q) + 2n\E_{i,p}(q), \\
\F_{i,p}(x+ n\pi , q) &=  \K_{i,p}(x , q) + 2n\K_{i,p}(q),
\end{split}
\end{align}
for any $x\in \R$, $q\in[0,1)$, $n \in \Z$, and $i=1,2$.
\end{remark}

\begin{remark}
In particular we also have the following alternative representations:
\begin{align}\label{eq:betsu-K1p}
\K_{1,p}(q)=\int_0^{\frac{\pi}{2}} \frac{|\cos\theta|^{1-\frac{2}{p}} }{ \sqrt[]{1-q^2\sin^2\theta} }\,d\theta
=\int_0^1\frac{1}{\sqrt{1-q^2z^2}}(1-z^2)^{-\frac{1}{p}}\,dz,
\end{align}
\begin{align}\label{eq:betsu-E1p}
    \E_{1,p}(q)=\int_0^{\frac{\pi}{2}}\sqrt{1-q^2 \sin^2 \theta}\, |\cos \theta|^{1-\frac{2}{p}}\,d\theta
=\int_0^1\sqrt{1-q^2z^2}(1-z^2)^{-\frac{1}{p}}\,dz.
\end{align}
\end{remark}

Now we discuss some basic properties.
As in the classical case $p=2$, the complete $p$-elliptic integrals $\E_{1,p}(q)$ and $\K_{1,p}(q)$ are not only monotone with respect to modulus $q\in[0,1)$ but also related by explicit derivative formulae:

%%%%%%%%%%%%%%%%%%%%%%%%%%%%%%%%%%%%%%
\begin{proposition}\label{prop:EK-monotone}
The function $\K_{1,p}(q)$ is strictly increasing with respect to $q\in[0,1)$ (including $q=1$ if $p>2$), and the function $\E_{1,p}(q)$ is strictly decreasing with respect to $q\in[0,1)$. 
Moreover, for $q\in (0,1)$,
\begin{align} \label{eq:diff-ellipint}
\begin{split}
    & \frac{d}{dq}\mathrm{K}_{1,p}(q) =
    \frac{(2-\frac{2}{p})\mathrm{E}_{1,p}(q)-(2-\frac{2}{p}-q^2)\mathrm{K}_{1,p}(q)}{q(1-q^2)}, \\
    & \frac{d}{dq}\mathrm{E}_{1,p}(q) = \frac{\mathrm{E}_{1,p}(q)-\mathrm{K}_{1,p}(q)}{q}.
\end{split}
\end{align}
\end{proposition}

In fact, the above derivative formulae are special cases of those for more general elliptic integrals obtained by Takeuchi \cite{Takeuchi16}: For $a,b,c\in(1,\infty)$ and $q\in[0,1)$, let
\[
\mathsf{K}_{a,b,c}(q):= \int_0^1 \frac{1}{ \sqrt[a]{1-z^b} \sqrt[c]{1-q^bz^b} }\,dz, \quad 
\mathsf{E}_{a,b,c}(q):= \int_0^1 \frac{ \sqrt[c]{1-q^bz^b} }{ \sqrt[a]{1-z^b} }\,dz.
\]
By \cite[Proposition 2]{Takeuchi16}, letting $c':={c}/{(c-1)}$ and $\xi:=1+ {b}/{c} - {b}/{a}$, we have
\begin{align}\label{eq:T16prop2}
\begin{split}
  &\frac{d}{dq}\mathsf{K}_{a,b,c'}(q) = \frac{\xi\mathsf{E}_{a,b,c}(q)-(\xi-q^b)\mathsf{K}_{a,b,c'}(q)}{q(1-q^b)}, \\
  &\frac{d}{dq}\mathsf{E}_{a,b,c}(q) = \frac{b(\mathsf{E}_{a,b,c}(q)-\mathsf{K}_{a,b,c'}(q))}{cq}.
\end{split}
\end{align}

%%%%%%%%%%%%%%%%%%%%%%%%%%%%%%%%%%%%%%
\begin{proof}[Proof of Proposition \ref{prop:EK-monotone}]
Monotonicity follows by differentiating in $q\in(0,1)$:
\begin{align} \label{eq:EK-bibun-q}
\begin{split}
    \frac{d}{dq}\Kcn(q) &= \int_0^{\frac{\pi}{2}}\frac{q\sin^2\theta|\cos\theta|^{1-\frac{2}{p}} }{ (1-q^2\sin^2\theta)^{\frac{3}{2}} }\,d\theta>0, \\
    \frac{d}{dq}\E_{1,p}(q)&= -\int_0^{\frac{\pi}{2}}\frac{q\sin^2\theta |\cos\theta|^{1-\frac{2}{p}}}{ \sqrt{1-q^2\sin^2\theta } }\,d\theta <0.
\end{split}
\end{align}
In addition, since $\E_{1,p}(q)=\mathsf{E}_{p,2,2}(q)$ and $\K_{1,p}(q)=\mathsf{K}_{p,2,2}(q)$ hold (cf.\ \eqref{eq:betsu-K1p} and \eqref{eq:betsu-E1p}), by using \eqref{eq:T16prop2}, we obtain \eqref{eq:diff-ellipint}.
\end{proof}

We will also consider the ratio of the first and second complete $p$-elliptic integrals.

%%%%%%%%%%%%%%%%%%%%%%%%%%%%%%%%%%%%%%
\begin{lemma}[{\cite[Lemma 2]{nabe14}}]\label{lem:nabe-lem2}
Let $Q_p:[0,1)\to \R$ be defined by
\begin{align} \label{eq:1116-4}
Q_p(q):= 2 \frac{\E_{1,p}(q)}{\K_{1,p}(q)}-1 , \quad q\in[0,1).
\end{align}
Then $Q_p$ is strictly decreasing on $[0,1)$. %$(0,1)$.
Moreover, $Q_p$ satisfies $Q_p(0)=1$ and
\begin{align}\label{eq:1117-1}
%Q(1)=
\lim_{q\uparrow1}Q_p(q)=
\begin{cases}
-1 \quad & \text{if} \ \ 1<p \leq 2, \\
-\dfrac{1}{p-1} & \text{if} \ \ p>2.
\end{cases}
\end{align}
\end{lemma}
%%%%%%%%%%%%%%%%%%%%%%%%%%%%%%%%%%%%%%

% It will be shown later that for fixed $q\in(0,1)$ $Q_p(q)$ is also monotone with respect to $p\in(1,\infty)$ (see Proposition~\ref{prop:Qp-monotone}).

By this monotonicity of $Q_p$ we can define the unique modulus corresponding to ``figure-eight'', which we will use frequently:

%%%%%%%%%%%%%%%%%%%%%%%%%%%%%%%%%%%%%%
\begin{definition}[Modulus of figure-eight]\label{def:q*}
Let $q^*=q^*(p) \in (0,1)$ denote a unique solution to the equation $Q_p(q)=0$.
\end{definition}
%%%%%%%%%%%%%%%%%%%%%%%%%%%%%%%%%%%%%%

\subsection{$p$-Elliptic functions}\label{sect:p-Jacobi2}

Next we recall the $p$-elliptic functions. 

%%%%%%%%%%%%%%%%%%%%%%%%%%%%%%%%%%%%%%
\begin{definition}[$p$-Elliptic functions] \label{def:cndn}
Let $ q \in[0, 1]$.
The amplitude functions $\amcn(x,q)$ and $\amdn(x,q)$ with modulus $q$ are defined by the inverse functions of $\mathrm{F}_{1,p}(x, q)$ and $\mathrm{F}_{2,p}(x, q)$, respectively, i.e., for $x\in\R$, 
\begin{align} \nonumber
x= \int_0^{\amcn(x,q)} \frac{|\cos\theta|^{1-\frac{2}{p}} }{ \sqrt[]{1-q^2\sin^2\theta} }\,d\theta, \quad x= \int_0^{\amdn(x,q)} \frac{1}{ \sqrt[p]{1-q^2\sin^2\theta} }\,d\theta.
\end{align}
The \emph{$p$-elliptic sine} $\sn_p(x,q)$ and \emph{$p$-elliptic cosine} $\cn_p(x,q)$ with modulus $q$ are defined by, for $x\in \R$, 
\begin{align}\notag%\label{eq:sn_p}
\sn_p (x, q)&:= \sin \amcn(x,q), 
\\
%\end{align}
%\begin{align}\label{eq:cn_p}
\cn_p (x, q)&:= |\cos \amcn(x,q)|^{\frac{2}{p}-1} \cos \amcn(x,q). \label{eq:cn_p}
\end{align}
The \emph{$p$-delta amplitude} $\dn_p(x,q)$ with modulus $q$ is defined by 
\begin{align}\notag%\label{eq:dn_p}
\dn_p (x, q):= \sqrt[p]{1-q^2 \sin^2 \big(\amdn(x, q)\big)}, \quad x\in \R.
\end{align}
\end{definition}

%%%%%%%%%%%%%%%%%%%%%%%%%%%%%%%%%%%%%%

In addition, a typical part of $\cn_p(\cdot, 1)$ and $\dn_p(\cdot, 1)$ is interpreted as the $p$-hyperbolic secant function.

%%%%%%%%%%%%%%%%%%%%%%%%%%%%%%%%%%%%%%
\begin{definition}[$p$-Hyperbolic functions] \label{def:sech}
The \emph{$p$-hyperbolic secant} $\sech_p x$ is defined by
\begin{align}\label{eq:sech_p} 
 \sech_p x:= 
\begin{cases}
 \cn_p(x,1)=\dn_p(x,1), \quad & x\in (-\K_p(1), \K_p(1)), \\
 0, &x\in \R \setminus (-\K_p(1), \K_p(1)).
\end{cases}
\end{align}
If $1<p \leq 2$, then we regard $(-\K_p(1), \K_p(1))$ as $\R$.
Moreover, the \emph{$p$-hyperbolic tangent} $\tanh_p x$ is defined by
\[\tanh_p x:=\int_0^x (\sech_p t)^p dt, \quad x\in \R. \]
\end{definition}
%%%%%%%%%%%%%%%%%%%%%%%%%%%%%%%%%%%%%%

In general our $p$-elliptic functions satisfy fundamental properties such as periodicity and monotonicity similar to the classical Jacobian elliptic functions ($p=2$).
In particular we recall some of those about
$\cn_p$ and $\sech_p$ for later use.

%%%%%%%%%%%%%%%%%%%%%%%%%%%%%%%%%%%%%%
\begin{proposition}[{\cite[Proposition 3.10]{MYarXiv2203}}]\label{prop:property-cndn}
Let $\cn_p$ and $\sech_p$ be given by \eqref{eq:cn_p} and \eqref{eq:sech_p}, respectively. 
\begin{itemize}
%\item [(i)] For $q\in[0,1)$, $\sn_p(\cdot, q)$ is an odd $2\K_{1,p}(q)$-antiperiodic function on $\R$ and, in $[-\K_{1,p}(q), \K_{1,p}(q)]$, strictly increasing from $1$ to $-1$. 
\item [(i)] For $q\in[0,1)$, $\cn_p(\cdot, q)$ is an even $2\K_{1,p}(q)$-antiperiodic function on $\R$ and, in $[0, 2\K_{1,p}(q)]$, strictly decreasing from $1$ to $-1$. 
%\item [(iii)] For $q\in[0,1)$, $\dn_p(\cdot, q)$ is an even, positive, $2\K_{2,p}(q)$-periodic function on $\R$ and, in $[0, \K_{2,p}(q)]$, strictly decreasing from $1$ to $\sqrt[p]{1-q^2 }$. 
\item [(ii)] If $1< p \leq 2$, then $\sech_p $ is an even positive function on $\R$, and strictly decreasing in $[0, \infty)$.
Moreover, $\sech_p 0=1$ and $\sech_p x \to 0$ as $x \uparrow \infty$.
\item [(ii')] %If $p> 2$, $\sech_p $ is an even function, in $[0, \K_{p} (1)]$, is strictly decreasing from $1$ to $0$.
If $p> 2$, then $\sech_p $ is an even nonnegative function on $\R$, and strictly decreasing in $[0, \K_{p} (1))$. 
Moreover, $\sech_p 0=1$ and 
$
\sech_p x \to 0
$
as $x \uparrow \K_{p} (1)$.
In particular, $\sech_p $ is continuous on $\R$.
\end{itemize}
\end{proposition}
%%%%%%%%%%%%%%%%%%%%%%%%%%%%%%%%%%%%%%

\subsection{The Euler--Lagrange equation and pinned $p$-elastica}

Now we rigorously define pinned (planar) $p$-elasticae, and recall their known properties.

%%%%%%%%%%%%%%%%%%%%%%%%%%%%%%%%%%%%%%
\begin{definition}[Pinned $p$-elastica] \label{critical_point}
Let $P_0,P_1\in\R^2$ and $L>0$ such that $|P_0-P_1|<L$.
Let $\mathcal{A}_{P_0, P_1, L}$ be the admissible space defined in \eqref{eq:admissiblespace}.
\begin{itemize}
    \item For $\gamma\in \mathcal{A}_{P_0, P_1, L}$, we call a one-parameter family $\varepsilon\mapsto \gamma_\varepsilon \in \mathcal{A}_{P_0, P_1, L}$ \emph{admissible perturbation} of $\gamma$ in $\mathcal{A}_{P_0, P_1, L}$ if $\gamma_0=\gamma$ and if the derivative $\frac{d}{d\varepsilon}\gamma_\varepsilon\big|_{\varepsilon=0}$ exists.
    \item We say that $\gamma\in \mathcal{A}_{P_0, P_1, L}$ is a \emph{critical point} of $\mathcal{B}_p$ in the admissible space $\mathcal{A}_{P_0, P_1, L}$ if
    for any admissible perturbation $(\varepsilon\mapsto\gamma_\varepsilon)$ of $\gamma$ in $\mathcal{A}_{P_0, P_1, L}$ the first variation of $\mathcal{B}_p$ vanishes:
\begin{align}\notag%\label{eq:1st-derivative}
\frac{d}{d\varepsilon}\mathcal{B}_p[\gamma_\varepsilon]\Big|_{\varepsilon=0}=0.
\end{align}
\end{itemize}
We also call such a critical point \emph{pinned $p$-elastica} in general.
\end{definition}
%%%%%%%%%%%%%%%%%%%%%%%%%%%%%%%%%%%%%%

Let $\gamma$ be a pinned $p$-elastica.
Then, by the Lagrange multiplier method (cf.\  Proposition~\ref{prop:Lmultiplier}),
there is a multiplier $\lambda\in \R$ such that 
\begin{align}\label{eq:1st-derivative-lam}
\big\langle D\mathcal{B}_p[\gamma] + \lambda D\mathcal{L}[\gamma], h \big\rangle=0
\end{align}
for $h\in W^{2,p}(0,1;\R^2)$ with $h(0)=h(1)=(0,0)$, 
where $D\mathcal{B}_p[\gamma]$ and $ D\mathcal{L}[\gamma]$ are the Fr\'echet derivatives of $\mathcal{B}_p$ and $\mathcal{L}$ at $\gamma$, respectively.
Using the arclength parameterization $\tilde{\gamma}$ of $\gamma$, we can rewrite \eqref{eq:1st-derivative-lam} as 
\begin{align}\label{eq:0525-1}
\int_0^L\Big( (1-2p) |\tilde{\vc}''|^p (\tilde{\vc}', \eta') +p|\tilde{\vc}''|^{p-2}(\tilde{\vc}'', \eta'') +\lm (\tilde{\vc}', \eta')\Big) ds = 0 
\end{align}
for $\eta \in W^{2,p}(0,L;\R^2)$ with $\eta(0)=\eta(L)=(0,0)$
(cf.\ Proposition~\ref{prop:arclengthLM-pin}).
Moreover, we can translate \eqref{eq:0525-1} in terms of the signed
curvature; indeed, according to \cite[Proposition 2.1]{MYarXiv2203}, if $ \gamma \in W^{2,p}_{\rm imm}(0,1;\R^2)$ satisfies \eqref{eq:0525-1}, then its signed curvature $k:[0,L]\to\R$ (parameterized by the arclength) belongs to $L^{\infty}(0,L)$, and further $k$ satisfies 
\begin{align}\label{eq:EL}\tag{EL}
\int_0^L \Big( p |k|^{p-2}k \vp'' +(p-1) |k|^pk\vp -\lambda k \vp \Big) ds = 0
% \quad \text{for all} \quad \vp \in C^{\infty}_{\rm c}(0,L).
\end{align}
for all $\varphi \in W^{2,p}(0,L)$ with $\varphi(0)=\varphi(L)=0$  
(although the class of test functions is different from \cite{MYarXiv2203}, the derivation is parallel).

In particular, we find that any pinned $p$-elastica in the sense of Definition \ref{critical_point} is always a \emph{$p$-elastica} in the sense of \cite[Definition 1.1]{MYarXiv2203} by restricting test functions to $\eta\in  C^{\infty}_{\rm c}(0,L;\R^2)$ in \eqref{eq:0525-1}.
The complete classification of $p$-elasticae is already known. 
In order to state the classification, we prepare the notation on a concatenation of curves.
For $\gamma_j:[a_j,b_j]\to\R^2$ with $L_j:=b_j-a_j\geq0$, we define $\gamma_1\oplus\gamma_2:[0,L_1+L_2]\to\R^2$ by
\begin{align*}
  (\gamma_1\oplus\gamma_2)(s) :=
  \begin{cases}
    \gamma_1(s+a_1), \quad & s \in[0, L_1], \\
    \gamma_2(s+a_2-L_1) +\gamma_1(b_1)-\gamma_2(a_2),  & s \in[L_1,L_1+L_2],
  \end{cases}
\end{align*}
 and inductively define $\gamma_1\oplus\dots\oplus\gamma_N := (\gamma_1\oplus\dots\oplus\gamma_{N-1})\oplus\gamma_N$.
 We also write
\begin{align}\notag%\label{eq:def-oplusN}
     \bigoplus_{j=1}^N\gamma_j:=\gamma_1\oplus\dots\oplus\gamma_N.
\end{align}

We now recall the classification of $p$-elasticae:

%%%%%%%%%%%%%%%%%%%%%%%%%%%%%%%%%%%%%%
\begin{proposition}[{\cite[Theorems 1.2 and 1.3]{MYarXiv2203}}]\label{prop:MY2203-thm1.1}
  Let $\gamma\in W^{2,p}_{\rm imm}(0,1;\R^2)$ be a $p$-elastica. 
  Then up to similarity (i.e., translation, rotation, reflection, and dilation) and reparameterization, the curve $\gamma$ is represented by $\gamma(s)=\gamma_*(s+s_0)$ with some $s_0\in\R$, where $\gamma_*:\R\to\R^2$ is one of the following arclength parameterizations:
   \begin{itemize}
     \item \textup{(Case I --- Linear $p$-elastica)}
     $\gamma_\ell(s)=(s,0)$, where $k_\ell\equiv0$.
     \item \textup{(Case II --- Wavelike $p$-elastica)}
     For some $q \in (0,1)$,
     \begin{equation}\label{eq:EP2}
       \gamma_w(s,q) =
       \begin{pmatrix}
       2 \E_{1,p}(\am_{1,p}(s,q),q )-  s  \\
       -q\frac{p}{p-1}|\cn_p(s,q)|^{p-2}\cn_p(s,q)
       \end{pmatrix}.
     \end{equation}
     In this case, $\theta_w(s)=2\arcsin(q\sn_p(s,q))$ and $k_w(s) = 2q\cn_p(s,q).$
     \item \textup{(Case III --- Borderline $p$-elastica, $1< p \leq 2$)}
     \begin{equation}\notag%\label{eq:EP3-1}
       \gamma_b(s) =
       \begin{pmatrix}
       2 \tanh_p{s} - s  \\
       - \frac{p}{p-1}(\sech_p{s})^{p-1}
       \end{pmatrix}.
     \end{equation}
     In this case, $\theta_b(s)=2\am_{1,p}(s,1)=2\am_{2,p}(s,1)$ and $k_b(s) = 2\sech_p{s}.$
       \item \textup{(Case III' --- Flat-core $p$-elastica, $p>2$)}
    For some integer $N\geq1$, signs $\sigma_1,\dots,\sigma_N\in\{+,-\}$, and nonnegative numbers $L_1,\dots,L_N\geq0$,
    \begin{equation}\label{eq:EP3-2}
      \gamma_f = \bigoplus_{j=1}^N (\gamma_\ell^{L_j}\oplus \gamma_b^{\sigma_j}),
    \end{equation}
    where $\gamma_b^\pm:[-\K_p(1),\K_p(1)]\to\R^2$ and $\gamma_\ell^{L_j}:[0,L_j]\to\R^2$ are defined by
    \begin{align}\label{eq:pm-border}
      \gamma_b^{\pm}(s) =
      \begin{pmatrix}
      2 \tanh_p{s} - s  \\
      \mp \frac{p}{p-1}(\sech_p{s})^{p-1}
      \end{pmatrix},
      \quad
      \gamma_\ell^{L_j}(s) =
      \begin{pmatrix}
      -s  \\
      0
      \end{pmatrix}.
    \end{align}
    The curves $\gamma_b^{\pm}(s)$ have $\theta_b^\pm(s)=\pm2\am_{1,p}(s,1)=\pm2\am_{2,p}(s,1)$ and $k_b^\pm(s) = \pm2\sech_p{s}$ for $s\in [-\K_p(1),\K_p(1)]$.
    In particular, 
         \begin{align}\label{eq:flat-type}
                k_f(s)=\sum_{i=1}^N\sigma_i2\sech_p(s-s_i), \quad \text{where}\quad  s_i=(2i-1)\K_p(1) + \sum_{j=1}^i L_j.     
         \end{align}
     \item \textup{(Case IV --- Orbitlike $p$-elastica)}
     For some $q \in (0,1)$,
         \begin{equation}\notag%\label{eq:EP4}
       \gamma_o(s,q) = \frac{1}{q^2}
       \begin{pmatrix}
       2 \E_{2,\frac{p}{p-1}}(\am_{2,p}(s,q),q)  + (q^2-2)s \\
       - \frac{p}{p-1}\dn_p(s,q)^{p-1}
       \end{pmatrix}.
     \end{equation}
     In this case, $\theta_o(s)=2\am_{2,p}(s,q)$ and $k_o(s) = 2 \dn_p(s,q).$
       \item \textup{(Case V --- Circular $p$-elastica)}
     $\gamma_c(s)=(\cos{s},\sin{s})$, where $k_c\equiv1$.
   \end{itemize}
   Here $\theta_*$ denotes the tangential angle $\partial_s\gamma_*=(\cos\theta_*,\sin\theta_*)$, and $k_*$ the (counterclockwise) signed curvature $k_*=\partial_s\theta_*$.
\end{proposition}
%%%%%%%%%%%%%%%%%%%%%%%%%%%%%%%%%%%%%%

The optimal regularity of $p$-elasticae is known to depend on $p$.
Here we recall a weak general regularity (independent of $p$), which will be sufficient for our purpose.
%%%%%%%%%%%%%%%%%%%%%%%%%%%%%%%%%%%%%%
\begin{proposition}[{\cite[Theorem 1.7]{MYarXiv2203}}] \label{prop:Soft_regularity}
  Let $\gamma$ be a $p$-elastica.
  Then $\gamma$ has continuous (arclength parameterized) signed curvature. 
\end{proposition}
%%%%%%%%%%%%%%%%%%%%%%%%%%%%%%%%%%%%%%

\section{Classification of pinned $p$-elasticae} \label{sect:classification}

In this section we give the complete classification of pinned $p$-elasticae, thus proving Theorem~\ref{thm:classify-pinned}.

As a key fact for reducing the possible candidates, we first deduce that any pinned $p$-elastica satisfies an additional (natural) boundary condition that the curvature vanishes at the endpoints.

%%%%%%%%%%%%%%%%%%%%%%%%%%%%%%%%%%%%%%
\begin{lemma} \label{lem:0th-bc}
If $\gamma$ is a pinned $p$-elastica, then the arclength parameterized signed curvature $k \in C([0,L])$ of $\gamma$ satisfies
\[
k(0) = k(L)=0.
\]
\end{lemma} 
%%%%%%%%%%%%%%%%%%%%%%%%%%%%%%%%%%%%%%
\begin{proof}
Recall that $k\in C([0,L])$ by Proposition~\ref{prop:Soft_regularity} so that $k$ is defined pointwise.
From the fact that $k$ satisfies \eqref{eq:EL} we deduce that the function
\[ w(s) := |k(s)|^{p-2}k(s), \quad s\in [0,L] \] 
satisfies 
\begin{align} \label{eq:0510-1}
\int_0^L \! \Big(p  w \varphi'' + (p-1)|w|^{\frac{2}{p-1}}w \varphi - \lm |w|^{\frac{2-p}{p-1}}w \varphi \Big) \,ds = 0
\end{align}
for any $\varphi \in W^{2,p}(0,L) \cap W^{1,p}_0(0,L)$.
Moreover, it is shown in \cite[Lemma 4.3]{MYarXiv2203} that $w$ also satisfies \eqref{eq:0510-1} in the classical sense, i.e., 
\begin{align}  \label{eq:0610-1}
p w''(s) + (p-1)|w(s)|^{\frac{2}{p-1}}w(s) - \lambda |w(s)|^{\frac{2-p}{p-1}}w(s) =0 \quad \text{in} \ [0,L].
\end{align}
By choosing $\varphi\in W^{2,p}(0,L)\cap W^{1,p}_0(0,L)$ with $\varphi'(0)=-1$ and $\varphi'(L)=0$ in \eqref{eq:0510-1}, and integrating by parts, we obtain
\begin{align*}
0 &= \int_0^L \Big(p  w \vp'' + (p-1)|w|^{\frac{2}{p-1}}w \vp - \lm |w|^{\frac{2-p}{p-1}}w \vp \Big) \,ds \\
&= \big[ p w(s) \varphi'(s) \big]_{s=0}^{s=L} +\int_0^L \Big(pw'' + (p-1)|w|^{\frac{2}{p-1}}w - \lm |w|^{\frac{2-p}{p-1}}w  \Big)\vp \,ds 
= p w(0), 
\end{align*}
where we used $\varphi(0)=\varphi(L)=0$, $\varphi'(0)=-1$, $\varphi'(L)=0$, and \eqref{eq:0610-1}. %(by a density argument, or even \eqref{eq:0610-1} in the classical sense).
Therefore, $p w(0)= p |k(0)|^{p-2}k(0)=0$, so that $k(0) = 0$ holds.
In the same way we obtain $k(L) = 0$.
\end{proof}

By Proposition~\ref{prop:MY2203-thm1.1} and Lemma~\ref{lem:0th-bc}, with the obvious fact that linear $p$-elasticae are ruled out by $|P_0-P_1|<L$, if $\gamma$ is a pinned $p$-elastica, then $\gamma$ must be
\begin{align*}
\begin{cases}
\text{a wavelike $p$-elastica} \quad & \text{if} \ \ 1<p\leq 2, \\
\text{a wavelike $p$-elastica} \ \ \text{or} \ \ \text{a flat-core $p$-elastica}  &\text{if} \ \ p>2.
\end{cases}
\end{align*}
This fact drastically reduces the candidates of pinned $p$-elasticae.
Motivated by this reduction, we prepare terminology for the following special $p$-elasticae --- we will later show that those curves are indeed the only possibilities.

%%%%%%%%%%%%%%%%%%%%%%%%%%%%%%%%%%%%%%
\begin{definition}[Arc, loop, and flat-core]\label{def:arcloop}
Let $p\in(1,\infty)$, $r\in[0,1)$, and $n\in \N$.
\begin{itemize}
   \item A curve $\gamma$ is called \emph{$(p,r,n)$-arc} if, up to similarity and reparameterization, $\gamma$ is given by
    \[\gamma(s)=\gamma_w(s-\K_{1,p}(q),q) \quad s\in[0,2n\K_{1,p}(q)], \]
    where $\gamma_w$ is defined by \eqref{eq:EP2} and $q = q(r) \in(0,1)$ is a unique solution of
    \begin{align}\label{eq:wave-arc-r}
      2\frac{\E_{1,p}(q)}{\K_{1,p}(q)} -1 = r.  
    \end{align}
   \item A curve $\gamma$ is called \emph{$(p,r,n)$-loop} if $r<\frac{1}{p-1}$ and, up to similarity and reparameterization, $\gamma$ is given by
    \[\gamma(s)=\gamma_w(s-\K_{1,p}(q),q) \quad s\in[0,2n\K_{1,p}(q)], \]
    where $\gamma_w$ is defined by \eqref{eq:EP2} and $q = q(r) \in(0,1)$ is a unique solution of 
    \begin{align}\label{eq:wave-loop-r}
      2\frac{\E_{1,p}(q)}{\K_{1,p}(q)} -1 = -r.  
    \end{align}
   \item In particular, we call a $(p,0,n)$-arc, or equivalently a $(p,0,n)$-loop, \emph{$\frac{n}{2}$-fold figure-eight $p$-elastica}.
   A $\frac{1}{2}$-fold figure-eight $p$-elastica is also called \emph{half-fold figure-eight $p$-elastica}.
   \item A curve $\gamma$ is called \emph{$(p, r, n, \boldsymbol{\sigma}, \boldsymbol{L})$-flat-core} if $p>2$, $r\geq\frac{1}{p-1}$, and there are  $\boldsymbol{\sigma}=(\sigma_1,\dots,\sigma_n) \in\{+,-\}^n$ and
   $\boldsymbol{L}=(L_1,\dots,L_{n+1}) \in [0,\infty)^{n+1}$ such that, up to similarity and reparameterization, $\gamma$ is given by
    \[\gamma=\bigg( \bigoplus_{j=1}^n \big( \gamma_{\ell}^{L_j} \oplus \gamma_{b}^{\sigma_j} \big) \bigg) \oplus \gamma_{\ell}^{L_{n+1}}, \]
    where $\gamma_\ell^{L_j}$ and $\gamma_b^\pm$ are given by \eqref{eq:pm-border}, and in addition $p, r, n$ and $\boldsymbol{L}$ satisfy  
    \begin{align} \label{eq:sum-flatparts}
    \sum_{j=1}^{n+1}L_j = 2n\frac{r-\frac{1}{p-1}}{1-r} \K_{p}(1).
    \end{align}
\end{itemize}
\end{definition}
%%%%%%%%%%%%%%%%%%%%%%%%%%%%%%%%%%%%%%

\begin{remark}
The monotonicity in Lemma \ref{lem:nabe-lem2} implies uniqueness of solutions to
equations \eqref{eq:wave-arc-r} and \eqref{eq:wave-loop-r}, respectively.
By $Q_p(0)=1$ and \eqref{eq:1117-1}, equation \eqref{eq:wave-arc-r} always admits a solution, while
\eqref{eq:wave-loop-r} admits a solution if and only if $r< |Q_p(1)|$, or equivalently $r<\frac{1}{p-1}$.
In order to define the flat-core case, the condition $p>2$ is obviously required since flat-core elasticae only appear for $p>2$; the condition $r\geq\frac{1}{p-1}$ is also necessary and sufficient since this is equivalent to the nonnegativity of the right-hand side of
\eqref{eq:sum-flatparts}.
\end{remark}

\begin{remark}
\label{rem:n-fold_8}
For $n=2N$ with some $N\in\N$, an $\frac{n}{2}$-fold figure-eight $p$-elastica in the sense of Definition~\ref{def:arcloop} is same as an $N$-fold figure-eight $p$-elastica introduced in \cite[Definition 5.3]{MYarXiv2203}.
It is shown in \cite[Proposition 5.5]{MYarXiv2203} that an $N$-fold figure-eight $p$-elastica indeed defines a closed curve.
\end{remark}

%%%%%%%%%%%%%%%%%%%%%%%%%%%%%%%%%%%%%%

One can observe from the above definition that, loosely speaking, the given parameter $r:=\frac{|P_0-P_1|}{L}\in[0,1)$ characterizes the modulus $q\in[0,1)$ in Proposition~\ref{prop:MY2203-thm1.1}.

%%%%%%%%%%%%%%%%%%%%%%%%%%%%%%%%%%%%%%

In what follows we verify that any pinned $p$-elastica is either a $(p,r,n)$-arc, a $(p,r,n)$-loop, or a $(p, r, n, \boldsymbol{\sigma}, \boldsymbol{L})$-flat-core.
We first obtain necessary conditions for pinned wavelike $p$-elasticae:

%%%%%%%%%%%%%%%%%%%%%%%%%%%%%%%%%%%%%%
\begin{lemma}[Pinned wavelike $p$-elasticae]\label{lem:wavelike-pinned}
Let $P_0,P_1\in\R^2$ and $L>0$ such that $|P_0-P_1|<L$.
Let $\gamma$ be a critical point of $\mathcal{B}_p$ in $\mathcal{A}_{P_0,P_1,L}$.
Suppose that $\gamma$ is a wavelike $p$-elastica.
Then, up to similarity and reparameterization, $\gamma$ is given by 
\begin{equation}\label{eq:wavelike_pinned}
    \gamma(s)=\gamma_w(s-\K_{1,p}(q), q), \quad s\in [0, 2n\K_{1,p}(q)]
\end{equation}
for some $n\in \N$, where $q$ is a solution to either \eqref{eq:wave-arc-r}
or \eqref{eq:wave-loop-r}. 
\end{lemma}
%%%%%%%%%%%%%%%%%%%%%%%%%%%%%%%%%%%%%%
\begin{proof}
By Proposition~\ref{prop:MY2203-thm1.1}, up to similarity and reparameterization, the curve $\gamma$ is represented by 
\begin{align} \label{eq:pinned-gamma_w}
 \hat\gamma(s)=\gamma_w(s+s_0,q) 
\end{align}
for some $q\in(0,1)$ and $s_0\in\R$, where $\gamma_w$ is defined by \eqref{eq:EP2}.
Let $\hat{L}$ (resp.\ $\hat{l}$) denote the length (resp.\ the distance between the endpoints) of $\hat{\gamma}$.
Note that $\hat{l}/\hat{L}=|P_0-P_1|/L=r$.
We see that the signed curvature of $\hat{\gamma}$ is 
\[
k_w(s+s_0)=2q\cn_p(s+s_0,q), \quad s\in [0, \hat{L}].
\]
By Lemma~\ref{lem:0th-bc}, we have $k_w(s_0)=0$. 
Moreover, in view of the fact that 
\begin{align} \label{eq:cn-zero}
    \cn_p(s,q)=0 \ \iff \  s=(2m-1)\K_{1,p}(q) \ \ \text{for some} \ \  m\in \Z 
\end{align}
and the fact that $\cn_p(\cdot, q)$ is a $2\K_{1,p}(q)$-antiperiodic function (cf.\ Proposition~\ref{prop:property-cndn}), without loss of generality, we can choose $s_0=-\K_{1,p}(q)$ in \eqref{eq:pinned-gamma_w}. 
Furthermore, since $k_w(\hat{L}+s_0)=0$ also follows from Lemma~\ref{lem:0th-bc}, we see that $\hat{L} = 2n \K_{1,p}(q)$ for some $n\in \N$.
To compute $\hat{l}$, set
\[(X(s), Y(s))^\top:=\hat{\gamma}(s)=\gamma_w(s-\K_{1,p}(q),q), \quad s\in [0,\hat{L}]. \]
It follows from \eqref{eq:EP2} that 
\begin{align*}
    X(s) &= 2 \E_{1,p}\big(\am_{1,p}(s-\K_{1,p}(q),q),q \big)  - (s-\K_{1,p}(q)), \\
    Y(s) &= -q\tfrac{p}{p-1}|\cn_p(s-\K_{1,p}(q),q)|^{p-2}\cn_p(s-\K_{1,p}(q),q).
\end{align*}
Combining this with \eqref{eq:cn-zero} and $\hat{L} = 2n \K_{1,p}(q)$, we see that $Y(0)=Y(\hat{L})=0$.
Hence $\hat{l}= |X(\hat{L}) - X(0)|$ holds.
Since $\hat{l}/\hat{L}=r$, we find that $q\in(0,1)$ in \eqref{eq:pinned-gamma_w} has to satisfy 
\begin{align}\label{eq:1116-2}
\frac{X(\hat{L})-X(0)}{2n \K_{1,p}(q)}
=r
\qquad
\text{or}
\qquad
\frac{X(\hat{L})-X(0)}{2n \K_{1,p}(q)}
=-r.
\end{align}
Now we compute the numerator.
By \eqref{eq:period_Ecn} we have
$
(2n-1)\K_{1,p}(q) = \F_{1,p}(n\pi-\frac{\pi}{2}),
$
and hence by definition of $\am_{1,p}$,
\[
\am_{1,p}((2n-1)\K_{1,p}(q),q) = n\pi-\frac{\pi}{2}.
\]
By this fact and \eqref{eq:period_Ecn} we deduce
\begin{align*}
\E_{1,p}(\am_{1,p}((2n-1)\K_{1,p}(q),q),q ) = \E_{1,p}(n\pi-\tfrac{\pi}{2}, q ) 
= (2n-1)\E_{1,p}(q).
\end{align*}
Consequently, we obtain
\begin{align*}
    X(\hat{L})-X(0) = 2n \big(2\E_{1,p}(q) - \K_{1,p}(q) \big).
\end{align*}
Therefore, the left-hand side of both equalities in \eqref{eq:1116-2} are reduced to 
\[
\frac{2n \big(2\E_{1,p}(q) - \K_{1,p}(q) \big)}{2n \K_{1,p}(q)} = 
2\frac{\E_{1,p}(q)}{\K_{1,p}(q)} -1, 
\]
which implies that $q$ in \eqref{eq:pinned-gamma_w} has to satisfy \eqref{eq:wave-arc-r} or \eqref{eq:wave-loop-r}.
The proof is complete.
\end{proof}

Next we obtain necessary conditions for pinned flat-core $p$-elasticae.

%%%%%%%%%%%%%%%%%%%%%%%%%%%%%%%%%%%%%%
\begin{lemma}[Pinned flat-core $p$-elasticae]\label{lem:flat-core-pinned}
Let $P_0,P_1\in\R^2$ and $L>0$ such that $r:=|P_0-P_1|/L\in[0,1)$.
Let $\gamma$ be a critical point of $\mathcal{B}_p$ in $\mathcal{A}_{P_0,P_1,L}$.
Suppose that $\gamma$ is a flat-core $p$-elastica.
Then it is necessary that
\begin{align} \label{eq:1117-3}
r \geq \frac{1}{p-1}.
\end{align}
Moreover, up to similarity and reparameterization, $\gamma$ is given by 
\begin{align}\label{eq:line-sand}
\gamma =\bigg( \bigoplus_{j=1}^n \big( \gamma_{\ell}^{L_j} \oplus \gamma_{b}^{\sigma_j} \big) \bigg) \oplus \gamma_{\ell}^{L_{n+1}} 
\end{align}
for some $n\in \N$, $\{\sigma_j\}_{j=1}^n\subset\{+,-\}$, and  $L_1, \ldots, L_{n+1}\geq0$ such that
 \eqref{eq:sum-flatparts} holds.
\end{lemma}
%%%%%%%%%%%%%%%%%%%%%%%%%%%%%%%%%%%%%%
\begin{proof}
By Proposition~\ref{prop:MY2203-thm1.1}, clearly $p>2$, and in addition, up to similarity and reparameterization, the curve $\gamma$ is represented by $\hat{\gamma}(s)=\gamma_f(s+s_0)$ for $s\in [0,\hat{L}]$ with some $s_0 \in \R$, where $\gamma_f$ is given by \eqref{eq:EP3-2} and $\hat{L}$ is the length of the curve after dilation.
Let $k:[0,\hat{L}] \to \R$ be the signed curvature of $\hat{\gamma}$.
Since $k(0)=k(\hat{L})=0$ follows from Lemma~\ref{lem:0th-bc}, there are $n\in \N$, $\{\sigma_j\}_{j=1}^n\subset\{+, -\}$, $L_1, \ldots, L_{n+1}\geq0$ such that for $s\in [0,\hat{L}]$,
\begin{align} \label{eq:border-pinned}
    \hat{\gamma}(s)=\left(\bigg(\bigoplus_{j=1}^n \big( \gamma_{\ell}^{L_j} \oplus \gamma_{b}^{\sigma_j} \big) \bigg)\oplus \gamma_{\ell}^{L_{n+1}} \right)(s) , 
\end{align}
where $\gamma_b^\pm:[-\K_p(1), \K_p(1)]\to\R^2$ and $\gamma_{\ell}^{L_j}:[0,L_j]\to\R^2$ are defined by \eqref{eq:pm-border}. 
Given the form of the right-hand side of \eqref{eq:border-pinned}, we notice that
\begin{align} \label{eq:hatL-typeIII'}
    \hat{L}=L_\Sigma + 2n\K_p(1), \quad L_\Sigma:=\sum_{j=1}^{n+1} L_j.
\end{align}
Set $\hat{l}:= | \hat{\gamma}(0) - \hat{\gamma}(\hat{L})|$.
Similar to the proof of Lemma~\ref{lem:wavelike-pinned}, we consider the condition to satisfy $\hat{l}/\hat{L} = |P_0-P_1|/L=r$.
We infer from \cite[Lemma 5.7]{MYarXiv2203} that the distance between the endpoints of $\gamma_b^{\pm}$ is
$2\K_p(1)/(p-1)$.
Therefore it follows that 
\[
\hat{l}=n\frac{2\K_p(1)}{p-1}+ L_{\Sigma}. %\quad \text{where} \ \ L_{\Sigma}:=\sum_{j=0}^{N+1} L_j.
\]
This together with \eqref{eq:hatL-typeIII'} yields
\begin{align*}
\frac{L_{\Sigma} +  \frac{2n}{p-1}\K_{p}(1) }{L_{\Sigma} + 2n\K_p(1) } =\frac{\hat{l}}{\hat{L}} =r, 
\end{align*}
which is equivalent to \eqref{eq:sum-flatparts}. 
Thus we find that \eqref{eq:1117-3} is also necessary since the right-hand side of \eqref{eq:sum-flatparts} must be nonnegative.
\end{proof}

Now we complete the proof of Theorem~\ref{thm:classify-pinned}.

\begin{proof}[Proof of Theorem~\ref{thm:classify-pinned}]
Let $\gamma$ be a pinned $p$-elastica. 
By Lemma~\ref{lem:0th-bc}, $\gamma$ is either a wavelike $p$-elastica or a flat-core $p$-elastica.

We first consider the case $r=0$.
In view of \eqref{eq:1117-3}, flat-core $p$-elasticae are ruled out in this case, so it suffices to consider $\gamma$ being a wavelike $p$-elastica.
Then we infer from Lemma~\ref{lem:wavelike-pinned} that, up to similarity and reparameterization, $\gamma$ is given by \eqref{eq:wavelike_pinned} for some $n\in \N$, where $q$ is a solution to \eqref{eq:wave-arc-r} with $r=0$, that is, $q=q^*(p)$ in Definition \ref{def:q*}.
This implies that $\gamma$ is an $\frac{n}{2}$-fold figure-eight $p$-elastica.

Next we assume $r\in(0,\frac{1}{p-1})$.
In this case, again by \eqref{eq:1117-3}, $\gamma$ is a wavelike $p$-elastica. 
We infer from Lemma~\ref{lem:nabe-lem2} that \eqref{eq:wave-arc-r} (resp.\ \eqref{eq:wave-loop-r}) has a unique solution $q\in (0, q^*(p))$ (resp.\ $q\in(q^*(p),1)$).
Hence, by Lemma~\ref{lem:wavelike-pinned}, $\gamma$ is either a $(p,r,n)$-arc or a $(p,r,n)$-loop, for some $n\in \N$. 

The remaining case is $r\in[\frac{1}{p-1},1)$.
First we consider the case where $\gamma$ is a wavelike $p$-elastica. 
From Lemma~\ref{lem:nabe-lem2} we infer that \eqref{eq:wave-arc-r} has a unique solution $q\in (0, q^*(p))$ but \eqref{eq:wave-loop-r} has no solution.
Therefore, in this case $\gamma$ is a $(p,r,n)$-arc for some $n\in \N$. 
If $\gamma$ is a flat-core $p$-elastica, then Lemma~\ref{lem:flat-core-pinned} implies the desired assertion.
\end{proof}

\begin{remark}[Sufficiency]
Conversely, all the candidates in Theorem \ref{thm:classify-pinned} are indeed pinned $p$-elasticae as long as the curves are admissible.
\end{remark}

Below we put down some geometric properties of arcs and loops, which follow by explicit formulae.
Although we focus on the one-fold case, the general case is obtained by its antiperiodic extension (cf.\ Figure \ref{fig:arc}) and hence the symmetry is inherited in a certain sense.

%%%%%%%%%%%%%%%%%%%%%%%%%%%%%%%%%%%%%%
\begin{lemma}
Let $p\in(1,\infty)$ and $r\in[0,1)$.
Let $\gamma$ be either a $(p,r,1)$-arc %with a solution $q$ to \eqref{eq:wave-arc-r} 
or a $(p,r,1)$-loop. %with a solution $q$ to \eqref{eq:wave-loop-r}.
Then $\gamma$ is reflectionally symmetric in the sense that, up to similarity and reparameterization, the curve has an arclength parameterization $\gamma=(X,Y):[0,L]\to\R^2$ satisfying that
\begin{align} \notag%\label{eq:sym-arcloop}
X(s) + X(L-s) = 2X(\tfrac{L}{2}), \quad Y(s) = Y(L-s), \quad \text{for}\ s\in[0,L].    
\end{align}
In addition, if $r>0$ and $\gamma$ is a $(p,r,1)$-loop, then $\gamma$ has a self-intersection, i.e., there is $\sigma \in (0,\tfrac{L}{2})$ such that $\gamma(\sigma) = \gamma(L-\sigma)$.
\end{lemma}
%%%%%%%%%%%%%%%%%%%%%%%%%%%%%%%%%%%%%%
\begin{proof}
Reflection symmetry follows since by Definition \ref{def:arcloop}, up to similarity and reparameterization, $\gamma$ is given by $(X_w,Y_w):=\gamma_w:[-\K_{1,p}(q),\K_{1,p}(q)]\to\R^2$, cf.\ Proposition \ref{prop:MY2203-thm1.1}, for which it is easy to check that $X_w(s)=-X_w(-s)$ and $Y_w(s)=Y_w(s)$. 

Now we check that a $(p,r,1)$-loop has a self-intersection.
By the above symmetry it suffices to find $\sigma\in(0,\K_{1,p}(q))$ such that $X_w(\sigma)=0$.
By the fact that $X_w(0)=0$ and $X_w'(0)=1>0$, and by the intermediate value theorem, it is now sufficient to show that $X_w(\K_{1,p}(q))<0$.
This follows since $X_w(\K_{1,p}(q)) = 2\E_{1,p}(q)-\K_{1,p}(q)=-r\K_{1,p}(q)$, cf.\ \eqref{eq:wave-loop-r}, where $r>0$ and $\K_{1,p}(q)>0$.
\end{proof}

A parallel argument shows that each loop of a flat-core is also a symmetric curve with a self-intersection.

\section{Unique existence and geometric properties of global minimizers} \label{sect:min-unique}

In this section we compute the normalized $p$-bending energy $\overline{\mathcal{B}}_p$ of each pinned $p$-elastica and detect a unique global minimizer, thus proving Theorem~\ref{thm:uniqueness}. 
We then investigate geometric properties of the half-fold figure-eights and prove Theorem \ref{thm:phi*-decrease}.

\subsection{Unique existence}

We first prove the existence of minimizers of $\mathcal{B}_p$ in $\mathcal{A}_{P_0, P_1, L}$ by the standard direct method.
%%%%%%%%%%%%%%%%%%%%%%%%%%%%%%%%%%%%%%
\begin{proposition}\label{prop:direct-method}
Given $P_0, P_1 \in\R^2$ and $L>0$ such that $|P_0-P_1|<L$, there exists a solution to the following minimization problem 
\[\min_{\gamma\in\mathcal{A}_{P_0, P_1, L}} \mathcal{B}_p[\gamma].
\]
\end{proposition} 
%%%%%%%%%%%%%%%%%%%%%%%%%%%%%%%%%%%%%%
\begin{proof}
Let $\{\gamma_j\}_{j\in \N}\subset \mathcal{A}_{P_0, P_1, L}$ be a minimizing sequence of $\mathcal{B}_p$ in $\mathcal{A}_{P_0, P_1, L}$, i.e.,
\begin{align} \label{eq:min_seq}
 \lim_{j\to\infty}\mathcal{B}_p[\gamma_j]=\inf_{\gamma\in \mathcal{A}_{P_0, P_1, L}}\mathcal{B}_p[\gamma]. 
\end{align}
Up to reparameterization, we may suppose that $\gamma_j$ is of constant speed so that $|\gamma'_j| \equiv \mathcal{L}[\gamma_j]=L$.
By \eqref{eq:min_seq} there is $C>0$ such that $\mathcal{B}_p[\gamma_j]\leq C$ for any $j\in \N$, and hence the assumption of constant-speed implies that 
\begin{align} \label{eq:normalized-B_p}
    \Vert \gamma_j'' \Vert_{L^p}^p = \int_0^L \big(L^2|\tilde{\gamma}_j''(s)| \big)^p\,\frac{ds}{L} = L^{2p-1}\mathcal{B}_p[\gamma_j], 
\end{align} 
where $\tilde{\gamma}_j$ stands for the arclength parameterization of $\gamma_j$.
This yields the uniform estimate of $\Vert {\gamma}_j'' \Vert_{L^p}$.
Using $|\gamma_j'|\equiv L$ and the boundary condition, we also obtain the bounds on the $W^{1,p}$-norm. 
Therefore, $\{\gamma_j\}_{j\in\N}$ is uniformly bounded in $W^{2,p}(0,1;\R^2)$ so that there is a subsequence (without relabeling) that converges in the senses of $W^{2,p}$-weak and $C^1$ topology.
Thus the limit curve $\gamma_\infty$ satisfies $\gamma_\infty\in W^{2,p}(0,1;\R^2)$, $|\gamma_\infty'|\equiv L$, $\gamma_\infty(0)=P_0$, and $\gamma_\infty(1)=P_1$, which implies that $\gamma_\infty\in \mathcal{A}_{P_0, P_1, L}$.
Moreover, similar to \eqref{eq:normalized-B_p}, we infer from $|\gamma_\infty'|\equiv L$ that
\begin{align*}
    \mathcal{B}_p[\gamma_\infty] &= %\int_0^L|\tilde{\gamma}_*''(s)|^p\,ds =\mathcal{L}[\gamma_*]^{1-2p}\int_0^1|\gamma_*''(t)|^p\,dt=
    L^{1-2p}\|\gamma_\infty'' \|_{L^p}^p.
\end{align*}
This together with the weak lower semicontinuity for $\|\cdot\|_{L^p}$ ensures that 
\begin{align} \notag%\label{eq:semi-conti}
    \mathcal{B}_p[\gamma_\infty] & \leq \liminf_{j\to\infty}L^{1-2p} \|\gamma_j'' \|_{L^p}^p
    =\liminf_{j\to\infty} \mathcal{B}_p[\gamma_j].
    %= \inf_{\gamma\in \mathcal{A}_{P_0, P_1, L}}\mathcal{B}_p[\gamma]
\end{align}
This with \eqref{eq:min_seq} implies that  $\gamma_\infty$ is a minimizer of $\mathcal{B}_p$ in $\mathcal{A}_{P_0, P_1, L}$. 
\end{proof}

Note that any minimizer of $\mathcal{B}_p$ in $\mathcal{A}_{P_0, P_1, L}$ is a pinned $p$-elastica.
Thanks to the classification in the previous section, for detecting a minimizer it is sufficient to compute the $p$-bending energy of each pinned $p$-elastica.
To this end it is sufficient to compute the {normalized $p$-bending energy}
\[
\overline{\mathcal{B}}_p:=\mathcal{L}^{p-1}\mathcal{B}_p,
\]
cf.\ \eqref{eq:def-normalizedBp}, since we compare the energy of curves of same length.
The energy $\overline{\mathcal{B}}_p$ has the advantage of being scale-invariant in the sense that for any curve $\gamma$ and $\Lambda>0$, if we define $\gamma_\Lambda(t):=\Lambda\gamma(t)$, then $\overline{\mathcal{B}}_p[\gamma]=\overline{\mathcal{B}}_p[\gamma_\Lambda]$.
Thus we need not adjust the length of pinned $p$-elasticae but can use their natural length involving complete $p$-elliptic integrals.

First, we address the wavelike case.
To this end, we prepare the following

%%%%%%%%%%%%%%%%%%%%%%%%%%%%%%%%%%%%%%
\begin{lemma} \label{lem:int-cn_p}
For each $q\in(0,1)$ (including $q=1$ if $p>2$),  
\begin{align} \notag%\label{eq:int-cn_p}
    \int_0^{\K_{1,p}(q)} |\cn_p(s, q)|^p\,ds 
= \frac{1}{q^2}\E_{1,p}(q) + \Big( 1-\frac{1}{q^2}\Big)\K_{1,p}(q).
\end{align}
\end{lemma}
%%%%%%%%%%%%%%%%%%%%%%%%%%%%%%%%%%%%%%
\begin{proof}
Fix $q\in (0,1)$ ($q\in(0,1]$ if $p>2$) arbitrarily. 
By definition of $\cn_p$, we have
\begin{align*}
  \int_0^{\K_{1,p}(q)} |\cn_p(s, q)|^p\,ds 
  &= \int_0^{\K_{1,p}(q)} |\cos \am_{1,p}(s, q)|^2\,ds \\
  &= \int_0^{\frac{\pi}{2}} |\cos \xi|^2 \frac{|\cos\xi|^{1-\frac{2}{p}}}{\sqrt{1-q^2\sin^2 \xi}}\,d\xi, 
\end{align*}
where we used the change of variables $\xi=\am_{1,p}(s, q)$. 
Keeping the definitions of $\K_{1,p}$ and $\E_{1,p}$ in mind, we compute
\begin{align*}
    \int_0^{\frac{\pi}{2}} |\cos \xi|^2 \frac{|\cos\xi|^{1-\frac{2}{p}}}{\sqrt{1-q^2\sin^2 \xi}}\,d\xi 
    &= \frac{1}{q^2}\int_0^{\frac{\pi}{2}} q^2(1-\sin^2 \xi) \frac{|\cos\xi|^{1-\frac{2}{p}}}{\sqrt{1-q^2\sin^2 \xi}}\,d\xi \\
    &= \frac{1}{q^2}\int_0^{\frac{\pi}{2}} (1-q^2\sin^2 \xi) \frac{|\cos\xi|^{1-\frac{2}{p}}}{\sqrt{1-q^2\sin^2 \xi}}\,d\xi \\
    &\quad +\frac{1}{q^2}\int_0^{\frac{\pi}{2}} (q^2-1) \frac{|\cos\xi|^{1-\frac{2}{p}}}{\sqrt{1-q^2\sin^2 \xi}}\,d\xi\\
    &= \frac{1}{q^2}\E_{1,p}(q) + \Big( \frac{q^2-1}{q^2}\Big)\K_{1,p}(q).
\end{align*}
The proof is complete.
\end{proof}

Using this formula, we can compute the normalized $p$-bending energy of $(p,r,n)$-arcs and $(p,r,n)$-loops.

%%%%%%%%%%%%%%%%%%%%%%%%%%%%%%%%%%%%%%
\begin{proposition}\label{prop:normalized-wave}
Let $p\in(1,\infty)$, $r\in[0,1)$, and $n\in \N$. 
Let $\gamma$ be a $(p,r,n)$-arc or a $(p,r,n)$-loop. 
Then the normalized $p$-bending energy of $\gamma$ is given by
\begin{align}\label{eq:normalized-wave}
    \overline{\mathcal{B}}_p[\gamma]=2^{2p}n^p q^p\K_{1,p}(q)^{p-1} \bigg(\frac{1}{q^2}\E_{1,p}(q) + \Big( 1-\frac{1}{q^2}\Big)\K_{1,p}(q) \bigg),
\end{align}
where if $\gamma$ is a $(p,r,n)$-arc, then $q\in(0,1)$ is a solution to \eqref{eq:wave-arc-r}, while if $\gamma$ is a $(p,r,n)$-loop, then $q\in(0,1)$ is a solution to \eqref{eq:wave-loop-r}.
\end{proposition} 
%%%%%%%%%%%%%%%%%%%%%%%%%%%%%%%%%%%%%%
\begin{proof}
Since the proof is completely parallel, we only consider the case that $\gamma$ is a $(p,r,n)$-arc.
Since $\overline{\mathcal{B}}_p$ is invariant with respect to similar transformation and reparameterization, by Definition~\ref{def:arcloop}, we may assume that
\[
\gamma(s)= \gamma_w(s-\K_{1,p}(q),q), \quad \mathcal{L}[\gamma]=2n\K_{1,p}(q),
\]
where $q\in(0,1)$ is a solution of \eqref{eq:wave-arc-r}.
Then we have
\begin{align*}
    \mathcal{B}_p[\gamma]=\int_0^{2n\K_{1,p}(q)}\left|2q\cn_p(s-\K_{1,p}(q),q)\right|^p\,ds
    =2n\int_0^{\K_{1,p}(q)}2^p q^p\left|\cn_p(s,q)\right|^p\,ds,
\end{align*}
where we used the fact that $\cn_p$ is an even $2\K_{1,p}(q)$-antiperiodic function (cf.\ Proposition~\ref{prop:property-cndn}). 
Combining this with Lemma~\ref{lem:int-cn_p}, we obtain \eqref{eq:normalized-wave}.
\end{proof}

Next we turn to the flat-core case.

%%%%%%%%%%%%%%%%%%%%%%%%%%%%%%%%%%%%%%
\begin{proposition}\label{prop:normalized-flat}
Let $p>2$, $r\in[0,1)$, and $n\in \N$ such that \eqref{eq:1117-3} holds. 
Let $\gamma$ be a $(p, r, n, \boldsymbol{\sigma}, \boldsymbol{L})$-flat-core. 
Then the normalized $p$-bending energy of $\gamma$ is given by
\begin{align}\label{eq:normalized-flat}
     \overline{\mathcal{B}}_p[\gamma]= 2^{2p}n^{p}\K_{1,p}(1)^{p-1}\E_{1,p}(1)\left(\frac{1-\frac{1}{p-1}}{1-r}\right)^{p-1}.
\end{align}
\end{proposition} 
%%%%%%%%%%%%%%%%%%%%%%%%%%%%%%%%%%%%%%
\begin{proof}
Let $\gamma$ be a $(p, r, n, \boldsymbol{\sigma}, \boldsymbol{L})$-flat-core, i.e., there are $\boldsymbol{L}=(L_1, \ldots, L_{n+1}) \in[0,\infty)^{n+1}$ and $\boldsymbol{\sigma}=(\sigma_1, \ldots, \sigma_n) \in\{+,-\}^n$ such that $\gamma$ is given by \eqref{eq:line-sand}. 
As discussed in \eqref{eq:hatL-typeIII'}, it follows that $\mathcal{L}[\gamma]=\sum_{j=1}^{n+1} L_j + 2n\K_p(1)$. 
Noting that $\boldsymbol{L}$ satisfies \eqref{eq:sum-flatparts}, we obtain 
\begin{align}\label{eq:length-flattype}
\mathcal{L}[\gamma]
=2n\frac{r-\frac{1}{p-1}}{1-r} \K_{p}(1)+ 2n\K_p(1)
=2n\frac{1-\frac{1}{p-1}}{1-r}\K_{p}(1).
\end{align}
In view of \eqref{eq:flat-type}, we have
\begin{align*}
    \mathcal{B}_p[\gamma]= \int_0^L |k_f(s)|^p\,ds 
    &= \sum_{i=1}^n \int_{s_i-\K_p(1)}^{s_i+\K_p(1)} 2^p(\sech_p (s-s_i))^p\,ds \\
    &= 2^p n \int_{-\K_p(1)}^{\K_p(1)} (\sech_p s)^p\,ds. 
\end{align*}
Moreover, since $\sech_p$ is even (cf.\ Proposition~\ref{prop:property-cndn}), we obtain
\[
\int_{-\K_p(1)}^{\K_p(1)} (\sech_p s)^p\,ds 
= 2\int_0^{\K_p(1)} (\sech_p s)^p\,ds 
= 2\int_0^{\K_p(1)} |\cn_p (s,1)|^p\,ds
= 2\E_{1,p}(1),
\]
where we also used the fact that $\cn_p(\cdot,1) =\sech_p$ in $[0,\K_p(1)]$ and Lemma~\ref{lem:int-cn_p}.
This together with \eqref{eq:length-flattype} yields \eqref{eq:normalized-flat}. 
The proof is complete.
\end{proof}

The following lemma plays an important role for the quantitative comparison of the normalized $p$-bending energy among pinned $p$-elasticae.

%%%%%%%%%%%%%%%%%%%%%%%%%%%%%%%%%%%%%%
\begin{lemma} \label{lem:monotone-w.r.t-q}
Let $b:(0,1)\to\R$ be the function defined by
\begin{align}\label{eq:def-varphi}
b(q):= q^p \bigg(\frac{1}{q^2}\E_{1,p}(q) + \Big( 1-\frac{1}{q^2}\Big)\K_{1,p}(q) \bigg), \quad q\in(0,1),    
\end{align}
and let $b(1):=\lim_{q\uparrow1}b(q)=\E_{1,p}(1)$ if $p>2$.
Then, for $q\in(0,1)$,
\begin{align}\label{eq:bibun-varphi}
    b'(q) = (p-1)q^{p-1}\K_{1,p}(q)+(p-1)\big(1-\tfrac{2}{p}\big)q^{p-3} \big(\E_{1,p}(q)-\K_{1,p}(q) \big)>0.
\end{align}
In particular, $b$ is strictly increasing with respect to $q\in (0,1)$ (including $q=1$ if $p>2$).
\end{lemma}
%%%%%%%%%%%%%%%%%%%%%%%%%%%%%%%%%%%%%%
\begin{proof}
In view of \eqref{eq:diff-ellipint}, we see that for $q\in(0,1)$
\begin{align*}
    b'(q)
    &= (p-2)q^{p-3}\E_{1,p}(q) +q^{p-2}\frac{\E_{1,p}(q)-\K_{1,p}(q)}{q}\\
    &\quad +\Big(pq^{p-1}-(p-2)q^{p-3}\Big)\K_{1,p}(q) \\
    &\quad +(q^p-q^{p-2})\left(-\frac{2-\frac{2}{p}-q^2}{q(1-q^2)}\K_{1,p}(q) + \frac{2-\frac{2}{p}}{q(1-q^2)}\E_{1,p}(q) \right) \\ 
    &= (p-1)q^{p-1}\K_{1,p}(q)+(p-1)\big(1-\tfrac{2}{p}\big)q^{p-3} \big(\E_{1,p}(q)-\K_{1,p}(q) \big).
\end{align*}
Therefore we obtain the left equality in \eqref{eq:bibun-varphi}.
Next we show the positivity of $b'$.
It follows from Proposition~\ref{prop:EK-monotone} that
\[
\frac{\E_{1,p}(q)-\K_{1,p}(q)}{q}
=\frac{d}{dq}\E_{1,p}(q)
=-\int_0^{\frac{\pi}{2}}\frac{q\sin^2\theta |\cos\theta|^{1-\frac{2}{p}}}{\sqrt{1-q^2\sin^2\theta}}d\theta.
\]
This together with \eqref{eq:betsu-K1p} and \eqref{eq:bibun-varphi} yields 
\begin{align*}
     b'(q) &= (p-1)q^{p-1}\int_0^{\frac{\pi}{2}}\frac{ |\cos\theta|^{1-\frac{2}{p}}}{\sqrt{1-q^2\sin^2\theta}}d\theta \\
    &\quad -(p-1)\big(1-\tfrac{2}{p}\big)q^{p-2}\int_0^{\frac{\pi}{2}}\frac{q\sin^2\theta |\cos\theta|^{1-\frac{2}{p}}}{\sqrt{1-q^2\sin^2\theta}}d\theta\\
    &=(p-1)q^{p-1}\int_0^{\frac{\pi}{2}} 
    \left(1-\big(1-\tfrac{2}{p}\big)\sin^2\theta\right)\frac{|\cos\theta|^{1-\frac{2}{p}}}{\sqrt{1-q^2\sin^2\theta}}\,d\theta >0.
\end{align*}
This implies the desired monotonicity and the proof is complete.
\end{proof}

We are in a position to prove the uniqueness of minimizers in $\mathcal{A}_{P_0, P_1, L}$.

\begin{proof}[Proof of Theorem~\ref{thm:uniqueness}]
The existence of minimizers of $\mathcal{B}_p$ in $\mathcal{A}_{P_0, P_1, L}$ follows from Proposition~\ref{prop:direct-method}.
Fix any minimizer $\gamma \in \mathcal{A}_{P_0, P_1, L}$. 
Then $\gamma$ is a pinned $p$-elastica.
We divide the proof into three cases along the classification in Theorem~\ref{thm:classify-pinned}. 

First we consider the case $r=0$.
Then $\gamma$ is an $\frac{n}{2}$-fold figure-eight $p$-elastica, i.e., a $(p,0,n)$-arc for some $n\in \N$.
By Proposition~\ref{prop:normalized-wave}, the normalized $p$-bending energy of a $(p,0,n)$-arc is $n^p$ times that of a $(p,0,1)$-arc, and hence a half-fold figure-eight $p$-elastica (corresponding to $n=1$) is a unique minimizer. 

Next we consider the case $r\in(0,\frac{1}{p-1})$.
Then $\gamma$ is either a $(p,r,n)$-arc or a $(p,r,n)$-loop for some $n\in\N$.
Let $\bar{B}^{r,n}_{\rm arc}$ and $\bar{B}^{r,n}_{\rm loop}$ denote the corresponding values of the normalized $p$-bending energy.
Then, using a solution $q_1\in(0,q^*(p))$ of \eqref{eq:wave-arc-r} and the function $b$ defined by \eqref{eq:def-varphi}, we have 
\begin{align} \notag%\label{eq:normalizedBp-arc}
    \bar{B}^{r,n}_{\rm arc}
=2^{2p}n^p \K_{1,p}(q_1)^{p-1} b(q_1).
\end{align}
Similarly, using a solution $q_2\in(q^*(p),1)$ of \eqref{eq:wave-loop-r}, we obtain
\[
\bar{B}^{r,n}_{\rm loop}
=2^{2p}n^p\K_{1,p}(q_2)^{p-1} b(q_2).
\]
Therefore, we see that $\bar{B}^{r,1}_{\rm arc}< \bar{B}^{r,n}_{\rm arc}$ and $ \bar{B}^{r,1}_{\rm loop} <\bar{B}^{r,n}_{\rm loop}$ for any $n \geq2$.
Furthermore, since $0<q_1 < q^*(p) <q_2<1$, we infer from Proposition~\ref{prop:EK-monotone} and  Lemma~\ref{lem:monotone-w.r.t-q} that 
\[
\bar{B}^{r,1}_{\rm arc} =2^{2p} \K_{1,p}(q_1)^{p-1} b(q_1) < 2^{2p} \K_{1,p}(q_2)^{p-1} b(q_2) = \bar{B}^{r,1}_{\rm loop}.
\]
Consequently a $(p,r,1)$-arc is a unique minimizer.

Finally we turn to the case $r\geq\frac{1}{p-1}$.
In this case, $\gamma$ is either a $(p,r,1)$-arc or a $(p,r,n,\boldsymbol{\sigma},\boldsymbol{L})$-flat-core.
Let $\bar{B}^{r,n}_{\rm flat}$ denote the value of the normalized $p$-bending energy of a  $(p,r,n,\boldsymbol{\sigma},\boldsymbol{L})$-flat-core (which is independent of $\boldsymbol{\sigma}, \boldsymbol{L}$, cf.\ Proposition~\ref{prop:normalized-flat}).
Then we infer from \eqref{eq:normalized-flat} that 
\[
\bar{B}^{r,n}_{\rm flat}
= 2^{2p}n^{p}\K_{1,p}(1)^{p-1}b(1)\left(\frac{1-\frac{1}{p-1}}{1-r}\right)^{p-1}.
\]
Since $\frac{1-\frac{1}{p-1}}{1-r}\geq1$ by \eqref{eq:1117-3}, using Proposition~\ref{prop:EK-monotone} and Lemma~\ref{lem:monotone-w.r.t-q}, we obtain 
\begin{align*}
\bar{B}^{r,1}_{\rm arc}
 &= 2^{2p} b(q_1)\K_{1,p}(q_1)^{p-1}
<2^{2p} b(1)\K_{1,p}(1)^{p-1} \\
&\leq2^{2p}b(1)\K_{1,p}(1)^{p-1}\left(\frac{1-\frac{1}{p-1}}{1-r}\right)^{p-1}
=\bar{B}^{r,1}_{\rm flat}
<\bar{B}^{r,n}_{\rm flat}
\end{align*}
for any $n\geq2$.
As a result a $(p,r,1)$-arc is a unique minimizer.
\end{proof}

\subsection{Properties of a half-fold figure-eight $p$-elastica}\label{subsect:angle-half-8}

In this subsection, we discuss some properties of a half-fold figure-eight $p$-elastica, 
and in particular, we prove Theorem~\ref{thm:phi*-decrease}.

To begin with, we collect the basic properties of a half-fold figure-eight $p$-elastica.

%%%%%%%%%%%%%%%%%%%%%%%%%%%%%%%%%%%%%%
\begin{proposition}[Basic properties of a half-fold figure-eight $p$-elastica]\label{prop:half-fold}
Let $\gamma$ be a half-fold figure-eight $p$-elastica.
Then up to similarlity and reparameterization, $\gamma$ is given by, for $s\in [0,2\K_{1,p}(q^*(p))]$, 
\begin{align}\label{eq:half-fold}
\begin{split}
    \gamma(s)%&=\gamma_w(s+\K_{1,p}(q^*),q^*) \\
    &=
       \begin{pmatrix}
       2 \E_{1,p}(\am_{1,p}(s-\K_{1,p}(q^*),q^*),q^* )- (s-\K_{1,p}(q^*)) \\
       -q^*\frac{p}{p-1}|\cn_p(s-\K_{1,p}(q^*),q^*)|^{p-2}\cn_p(s-\K_{1,p}(q^*),q^*)
       \end{pmatrix},
\end{split}
\end{align}
where $q^*=q^*(p)$.
In addition, the following properties hold.
\begin{itemize}
    \item[(i)] The tangential angle is $\theta(s)=2\arcsin{(q^*(p)\sn_p(s-\K_{1,p}(q^*(p)),q^*(p)))}$. 
    \item[(ii)] The curve $\gamma$ defined by \eqref{eq:half-fold} takes the origin if and only if $s=0,2\K_{1,p}(q^*(p))$.
    \item[(iii)] Let $\phi^*:(1,\infty)\to(0,\pi)$ be 
    \begin{align}\label{eq:phi*}
      \phi^*(p)%=\pi-2\arcsin(q^*\sn_p(\K_{1,p}(q^*), q^*))
      =\pi-2\arcsin(q^*(p)). 
    \end{align}
    Then $\gamma'(2\K_{1,p}(q^*(p)))=R_{-2\phi^*(p)}\gamma'(0)$, where $R_\theta$ stands for the counterclockwise rotation matrix through angle $\theta\in \R$.
\end{itemize}
\end{proposition}
%%%%%%%%%%%%%%%%%%%%%%%%%%%%%%%%%%%%%%
\begin{proof}
By Proposition~\ref{prop:MY2203-thm1.1} and by definition of a half-fold figure-eight $p$-elastica, up to similarity and reparameterization, we have $\gamma(s)=\gamma_w(s-\K_{1,p}(q^*(p)),q^*(p))$ implying \eqref{eq:half-fold} and property (i).
Property (ii) follows from \eqref{eq:half-fold}.
We finally prove property (iii).
By property (i) we have $\theta(0)=-2\arcsin(q^*(p))$ and $\theta(2\K_{1,p}(q^*(p)))=2\arcsin(q^*(p))$.
By definition of $\phi^*$, it follows that
$\theta(2\K_{1,p}(q^*(p)))+2\phi^*(p)=2\pi-2\arcsin(q^*(p))=2\pi+\theta(0)$.
This together with $\gamma'(s)=(\cos{\theta(s)}, \sin{\theta(s)})$ yields $R_{2\phi^*(p)} \gamma'(2\K_{1,p}(q^*(p)))= \gamma'(0)$ and hence property (iii).
\end{proof}

From the previous proposition we see that $\phi^*(p)$ characterizes the crossing angle of a figure-eight $p$-elastica, cf.\ Figure~\ref{fig:phi^*}.
Notice that the parameterization in \eqref{eq:half-fold} rotated through angle $\pi$ represents the left curve in Figure~\ref{fig:phi^*}.

Hereafter we investigate properties of $\phi^*(p)$ through the analysis of $q^*(p)$.
Noting that $q^*(p)$ is a unique solution of $Q_p(q)=0$, cf.\ \eqref{eq:1116-4}, we introduce 
\begin{align} \label{eq:def-calQ_p}
    \widetilde{Q}_p(q):= \K_{1,p}(q) Q_p(q) = 2\E_{1,p}(q)- \K_{1,p}(q), \quad q\in[0,1).
\end{align}
Since $\K_{1,p}(q)>0$, the roots of $\widetilde{Q}_p$ and $Q_p$ coincide.
Both $\widetilde{Q}_p$ and $Q_p$ have own advantages; in general the former is easy to compute at least for a fixed $p$, but we will later discover that the latter has a remarkable monotone structure as $p$ varies, cf.\ Figure \ref{fig:Qp}.

%%%%%%%%%%%%%%%%%%%%%%%%%%%%%%%%%%%%%%
\begin{lemma}\label{lem:calQ_p}
Let $\widetilde{Q}_p:[0,1)\to \R$ be the function defined by \eqref{eq:def-calQ_p}.
Then $\widetilde{Q}_p$ is strictly decreasing and strictly concave. 
In addition, if $p>2$, then 
\begin{align} \label{eq:calQ(1)}
    \widetilde{Q}_p(1):= \lim_{q\uparrow 1}\widetilde{Q}_p(q)
    =-\frac{\K_{1,p}(1)}{p-1} \in (-\infty,0).
\end{align}
\end{lemma}
%%%%%%%%%%%%%%%%%%%%%%%%%%%%%%%%%%%%%%
\begin{proof}
By Proposition~\ref{prop:EK-monotone}, $2\E_{1,p}(q)$ and $-\K_{1,p}(q) $ are strictly decreasing, and hence so is $\widetilde{Q}_p$.
By the change of variables, we can rewrite \eqref{eq:EK-bibun-q} as
\begin{align} \label{eq:EK-bibun-q-2}
\begin{split}
\frac{d}{dq}\K_{1,p}(q)&= \int_0^1\frac{qz^2}{(1-q^2z^2)^{\frac{3}{2}}}(1-z^2)^{-\frac{1}{p}}dz, 
\\
\frac{d}{dq}\E_{1,p}(q)&=-\int_0^1\frac{qz^2}{(1-q^2z^2)^{\frac{1}{2}}}(1-z^2)^{-\frac{1}{p}}dz.
\end{split}
\end{align}
In particular we have 
\begin{align} \label{eq:bibun-calQ}
\frac{d}{dq}\widetilde{Q}_p(q)&=2\frac{d}{dq}\E_{1,p}(q)-\frac{d}{dq}\K_{1,p}(q)=
\int_0^1\frac{qz^2(-3+2q^2z^2)}{(1-q^2z^2)^{\frac{3}{2}}}(1-z^2)^{-\frac{1}{p}}dz.
\end{align}
Differentiating once more, we obtain for $q\in(0,1)$ 
\begin{align*}
\frac{d^2}{dq^2}\widetilde{Q}_p(q)
&= -3\int_0^1\frac{z^2}{(1-q^2z^2)^{\frac{5}{2}}}(1-z^2)^{-\frac{1}{p}}dz<0.
\end{align*}
This implies the concavity of $\widetilde{Q}_p$.
It remains to show \eqref{eq:calQ(1)} for $p>2$.
Combining \eqref{eq:1117-1} with the fact that $\K_{1,p}(1)<\infty$ (cf.\ \eqref{eq:K_p}), we obtain \eqref{eq:calQ(1)}.
\end{proof}

In what follows we prove that $\phi^*(p)$ is continuously decreasing from $\pi/2$ to $0$ as $p$ varies from $1$ to $\infty$ (cf.\ Figure~\ref{fig:plot-phi}).
First we prepare some basic properties of $\phi^*(p)$. 

%%%%%%%%%%%%%%%%%%%%%%%%%%%%%% Maple figure ver.1 
\begin{center}
    \begin{figure}[htbp]
      \includegraphics[scale=0.3]{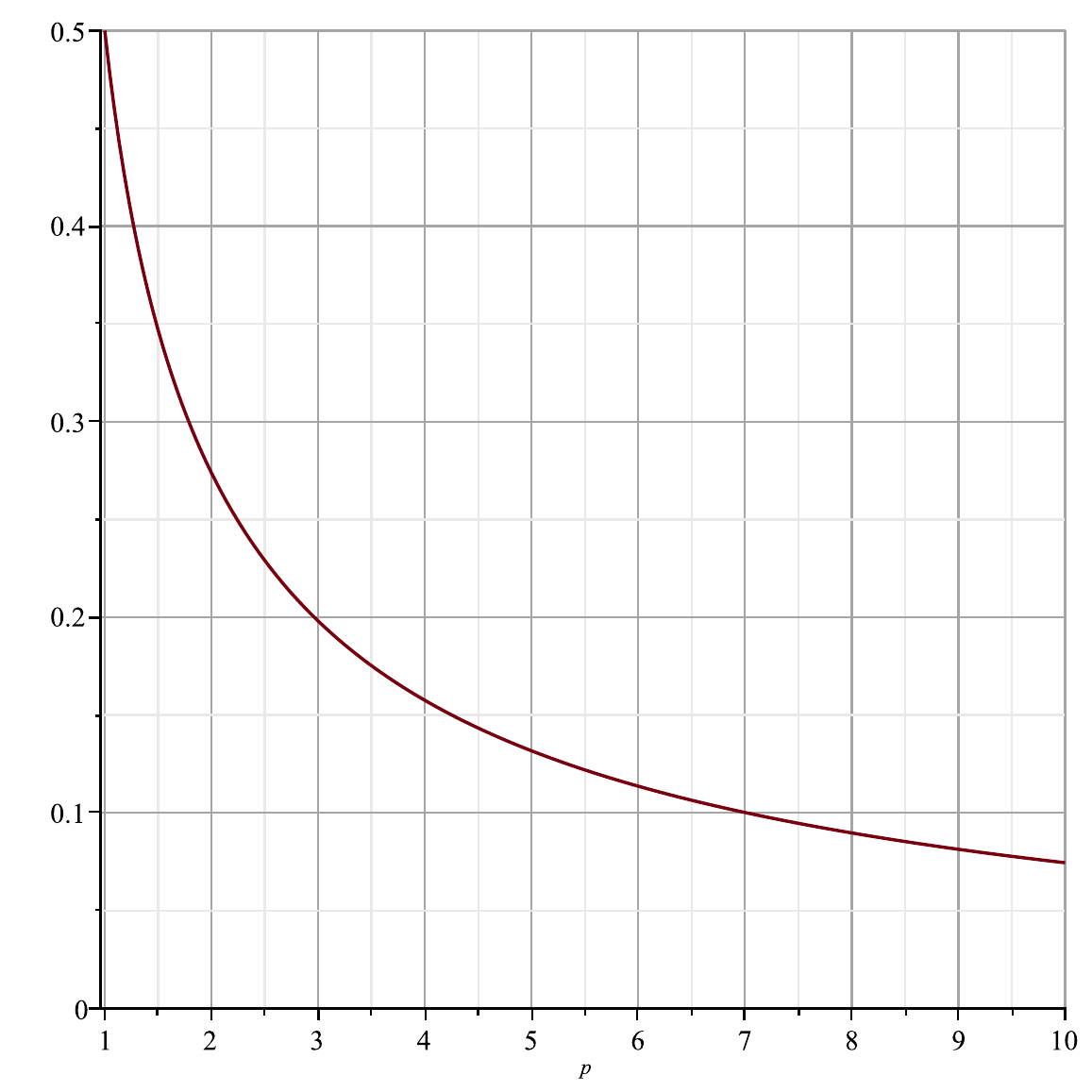}
  \caption{The graph of $p\mapsto \phi^*(p)/\pi$.}
  \label{fig:plot-phi}
  \end{figure}
\end{center}
%%%%%%%%%%%%%%%%%%%%%%%%%%%%%% Maple figure ver.1 

%%%%%%%%%%%%%%%%%%%%%%%%%%%%%%%%%%%%%%
\begin{proposition}\label{prop:phi*-limit}
Let $\phi^*:(1,\infty)\to \R$ be given by \eqref{eq:phi*}.
Then $\phi^*(p)$ satisfies the following properties. 
\begin{enumerate}
    \item[(i)] $0 < \phi^*(p) < \pi/2$  \ for any \  $p\in(1,\infty)$.
    \item[(ii)] $\phi^*(p)\to \pi/2$ \ as \ $p\downarrow1$.
    \item[(iii)] $\phi^*(p)\to 0$ \ as \ $p\to \infty$.
    \item[(iv)] $\phi^*$ is continuously differentiable on $(1,\infty)$. 
\end{enumerate}
\end{proposition}
%%%%%%%%%%%%%%%%%%%%%%%%%%%%%%%%%%%%%%
\begin{proof}
First we check property (i). 
To this end, by \eqref{eq:phi*} it suffices to show that $1/\sqrt{2} < q^*(p) < 1$ for any $p\in(1,\infty)$.
From the fact that 
\[
 Q_p(\tfrac{1}{\sqrt{2}}) = \K_{1,p}(\tfrac{1}{\sqrt{2}})^{-1}\int_0^{\frac{\pi}{2}}\frac{1-\sin^2\theta}{\sqrt{1-\frac{1}{2}\sin^2\theta}}|\cos\theta|^{1-\frac{2}{p}}\,d\theta >0 = Q_p(q^*(p))
\]
and monotonicity of $Q_p$ (cf.\ Lemma~\ref{lem:nabe-lem2}), it follows that $q^*(p) > 1/\sqrt{2}$. 
By definition of $q^*(p)$, it also follows that $q^*(p) < 1$.

Next we show property (ii). 
In view of \eqref{eq:phi*}, it suffices to show that $q^*(p) \to 1/\sqrt{2}$ as $p\downarrow 1$. 
By concavity of $\widetilde{Q}_p$ (cf.\ Lemma~\ref{lem:calQ_p}), it turns out that 
\begin{align} \label{eq:concave-calQ-1}
    \frac{\widetilde{Q}_p(q^*(p))- \widetilde{Q}_p(\tfrac{1}{\sqrt{2}})}{q^*(p)-\tfrac{1}{\sqrt{2}}} < \widetilde{Q}_p'(\tfrac{1}{\sqrt{2}}).
\end{align}
By definition we have
\begin{align*}
    \widetilde{Q}_p(q)= \int_0^1\frac{1-2q^2z^2}{\sqrt{1-q^2z^2}}(1-z^2)^{-\frac{1}{p}}\,dz,
    \quad q\in[0,1), 
\end{align*} 
from which it follows that
\begin{align} \label{eq:calQ-pi/4}
    0<\widetilde{Q}_p(\tfrac{1}{\sqrt{2}}) &= \int_0^1\frac{1-z^2}{\sqrt{1-\frac{1}{2}z^2}}(1-z^2)^{-\frac{1}{p}}\,dz 
    \leq \int_0^1\frac{1}{\sqrt{1-\frac{1}{2}z^2}}\,dz =:C_1. 
\end{align}
In addition, using \eqref{eq:bibun-calQ}, we obtain 
\begin{align} \label{eq:calQ'-1/sqrt2}
\begin{split}
    -\,\widetilde{Q}_p'(\tfrac{1}{\sqrt{2}}) 
    &=\int_0^1\frac{z^2(3-z^2)}{\sqrt{2}(1-\frac{1}{2}z^2)^{\frac{3}{2}}}(1-z^2)^{-\frac{1}{p}}\,dz 
     \geq \frac{2}{\sqrt{2}} \int_0^1z^2(1-z)^{-\frac{1}{p}}\,dz \\
    & \geq\sqrt{2} \int_{\frac{1}{2}}^1\frac{1}{4}(1-z)^{-\frac{1}{p}}\,dz
    \geq \frac{1}{4\sqrt{2}}\frac{p}{p-1}.
\end{split}
\end{align}
Combining this with \eqref{eq:concave-calQ-1} and \eqref{eq:calQ-pi/4}, we have 
\[
0<q^*(p)-\tfrac{1}{\sqrt{2}} <  \frac{\widetilde{Q}_p(\tfrac{1}{\sqrt{2}})}{-\widetilde{Q}_p'(\tfrac{1}{\sqrt{2}})} \leq 4\sqrt{2} C_1 \frac{p-1}{p}, 
\]
which yields $q^*(p) \to 1/\sqrt{2}$ as $p\downarrow 1$.

We turn to property (iii). 
To this end, we show that $q^*(p) \to 1$ as $p\to \infty$. 
By Lemma~\ref{lem:calQ_p}, if $p>2$, then $\widetilde{Q}_p(1)$ is well defined, and we infer from concavity of $\widetilde{Q}_p$ that 
\begin{align} \label{eq:concave-calQ-2}
    \frac{\widetilde{Q}_p(1)-\widetilde{Q}_p(q^*(p))}{1-q^*(p)} < \widetilde{Q}_p'(q^*(p)).
\end{align}
Moreover, concavity of $\widetilde{Q}_p$ and $q^*(p)>1/\sqrt{2}$ also imply that 
\begin{align}\label{eq:calQ-bibun-q-*}
    \widetilde{Q}_p'(q^*(p)) < \widetilde{Q}_p'(\tfrac{1}{\sqrt{2}}) <0.
\end{align}
This together with \eqref{eq:calQ(1)} and \eqref{eq:concave-calQ-2} gives 
\[
0<1-q^*(p) < \frac{\widetilde{Q}_p(1) }{\widetilde{Q}_p'(q^*(p)) }
\leq \frac{\K_{1,p}(1)}{p-1}\frac{1}{-\widetilde{Q}_p'(\tfrac{1}{\sqrt{2}}) }\leq 4 \sqrt{2} \K_{1,p}(1)\frac{1}{p},
\]
where we also used \eqref{eq:calQ'-1/sqrt2}.
The right-hand side tends to be zero since
\[
\K_{1,p}(1) = \int_0^1(1-z^2)^{-\frac{1}{2}-\frac{1}{p}}\,dz = O(1), \quad\text{as}\quad p\to\infty.
\]
Consequently, we obtain $q^*(p) \to 1$ as $p\to \infty$.

It remains to prove property (iv).
Let $F$ be a function defined by $F:(1,\infty)\times(0,1) \to \R;(p,q) \mapsto \widetilde{Q}_p(q)$. 
Then, $F$ is continuously differentiable with respect to $p$ and $q$.
Indeed, it follows that
\begin{align*}
    \frac{\partial F}{\partial p}(p, q) &= \int_0^1\frac{1-2q^2z^2}{\sqrt{1-q^2z^2}}\frac{1}{p^2}(1-z^2)^{-\frac{1}{p}}\log(1-z^2)\,dz, \\
    \frac{\partial F}{\partial q}(p, q) &= %\frac{d}{dq}(2\E_{1,p}(q)-\K_{1,p}(q))=
    \int_0^1\frac{qz^2(-3+2q^2z^2)}{(1-q^2z^2)^{\frac{3}{2}}}(1-z^2)^{-\frac{1}{p}}\,dz.
\end{align*}
From definition of $q^*(p)$ and \eqref{eq:calQ-bibun-q-*}, we infer that  
\[ F(p, q^*(p))=\widetilde{Q}_p(q^*(p))=0, \quad \frac{\partial F}{\partial q}(p, q^*(p))=\widetilde{Q}_p'(q^*(p))<0. \] 
Hence, by the implicit function theorem, for each $p\in(1,\infty)$ there exists a neighborhood $U$ of $(p,q^*(p))$ and a function $f:U\to\R$ such that $F(p, f(p))=0$.
Moreover, Lemma~\ref{lem:nabe-lem2} implies uniqueness of solutions $q$ to $F(p,q)=\widetilde{Q}_p(q)=0$ for each $p\in(1,\infty)$, and hence we see that $q^*(p)=f(p)$ in $U$. 
Furthermore, by differentiability of $F$, we notice that $f \in C^1(U)$.
Consequently the map $p\mapsto q^*(p)$ is continuously differentiable with respect to $p\in(1,\infty)$.
The proof is complete.
\end{proof}

In order to obtain monotonicity of $\phi^*$, we here prove the key fact that $Q_p$ has a remarkable monotone structure with respect to $p$ as in the left of Figure~\ref{fig:Qp}, in contrast to $\widetilde{Q}_p$ as in the right of Figure~\ref{fig:Qp}.

%%%%%%%%%%%%%%%%%%%%%%%%%%%%%% Maple figure ver.1 
\begin{center}
    \begin{figure}[htbp]
      \includegraphics[scale=0.3]{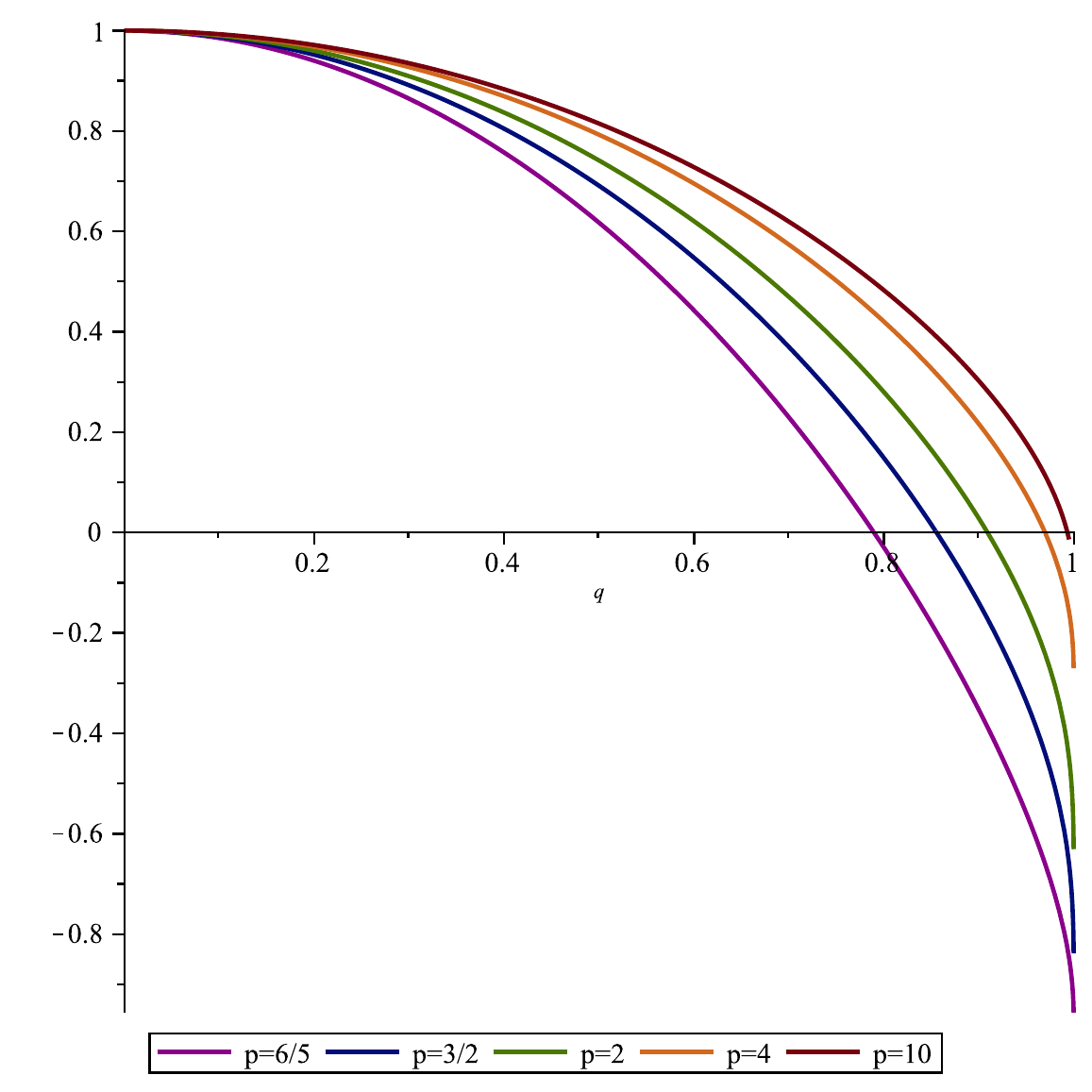}
      \includegraphics[scale=0.3]{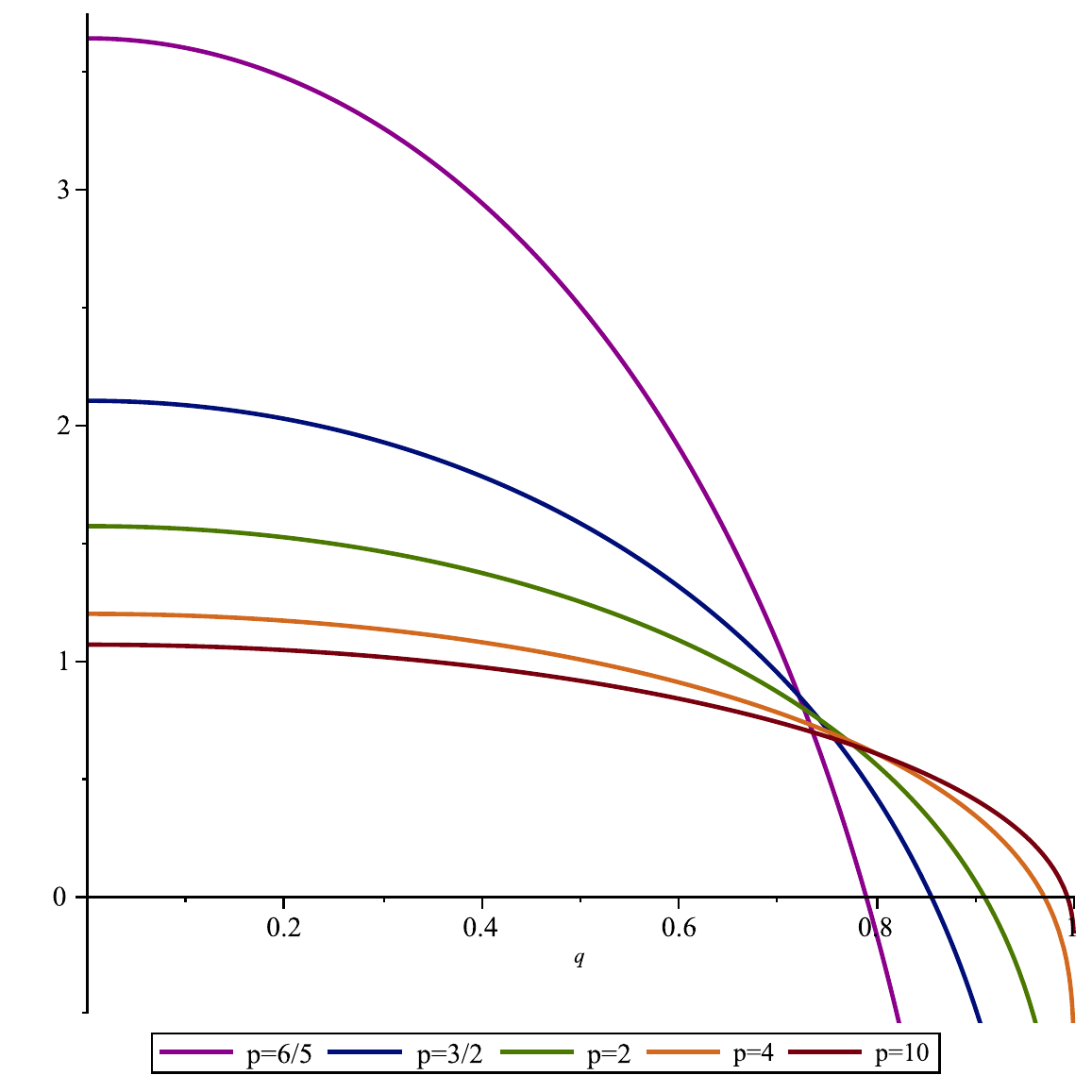}
  \caption{Graphs of $Q_p$ (left) and graphs of $\widetilde{Q}_p$ (right).}
  \label{fig:Qp}
  \end{figure}
\end{center}
%%%%%%%%%%%%%%%%%%%%%%%%%%%%%% Maple figure ver.1 

%%%%%%%%%%%%%%%%%%%%%%%%%%%%%%%%%%%%%%
\begin{proposition}\label{prop:Qp-monotone}
If $1< p_1 < p_2 < \infty$, then
\begin{align}\label{eq:layer}
    Q_{p_1}(q) < Q_{p_2}(q) \quad \text{for any} \ \ q\in(0,1).
\end{align}
\end{proposition}
%%%%%%%%%%%%%%%%%%%%%%%%%%%%%%%%%%%%%%
\begin{proof}
Fix $1< p_1 < p_2 < \infty$ arbitrarily and define 
\[
g(q):= Q_{p_1}(q) - Q_{p_2}(q), \quad q\in[0,1).
\]
First we show that there is $\delta>0$ such that 
\begin{align}\label{eq:g0-negative}
    g(q)<0 \quad \text{for all} \quad q\in(0,\delta].
\end{align}
Note that \eqref{eq:EK-bibun-q-2} yields $\K'_{1,p}(0)=\E'_{1,p}(0)=0$, where $'=\frac{d}{dq}$.
Combining this with
\[
Q''_{p}(q)=2\frac{\E_{1,p}''(q)}{\K_{1,p}(q)} - 4\frac{\E_{1,p}'(q)\K_{1,p}'(q)}{\K_{1,p}(q)^2} +4\frac{\E_{1,p}(q)\K_{1,p}'(q)^2}{\K_{1,p}(q)^3} -2\frac{\E_{1,p}(q)\K_{1,p}''(q)}{\K_{1,p}(q)^2},
\]
we obtain 
\[
Q''_{p}(0)=\frac{2}{\K_{1,p}(0)} \Big(\E_{1,p}''(0)-\K_{1,p}''(0)\Big),
\]
where we also used $\K_{1,p}(0)=\E_{1,p}(0)$.
Moreover, it follows from \eqref{eq:EK-bibun-q-2} that 
\[
\E_{1,p}''(0)=-\int_0^1 z^2(1-z^2)^{-\frac{1}{p}}\,dz, \quad
\K_{1,p}''(0)=\int_0^1 z^2(1-z^2)^{-\frac{1}{p}}\,dz.
\]
Therefore, using the Beta function $\mathrm{B}(x,y)=\int_0^1 t^{x-1}(1-t)^{y-1}\,dt$, we see that
\[
Q''_{p}(0)
=-4\dfrac{\int_0^1 z^2(1-z^2)^{-\frac{1}{p}}\,dz}{\int_0^1(1-z^2)^{-\frac{1}{p}} \,dz}
=-4\dfrac{\int_0^1 t(1-t)^{-\frac{1}{p}}\,\frac{dt}{2\sqrt{t}}}{\int_0^1(1-t)^{-\frac{1}{p}} \,\frac{dt}{2\sqrt{t}}}
=-4\frac{\mathrm{B}\big(\frac{3}{2},1-\frac{1}{p}\big)}{\mathrm{B}\big(\frac{1}{2},1-\frac{1}{p}\big)}.
\]
By the well-known relations $\mathrm{B}(x,y)=\frac{\Gamma(x)\Gamma(y)}{\Gamma(x+y)}$ and $\Gamma(x+1)=x\Gamma(x)$, where $\Gamma$ denotes the Gamma function,
we obtain
\[
Q''_{p}(0)
=-4\,\frac{\Gamma\big(\frac{3}{2}\big)\Gamma\big(\frac{3}{2}-\frac{1}{p}\big)}{\Gamma\big(\frac{5}{2}-\frac{1}{p}\big)\Gamma\big(\frac{1}{2}\big)}
=\frac{4p}{2-3p}.
\]
This implies that $Q''_{p}(0)$ is strictly increasing with respect to $p\in(1,\infty)$, so we have $g''(0)=Q''_{p_1}(0)-Q''_{p_2}(0)<0$.
This together with $g(0)=g'(0)=0$ yields \eqref{eq:g0-negative}.

Next we calculate the derivative of $g$. 
Using \eqref{eq:diff-ellipint}, we obtain
\begin{align*}
   Q'_{p}(q)
    &=2\frac{\E_{1,p}'(q)}{\K_{1,p}(q)} -2\frac{\E_{1,p}(q)\K_{1,p}'(q)}{\K_{1,p}(q)^2}\\
    &=-\frac{2}{q} + 2\left(\frac{1-2q^2}{q(1-q^2)}+\frac{2\hat{p}}{q(1-q^2)}\right)\frac{\E_{1,p}(q)}{\K_{1,p}(q)} -\frac{4\hat{p}}{q(1-q^2)}\frac{\E_{1,p}(q)^2}{\K_{1,p}(q)^2},
\end{align*}
where we put $\hat{p}:=1-\frac{1}{p}$.
Therefore, 
\begin{align} \label{eq:g-bibun}
\begin{split}
        g'(q)
    &=Q'_{p_1}(q)-Q'_{p_2}(q) \\
    &=2\left(\frac{1-2q^2}{q(1-q^2)}+\frac{2\hat{p}_1}{q(1-q^2)}\right)\frac{\E_{1,p_1}(q)}{\K_{1,p_1}(q)} -\frac{4\hat{p_1}}{q(1-q^2)}\frac{\E_{1,p_1}(q)^2}{\K_{1,p_1}(q)^2}\\
    &\qquad-2\left(\frac{1-2q^2}{q(1-q^2)}+\frac{2\hat{p}_2}{q(1-q^2)}\right)\frac{\E_{1,p_2}(q)}{\K_{1,p_2}(q)} +\frac{4\hat{p_2}}{q(1-q^2)}\frac{\E_{1,p_2}(q)^2}{\K_{1,p_2}(q)^2}\\
     &=\frac{1-2q^2}{q(1-q^2)} g(q)
     +\frac{2\hat{p}_2}{q(1-q^2)} g(q)
     -\frac{2\hat{p}_2}{q(1-q^2)} g(q)\left(\frac{\E_{1,p_1}(q)}{\K_{1,p_1}(q)}+\frac{\E_{1,p_2}(q)}{\K_{1,p_2}(q)}\right) \\
     &\qquad+ \frac{4}{q(1-q^2)}(\hat{p}_1-\hat{p}_2)\frac{\E_{1,p_1}(q)}{\K_{1,p_1}(q)}\left(1-\frac{\E_{1,p_1}(q)}{\K_{1,p_1}(q)} \right).
\end{split}
\end{align}

Now we are ready to prove \eqref{eq:layer}, i.e., 
\begin{align}\label{eq:layer-g}
    g(q) <0 \quad \text{for any}\quad q\in(0,1).
\end{align}
With the help of \eqref{eq:g0-negative}, in order to obtain \eqref{eq:layer-g}, it suffices to show that $g(q)\neq0$ for any $q\in(0,1)$. 
We prove this by contradiction. 
Suppose that there is $q_0\in(0,1)$ such that $g(q_0)=0$. 
In view of \eqref{eq:g0-negative} we may assume that $q_0=\inf\Set{q\in(0,1) | g(q)=0}>0$.
Combining $g(q_0)=0$ with \eqref{eq:g-bibun}, we have 
\begin{align}\label{eq:g'-negative}
    g'(q_0)= \frac{4}{q_0(1-q_0^2)}(\hat{p}_1-\hat{p}_2)\frac{\E_{1,p_1}(q_0)}{\K_{1,p_1}(q_0)}\left(1-\frac{\E_{1,p_1}(q_0)}{\K_{1,p_1}(q_0)} \right).
\end{align}
Note that Proposition~\ref{prop:EK-monotone} and $\K_{1,p}(0)=\E_{1,p}(0)$ yield $0<\E_{1,p}(q_0)<\K_{1,p}(q_0)$. 
Since $\hat{p}_1<\hat{p}_2$, we infer from \eqref{eq:g'-negative} that $g'(q_0)<0$. 
However, this together with the fact that $g<0$ in $(0,q_0)$ contradicts $g(q_0)=0$. 
Thus we obtain \eqref{eq:layer-g}. 
\end{proof}

Theorem~\ref{thm:phi*-decrease} now follows by Propositions~\ref{prop:phi*-limit} and \ref{prop:Qp-monotone}.

\begin{proof}[Proof of Theorem~\ref{thm:phi*-decrease}]
Recall from Proposition~\ref{prop:phi*-limit} that $p\mapsto\phi^*(p)$ is continuous and satisfies $\phi^*(p)\to \pi/2$ as $p\downarrow1$, and $\phi^*(p)\to 0$ as $p\to \infty$.
Then, in view of \eqref{eq:phi*}, it suffices to prove that
$q^*(p_1) < q^*(p_2)$ for any $1<p_1<p_2<\infty$.
We infer from \eqref{eq:layer} that
\[
Q_{p_2}(q^*(p_2)) = 0 = Q_{p_1}(q^*(p_1)) < Q_{p_2}(q^*(p_1)), 
\]
which in combination with monotonicity of $Q_{p_2}$ (cf.\ Lemma~\ref{lem:nabe-lem2}) yields the desired monotonicity of $q^*(p)$.
\end{proof}

\section{Li--Yau type multiplicity inequality} \label{sect:Li-Yau}

This section is devoted to discussing the Li--Yau type multiplicity inequality of the form \eqref{eq:LiYau-Bp}. 
In Subsection~\ref{subsect:LYineq} we first prove Theorem~\ref{thm:LiYau-Bp}. 
The key fact to prove Theorem~\ref{thm:LiYau-Bp} is that a half-fold figure-eight $p$-elastica is a unique minimizer of $\mathcal{B}_p$ in $\mathcal{A}_{P_0,P_0, L}$.
From the proof we observe that whether the equality in \eqref{eq:LiYau-Bp} is attained is equivalent to whether there exists a closed $m$-leafed $p$-elastica, i.e., a $C^1$-concatenation of $m$ half-fold figure-eight $p$-elasticae.
Then we discuss existence of closed $m$-leafed $p$-elasticae in Subsection~\ref{subsect:leaf-existence}, and observe new phenomena related to optimality in Subsection~\ref{subsect:optimal}.
Here the angle monotonicity in Theorem~\ref{thm:phi*-decrease} plays a key role.

\subsection{Multiplicity inequality}\label{subsect:LYineq}
Theorems~\ref{thm:classify-pinned} and \ref{thm:uniqueness} imply
%%%%%%%%%%%%%%%%%%%%%%%%%%%%%%%%%%%%%%
\begin{corollary}\label{cor:pinned-Case1}
Let $\varpi^*_p>0$ be a constant given by
\begin{align}\label{eq:def-varpi_p}
    \varpi^*_p = 2^{3p-1}(q^*(p))^{p-2}\big(2(q^*(p))^2-1\big) \E_{1,p}(q^*(p))^p, 
\end{align}
and $\gamma\in W^{2,p}_{\rm imm}(0,1;\R^2)$ satisfy $\gamma(0)=\gamma(1)$.
Then
\begin{align*}
    \overline{\mathcal{B}}_p[\gamma] \geq \varpi^*_p,
\end{align*}
where equality is attained if and only if $\gamma$ is a half-fold figure-eight $p$-elastica.
\end{corollary}
%%%%%%%%%%%%%%%%%%%%%%%%%%%%%%%%%%%%%%
\begin{proof}
Fix an arbitrary $\gamma\in W^{2,p}_{\rm imm}(0,1;\R^2)$ such that $\gamma(0)=\gamma(1)$ and let $\bar{\gamma}$ be a half-fold figure-eight $p$-elastica.
It follows from Theorem~\ref{thm:uniqueness} that  $\overline{\mathcal{B}}_p[\gamma]\geq \overline{\mathcal{B}}_p[\bar{\gamma}]$, and equality is attained if and only if $\gamma$ is also a half-fold figure-eight elastica.
Recalling the fact that $2\E_{1,p}(q^*(p))=\K_{1,p}(q^*(p))$, we infer from \eqref{eq:normalized-wave} with $q=q^*(p)$ and $n=1$ that
\begin{align*}
    \overline{\mathcal{B}}_p[\bar{\gamma}]
    =2^{3p-1}q^*(p)^{p-2}\big(2q^*(p)^2-1\big) \E_{1,p}(q^*(p))^p,
\end{align*}
which coincides with \eqref{eq:def-varpi_p}.
The proof is complete.
\end{proof}

In order to prove inequality \eqref{eq:LiYau-Bp}, we first prove its open-curve counterpart.

%%%%%%%%%%%%%%%%%%%%%%%%%%%%%%%%%%%%%%
\begin{theorem}[Multiplicity inequality for open curves]\label{thm:LiYau-Bp-open}
Let $\gamma\in W^{2,p}_{\rm imm}(0,1;\R^2)$ be a curve with a point of multiplicity $m\geq2$.
Then
\begin{align}\label{eq:open-counterpart}
    \overline{\mathcal{B}}_p[\gamma] \geq \varpi^*_p(m-1)^p.
\end{align}
\end{theorem}
%%%%%%%%%%%%%%%%%%%%%%%%%%%%%%%%%%%%%%
\begin{proof}
By the assumption on multiplicity, there are $0\leq a_1 < \cdots < a_m \leq 1$ such that $\gamma(a_1) = \cdots = \gamma(a_m)$. 
If $a_1 \neq0$, then $ \overline{\mathcal{B}}_p[\gamma] >  \overline{\mathcal{B}}_p[\gamma|_{[a_1, 1]}]$ follows and it suffices to show that $\overline{\mathcal{B}}_p[\gamma|_{[a_1, 1]}]\geq \varpi^*_p(m-1)^p$.
Therefore we may assume that $a_1=0$. 
Similarly, we may assume that $a_m=1$. 
Set $\gamma_i:=\gamma|_{[a_{i}, a_{i+1}]}$ for $i=1, \ldots, m-1$.
Then we infer from Corollary~\ref{cor:pinned-Case1} that 
\begin{align}\label{eq:equal-if-half8}
\overline{\mathcal{B}}_p[\gamma_i] = \mathcal{L}[\gamma_i]^{p-1}\mathcal{B}_p[\gamma_i] \geq \varpi^*_p.
\end{align}
Note that $\mathcal{B}_p[\gamma]=\sum_{i=1}^{m-1}\mathcal{B}_p[\gamma_i]$ and $\mathcal{L}[\gamma]=\sum_{i=1}^{m-1}\mathcal{L}[\gamma_i]$.
It follows from \eqref{eq:equal-if-half8} that
\begin{align}\label{eq:sum-estimate}
    \overline{\mathcal{B}}_p[\gamma] 
    %&= \mathcal{L}[\gamma]^{p-1}\mathcal{B}_p[\gamma]
    &=\left(\sum_{i=1}^{m-1}\mathcal{L}[\gamma_i]\right)^{p-1}\sum_{i=1}^{m-1}\mathcal{B}_p[\gamma_i]
    \geq \left(\sum_{i=1}^{m-1}\mathcal{L}[\gamma_i]\right)^{p-1}\sum_{i=1}^{m-1}\frac{\varpi^*_p}{\mathcal{L}[\gamma_i]^{p-1}}.
\end{align}
For the case $p>2$, it follows from H\"older's inequality that
%combining \eqref{eq:sum-estimate} with H\"older's inequality 
\[
\left(\sum_{i=1}^{m-1}\frac{1}{\mathcal{L}[\gamma_i]}\right)^{p-1} \leq (m-1)^{p-2}\sum_{i=1}^{m-1}\left(\frac{1}{\mathcal{L}[\gamma_i]}\right)^{p-1}.
\]
This together with \eqref{eq:sum-estimate} yields
%as $a_i=1/\mathcal{L}[\gamma_i]$, we have
\begin{align}\label{eq:open-LiYau-p>2}
\begin{split}
    \overline{\mathcal{B}}_p[\gamma] 
    &\geq \left(\sum_{i=1}^{m-1}\mathcal{L}[\gamma_i]\right)^{p-1}\varpi^*_p(m-1)^{2-p}\left(\sum_{i=1}^{m-1}\frac{1}{\mathcal{L}[\gamma_i]}\right)^{p-1} \\
    & \geq\varpi^*_p(m-1)^{2-p} (m-1)^{2(p-1)} = \varpi^*_p(m-1)^p,
\end{split}
\end{align}
where we used the HM-AM inequality.
For the case $1<p\leq2$, it follows from H\"older's inequality that
%, for $z_i>0$ ($i=1, \ldots,m-1$), 
\[
\left(\sum_{i=1}^{m-1}\mathcal{L}[\gamma_i]\right)^{p-1} \geq (m-1)^{p-2}\sum_{i=1}^{m-1}\mathcal{L}[\gamma_i]^{p-1}.
\]
Combining this with \eqref{eq:sum-estimate}, we see that
%Using this as $z_i=\mathcal{L}[\gamma_i]$, we infer from \eqref{eq:sum-estimate} that 
\begin{align}\label{eq:open-LiYau-p<2}
\begin{split}
    \overline{\mathcal{B}}_p[\gamma] 
    &\geq (m-1)^{p-2}\left(\sum_{i=1}^{m-1}\mathcal{L}[\gamma_i]^{p-1}\right)\varpi^*_p\left(\sum_{i=1}^{m-1}\frac{1}{\mathcal{L}[\gamma_i]^{p-1}}\right) \\
    & \geq\varpi^*_p(m-1)^{p-2} (m-1)^{2} = \varpi^*_p(m-1)^p,
\end{split}
\end{align}
where we used the HM-AM inequality.
The proof is complete.
\end{proof}

We later show that inequality \eqref{eq:LiYau-Bp} follows from a special consequence of Theorem~\ref{thm:LiYau-Bp-open}.
In order to discuss optimality, we introduce leafed $p$-elasticae as in \cite{Miura_LiYau}. 

%%%%%%%%%%%%%%%%%%%%%%%%%%%%%%%%%%%%%%
\begin{definition}[Leafed $p$-elastica]\label{def:m-leaf}
Let $p\in(1,\infty)$ and $m\in\N$.
\begin{itemize}
    \item[(1)] We call $\gamma\in W^{2,p}_{\rm imm}(0,1;\R^2)$ \emph{$m$-leafed $p$-elastica} if there are $0=a_0<a_1<\cdots< a_m=1$ such that for each $i=1,\ldots,m$ the curve $\gamma_i:=\gamma|_{[a_{i-1},a_i]}$ is a half-fold figure-eight $p$-elastica, and also $\mathcal{L}[\gamma_{1}]=\cdots=\mathcal{L}[\gamma_{m}]$.
    \item[(2)] We call $\gamma\in W^{2,p}_{\rm imm}(\mathbf{T}^1;\R^2)$ \emph{closed $m$-leafed $p$-elastica} if there is $t_0\in\T^1$ such that the curve $\bar{\gamma}:[0,1]\to\R^2$ defined by $\bar{\gamma}(t):=\gamma(t+t_0)$ is an $m$-leafed $p$-elastica.
\end{itemize}
\end{definition}
%%%%%%%%%%%%%%%%%%%%%%%%%%%%%%%%%%%%%%

We mention an obvious example in the following lemma.
The proof is safely omitted since it follows immediately by definition. 

%%%%%%%%%%%%%%%%%%%%%%%%%%%%%%%%%%%%%%
\begin{lemma}[Closed $m$-leafed $p$-elasticae for even $m$]\label{lem:even-leafed}
For any $p\in(1,\infty)$ and any even $m\geq2$, an $\frac{m}{2}$-fold figure-eight $p$-elastica is a closed $m$-leafed $p$-elastica.
\end{lemma}
%%%%%%%%%%%%%%%%%%%%%%%%%%%%%%%%%%%%%%

It is also easy to see that leafed $p$-elasticae indeed attain equality in \eqref{eq:open-counterpart}.

%%%%%%%%%%%%%%%%%%%%%%%%%%%%%%%%%%%%%%
\begin{proposition}[Energy of $m$-leafed $p$-elastica]\label{prop:leaf-energy}
Let $m\in\N$.
Then any $m$-leafed (resp.\ closed $m$-leafed) elastica has a point of multiplicity $m+1$ (resp.\ $m$), and its normalized $p$-bending energy is $\varpi^*_p m^p$.
\end{proposition}
%%%%%%%%%%%%%%%%%%%%%%%%%%%%%%%%%%%%%%
\begin{proof}
Let $\gamma$ be an $m$-leafed $p$-elastica.
Let $L>0$ denote the length of each half-fold figure-eight elastica $\gamma_i$, cf.\ Definition \ref{def:m-leaf}.
Then $\mathcal{L}[\gamma]=mL$.
In addition, since $L^{p-1}\mathcal{B}_p[\gamma_i]=\overline{\mathcal{B}}_p[\gamma_i]=\varpi^*_p$, the additivity of $\mathcal{B}$ yields $\mathcal{B}_p[\gamma]=\sum_{i=1}^m\mathcal{B}_p[\gamma_i]=mL^{1-p}\varpi^*_p$.
Thus we get $\overline{\mathcal{B}}_p[\gamma]=
(mL)^{p-1}(mL^{1-p}\varpi^*_p)=\varpi^*_pm^p$.
\end{proof}

We now prove that for open curves, not only rigidity but also optimality always holds.

%%%%%%%%%%%%%%%%%%%%%%%%%%%%%%%%%%%%%%
\begin{theorem}[Optimality and rigidity for open curves]\label{thm:Open-optimality}
Let $p\in(1,\infty)$ and $m\geq2$.
Then there exists $\gamma\in W^{2,p}_{\rm imm}(0,1; \R^2)$ with a point of multiplicity $m$ such that
\begin{align}\label{eq:optimal-open}
    \overline{\mathcal{B}}_p[\gamma]=\varpi^*_p(m-1)^p.
\end{align}
In addition, equality \eqref{eq:optimal-open} is attained if and only if $\gamma$ is an $(m-1)$-leafed $p$-elastica.
\end{theorem}
%%%%%%%%%%%%%%%%%%%%%%%%%%%%%%%%%%%%%%
\begin{proof}
By definition, an $\frac{m-1}{2}$-fold figure-eight $p$-elastica is an open curve with a point of multiplicity $m$, and attains equality.
From this fact the existence of an optimal curve follows.
In addition, any $(m-1)$-leafed $p$-elastica attains equality \eqref{eq:optimal-open} by Proposition~\ref{prop:leaf-energy}.

We now prove rigidity. 
Suppose that $\gamma$ attains \eqref{eq:optimal-open}. 
As in the proof of Theorem~\ref{thm:LiYau-Bp-open}, there are $0\leq a_1 < \cdots < a_m\leq1$ such that $\gamma(a_1)=\cdots=\gamma(a_m)$. 
Then $a_1$ and $a_m$ must be $0$ and $1$, respectively.
Set $\gamma_i:=\gamma|_{[a_i, a_{i+1}]}$ for $i=1, \ldots, m-1$.
Then equality holds for all the inequalities in the proof of Theorem~\ref{thm:LiYau-Bp-open}, i.e., \eqref{eq:sum-estimate}, \eqref{eq:open-LiYau-p>2}, and \eqref{eq:open-LiYau-p<2}.
Focusing on H\"older's inequality and the HM-AM inequality, we have $\mathcal{L}[\gamma_1]=\cdots=\mathcal{L}[\gamma_{m-1}]$.
Moreover, in view of \eqref{eq:sum-estimate}, we also have $\overline{\mathcal{B}}_p[\gamma_i]=\varpi^*_p$ for all $i$, and hence by Corollary~\ref{cor:pinned-Case1} each curve $\gamma_i$ needs to be a half-fold figure-eight $p$-elastica.
This means that $\gamma$ is an $(m-1)$-leafed $p$-elastica.
\end{proof}

We are in a position to prove Theorems~\ref{thm:LiYau-Bp} and \ref{thm:LiYau-optimal-even}.

\begin{proof}[Proof of Theorem~\ref{thm:LiYau-Bp}]
First, we prove inequality \eqref{eq:LiYau-Bp}.
Let $\gamma\in W^{2,p}_{\rm imm}(\mathbf{T}^1;\R^2)$ be a curve with a point of multiplicity $m\geq2$.
Then we can create an open curve $\bar{\gamma}$ with a point of multiplicity $m+1$ after cutting $\gamma$ at the original point of multiplicity and opening the domain $\mathbf{T}^1$ to $[0,1]$.
Applying Theorem~\ref{thm:LiYau-Bp-open} to the curve $\bar{\gamma}$, we obtain \eqref{eq:LiYau-Bp}. 

We turn to optimality and rigidity.
By Proposition~\ref{prop:leaf-energy}, any closed $m$-leafed $p$-elastica attains equality in \eqref{eq:LiYau-Bp}.
Conversely, suppose that $\gamma \in W^{2,p}_{\rm imm}(\mathbf{T}^1; \R^2)$ attains equality in \eqref{eq:LiYau-Bp}. 
As in the previous procedure, by creating an open curve $\bar{\gamma}$ with a point of multiplicity $m+1$, and applying Theorem~\ref{thm:Open-optimality} to $\bar{\gamma}$, we see that $\gamma$ must be a closed $m$-leafed $p$-elastica.
\end{proof}

\begin{proof}[Proof of Theorem~\ref{thm:LiYau-optimal-even}]
For any even integer $m\geq2$, let $\gamma$ be an $\frac{m}{2}$-fold figure-eight $p$-elastica.
Then $\gamma$ is also a closed $m$-leafed $p$-elastica as mentioned in Lemma~\ref{lem:even-leafed}, and this together with Theorem~\ref{thm:LiYau-Bp} completes the proof.
\end{proof}

\subsection{Closed $m$-leafed $p$-elasticae}\label{subsect:leaf-existence}

Here we seek for exponents $p$ admitting closed leafed $p$-elasticae with odd multiplicities.

We first prepare the following lemma in the same spirit of \cite[Lemma 3.7]{Miura_LiYau}, which states that 
whether there exists a closed $m$-leafed $p$-elastica can be characterized by whether the leaves can be joined up to first order.
By Proposition \ref{prop:half-fold} (iii), the angle $2\phi^*(p)$ plays a fundamental role.

%%%%%%%%%%%%%%%%%%%%%%%%%%%%%%%%%%%%%%
\begin{lemma}[Characterization of closed $m$-leafed $p$-elasticae]\label{lem:character}
Let $m\geq2$ be an integer and $p\in(1,\infty)$. 
Let $\Omega^*(m,p)$ be the set of all $m$-tuples $(\omega_1, \ldots, \omega_m)$ of unit-vectors $\omega_1, \ldots, \omega_m \in \mathbf{S}^1 \subset \R^2$ such that $\langle \omega_i, \omega_{i-1} \rangle=\cos{2\phi^*(p)}$ holds for any $i=1, \ldots, m$, where we interpret $\omega_0=\omega_m$. 

If $\gamma:\T^1\to\R^2$ is a unit-speed closed $m$-leafed $p$-elastica, then $(\omega_1, \ldots, \omega_m)\in\Omega^*(m,p)$ holds for  $\omega_i:=\gamma'(\frac{i}{m}+t_0)$ and $i=1, \ldots, m$, where $t_0\in\mathbf{T}^1$ is a point of multiplicity $m$.

Conversely, for any element $(\omega_1, \ldots, \omega_m)\in\Omega^*(m,p)$, there exists a unique unit-speed closed $m$-leafed $p$-elastica $\gamma:\T^1\to\R^2$ such that $\gamma(0)=0$ and $\gamma'(\frac{i}{m})=\omega_i$ for $i=1, \ldots, m$.
\end{lemma}
%%%%%%%%%%%%%%%%%%%%%%%%%%%%%%%%%%%%%%

The proof of Lemma~\ref{lem:character} is straightforward and safely omitted.
Instead, we discuss a typical example: 
Let $\gamma$ be a closed $\frac{m}{2}$-fold figure-eight $p$-elastica.
Up to similarity and reparameterization we may assume that $\gamma$ is of unit-speed and $\gamma(0)=0$ is a point of multiplicity $m$.
Then setting $\omega_i:=\gamma'(\frac{i}{m})$ for $i=1,\ldots, m$, we see that $\omega_i=\gamma'(0)$ if $i$ is even, and (up to reflection) $\omega_i=R_{2\phi^*}\gamma'(0)$ if $i$ is odd.
Hence, clearly $(\omega_1, \ldots, \omega_m)\in\Omega^*(m,p)$.

Now we discuss the existence issue for special exponents.
For an odd integer $m\geq3$, and $n\in\N$ with $n< {m}/{2}$ so that $\frac{n}{m}\pi\in(0,\frac{\pi}{2})$, we define $p_{m,n}\in(1,\infty)$ by
\begin{align}\label{eq:def-p_mn}
    p_{m,n} := \left(\phi^*\right)^{-1}(\tfrac{n}{m}\pi).
\end{align}
Such an exponent $p_{m,n}$ uniquely exists by Theorem~\ref{thm:phi*-decrease}.
If $n=1$ we also write
\begin{equation}\label{eq:def-p_m}
    p_m:=p_{m,1}=\left(\phi^*\right)^{-1}(\tfrac{\pi}{m}).
\end{equation}
\if0
In view of \eqref{eq:phi*}, $p_{m,n}$ is also characterized as a solution $p_{m,n}$ to
\begin{align}\label{eq:odd-leaf-cond}
    Q_p\Big( \sin {\big(\tfrac{m-n}{2m}\pi\big)} \Big)=0.
\end{align}
\fi
For this $p_{m,n}$ we will create closed leafed elasticae with odd multiplicities.

%%%%%%%%%%%%%%%%%%%%%%%%%%%%%%%%%%%%%%
To gain insight, we first discuss the special exponent $p_3=p_{3,1}$ for which we can construct closed leafed elasticae even for all odd multiplicities.

\begin{example}[Closed $m$-leafed $p_3$-elasticae]\label{exam:3-leafed}
Let $\gamma$ be a half-fold figure-eight $p_3$-elastica. 
Then, by $2\phi^*(p_3)=2\pi/3$, cf.\ \eqref{eq:def-p_m}, and by Proposition \ref{prop:half-fold} (iii), (up to reflection) the concatenation $\gamma\oplus R_{-2\pi/3}\gamma\oplus R_{-4\pi/3}\gamma$ is a closed $3$-leafed $p_3$-elastica, unique up to invariances.
For a general $m=3+2\ell$ with $\ell\in \N$, we can construct a closed $m$-leafed $p_3$-elastica by concatenating a closed $3$-leafed $p_3$-elastica and an $\ell$-fold figure-eight $p_3$-elastica.
%(See Figure~\ref{fig:5-leaf}.)
\end{example}
%%%%%%%%%%%%%%%%%%%%%%%%%%%%%%%%%%%%%%

This kind of construction extends to a general multiplicity:

%%%%%%%%%%%%%%%%%%%%%%%%%%%%%%%%%%%%%%
\begin{proposition}[Closed $m$-leafed $p$-elasticae for odd $m$]\label{prop:odd-leafed}
Let $m$ and $n$ be integers such that $m\geq3$ is odd and $1\leq n< m/2$, and let $p_{m,n}$ be defined by \eqref{eq:def-p_mn}. 
Let $m'$ be an odd integer. 
If $m'\geq m$, then there exists a closed $m'$-leafed $p_{m,n}$-elastica.
In addition, in the case of $n=1$, there exists a closed $m'$-leafed $p_{m}$-elastica if and only if $m'\geq m$. 
\end{proposition}
%%%%%%%%%%%%%%%%%%%%%%%%%%%%%%%%%%%%%%
\begin{proof}
First, we show existence of a closed $m$-leafed $p_{m,n}$-elastica. 
For simplicity, hereafter we write $\phi^*:=\phi^*(p_{m,n})=\frac{n}{m}\pi$, cf.\ \eqref{eq:def-p_mn}. 
Fix any $\omega_1 \in \mathbf{S}^1$, and inductively define $\omega_i:=R_{2\phi^*}\omega_{i-1}$ for $i=2, \ldots, m$.
Then $\omega_m=R_{2(m-1)\phi^*}\omega_1=R_{2n\pi-2\phi^*}\omega_1=R_{-2\phi^*}\omega_1$ and hence the $m$-tuple $(\omega_1, \ldots, \omega_m)$ is an element of $\Omega^*(m,p_m)$.
Therefore, Lemma~\ref{lem:character} ensures existence of a closed $m$-leafed $p_m$-elastica (corresponding to the case of $m'=m$).

Next let $m'$ be an odd integer with $m'> m$. 
Fix any $\omega_1 \in \mathbf{S}^1$, and inductively define $\omega_i:=R_{2\phi^*}\omega_{i-1}$ for $i=2, \ldots, m$.
For $m<i\leq m'$, we define $\omega_i:=\omega_1$ if $i$ is even and $\omega_i:=R_{2\phi^*}\omega_1$ if $i$ is odd.
Then we see that $(\omega_1, \ldots, \omega_{m'}) \in \Omega^*(m',p_{m,n})$, and hence thanks to Lemma~\ref{lem:character} there exists a closed $m'$-leafed $p_{m,n}$-elastica.

In the rest of the proof, provided that $n=1$ so that $p_{m,1}=p_m$, we show nonexistence of closed $m'$-leafed $p_{m}$-elasticae if $m'<m$.
We argue by contradiction. 
Suppose that there exists a closed $m'$-leafed $p_m$-elastica for an odd integer $m'<m$.
Then, by Lemma~\ref{lem:character}, there is an $m'$-tuple $(\omega_1, \ldots, \omega_{m'}) \in \Omega^*(m',p_m)$ of  unit-vectors $\omega_1, \ldots, \omega_{m'}\in \mathbf{S}^1$.
Here we note that $\langle\omega, \omega'\rangle=\cos{2\phi^*}$ holds for $\omega, \omega'\in \mathbf{S}^1$ if and only if $\omega=R_{2\phi^*}\omega'$ or $\omega=R_{-2\phi^*}\omega'$.
From this fact there is an $m'$-tuple of rotation matrices $A_1, \cdots, A_{m'}$ such that $A_i$ is either $R_{2\phi^*}$ or $R_{-2\phi^*}$ and $A_1\cdots A_{m'}=I$, where $I$ denotes the identity matrices.
In particular, there is a sequence $\sigma_1, \ldots,\sigma_{m'} \in \{-1,1\}$ such that $\sum_{i=1}^{m'}\sigma_i2\phi^*\in2\pi\Z$.
However, we infer from \eqref{eq:def-p_m} and $m'<m$ that $|\sum_{i=1}^{m'}\sigma_i2\phi^*|\leq 2m'\phi^*<2m\phi^*=2\pi$, and also infer from  oddness of $m'$ that $|\sum_{i=1}^{m'}\sigma_i2\phi^*|>2\phi^*>0$.
This contradicts $\sum_{i=1}^{m'}\sigma_i2\phi^*\in2\pi\Z$, and the proof is complete.
\end{proof}

In general, for a given integer $m$, there may be multiple exponents $p$ that admit $m$-leafed $p$-elasticae.
For example, there exist $5$-leafed $p$-elasticae if $p$ is $p_3$ (cf.\ Remark~\ref{exam:3-leafed}), $p_5$, or $p_{5,2}$, see Figure~\ref{fig:5-leaf}. 
 \begin{figure}[htbp]
      \includegraphics[scale=0.2]{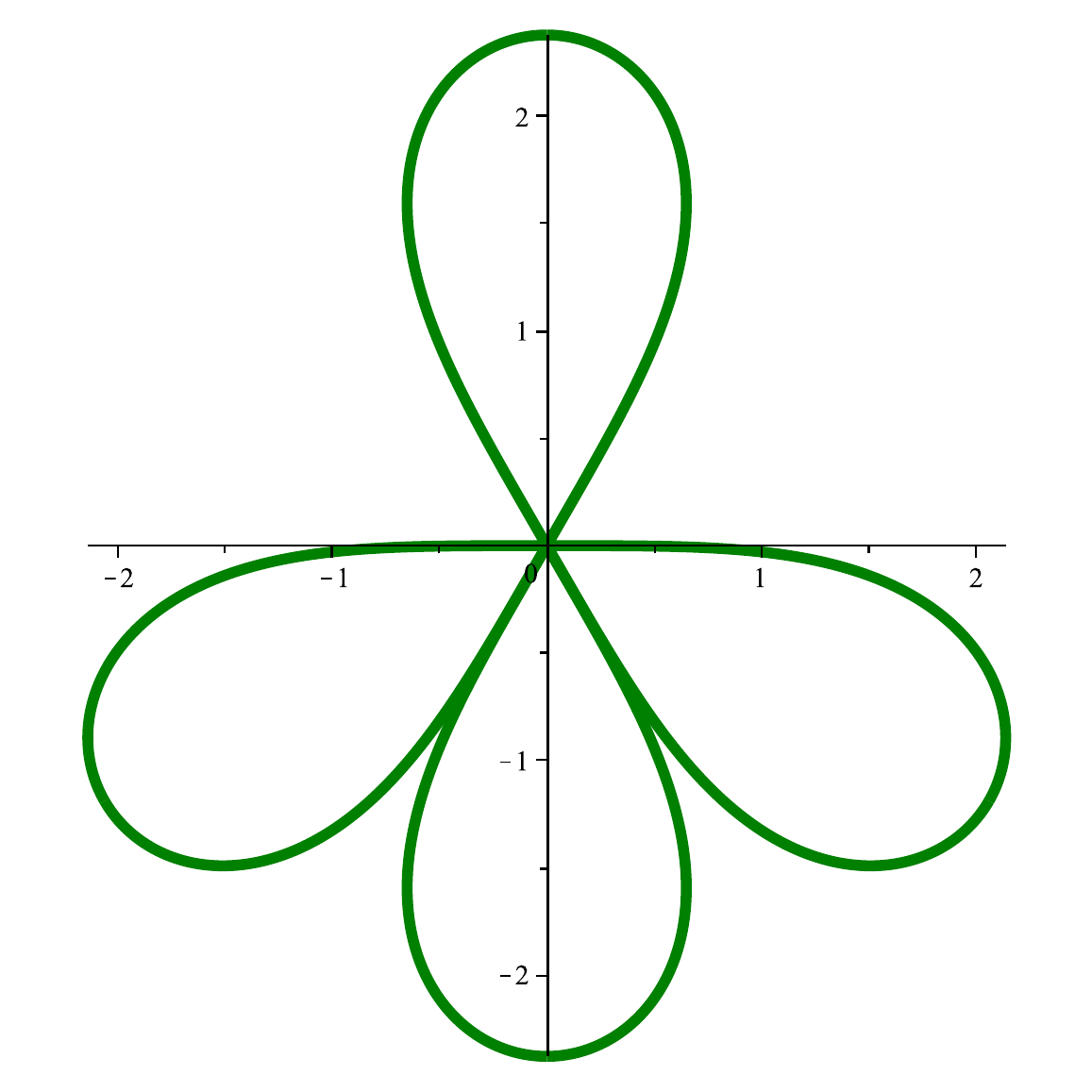}
    \hspace{5pt}
      \includegraphics[scale=0.2]{5-leafed.pdf}
      \hspace{5pt}
      \includegraphics[scale=0.2]{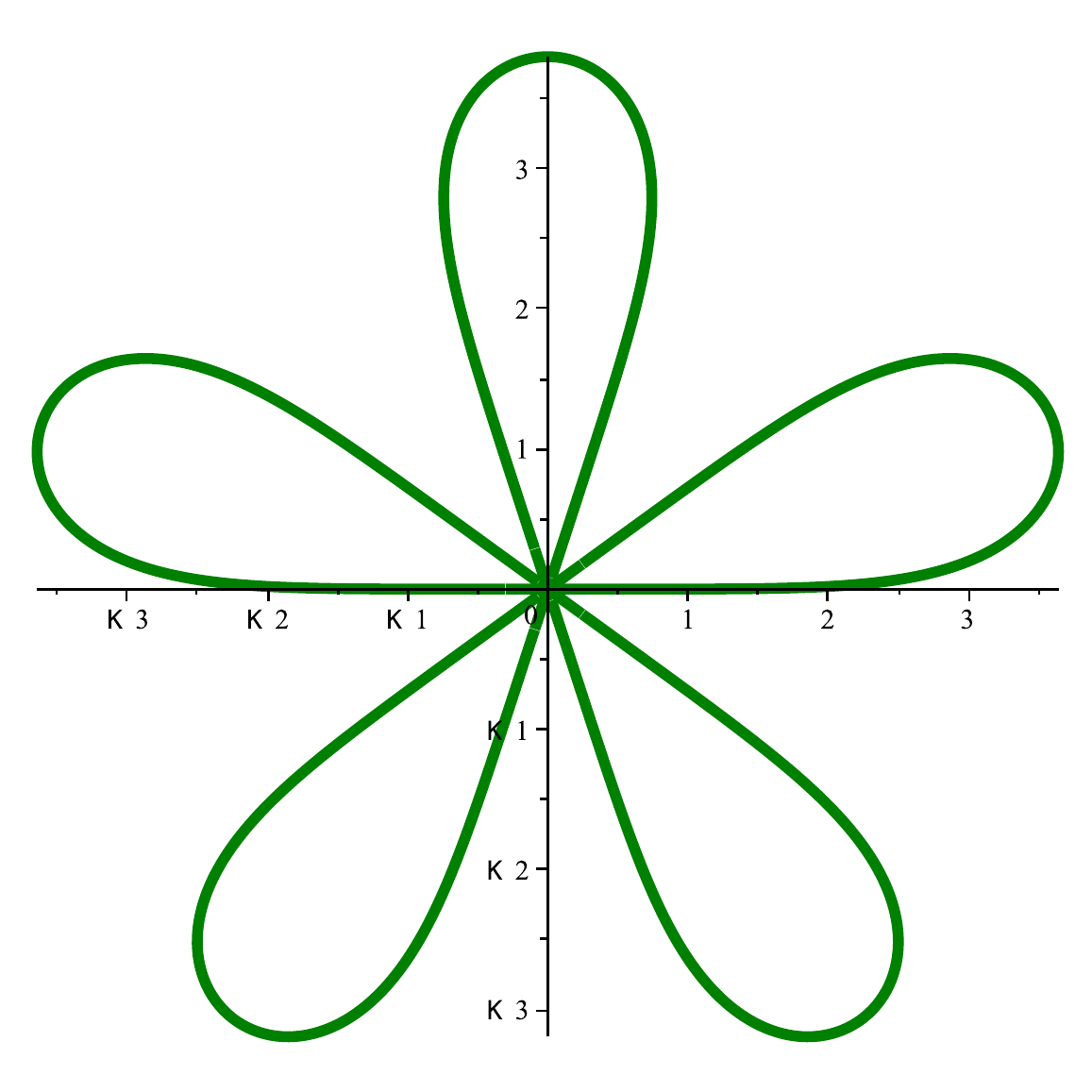}
  \caption{$5$-leafed figure-eight $p$-elasticae with $p=p_3$ (left), $p=p_5$ (middle), and $p=p_{5,2}$ (right).
  In the left curve the top leaf is covered twice.}
  \label{fig:5-leaf}
  \end{figure}
  
However, the exponent admitting a $3$-leafed $p$-elastica is uniquely determined:
%%%%%%%%%%%%%%%%%%%%%%%%%%%%%%%%%%%%%%
\begin{proposition}\label{prop:3-leafed}
Let $p\in(1,\infty)$. 
There exists a closed $3$-leafed $p$-elastica if and only if $p=p_3$, where $p_3:=(\phi^*)^{-1}(\frac{\pi}{3})$. 
\end{proposition}
%%%%%%%%%%%%%%%%%%%%%%%%%%%%%%%%%%%%%%
\begin{proof}
Existence of closed $3$-leafed $p_3$-elasticae follows from Proposition~\ref{prop:odd-leafed}.
Fix any $p\in(1,\infty)$ and assume that a closed $3$-leafed $p$-elastica exists. 
Then, by Lemma~\ref{lem:character}, there are $\omega_1, \omega_2, \omega_3 \in \mathbf{S}^1$ such that $\langle \omega_1, \omega_2 \rangle=\langle \omega_2, \omega_3 \rangle = \langle \omega_3, \omega_1 \rangle= \cos 2\phi^*(p)$. 
Since $0<2\phi^*(p)<\pi$ holds by Proposition~\ref{prop:phi*-limit}, $\phi^*(p)$ must be $\pi/3$. 
This together with Theorem~\ref{thm:phi*-decrease} implies that $p=p_3$.
The proof is complete.
\end{proof}

\subsection{Optimality for odd multiplicity}\label{subsect:optimal}

Now we translate the phenomena observed above into optimality of our Li--Yau type multiplicity inequality.

%%%%%%%%%%%%%%%%%%%%%%%%%%%%%%%%%%%%%%
\begin{theorem}[Optimality for $p_{m,n}$]\label{thm:LiYau-optimal-odd}
Let $S\subset(1,\infty)$ be the set of all exponents $p_{m,n}\in(1,\infty)$ with integers $m$ and $n$ such that $m\geq3$ is odd and $0<n<m/2$, cf.\ \eqref{eq:def-p_mn}.
Then $S$ is dense in $(1,\infty)$. 
In addition, for any $p=p_{m,n}\in S$, and for any odd integer $m'\geq m$, there exists a curve $\gamma \in W^{2,p}_{\rm imm}(\mathbf{T}^1;\R^2)$ with a point of multiplicity $m'$ such that $\overline{\mathcal{B}}_{p}[\gamma] = \varpi^*_p (m')^{p}$.
\end{theorem}
%%%%%%%%%%%%%%%%%%%%%%%%%%%%%%%%%%%%%%
\begin{proof}
Note that the set $R\subset(0,\frac{\pi}{2})$ of all angles $\frac{n}{m}\pi\in(0, \tfrac{\pi}{2})$ with $m,n$ under consideration is dense in $(0, \frac{\pi}{2})$.
Since $S$ is the preimage of $R$ under the continuous and bijective map $\phi^*$ (cf.\ Theorem~\ref{thm:phi*-decrease}), the set $S$ is also dense in $(1,\infty)$.

For optimality, in view of Theorem~\ref{thm:LiYau-Bp}, it suffices to check existence of $m'$-leafed $p_{m,n}$-elasticae,
and this follows from Proposition~\ref{prop:odd-leafed}.
\end{proof}

In particular, if $p=p_m$, then we can obtain a complete classification of multiplicity which ensures optimality.

%%%%%%%%%%%%%%%%%%%%%%%%%%%%%%%%%%%%%%
\begin{theorem}[Thresholding phenomenon]\label{thm:thresholding}
Let $m\geq3$ be an odd integer and $p_m\in(1,\infty)$ be defined by \eqref{eq:def-p_m}.
Let $m'\geq2$ be an integer. 
Then there exists a curve $\gamma\in W^{2,p_m}_{\rm imm}(\mathbf{T}^1;\R^2)$ with a point of multiplicity $m'$ such that $\overline{\mathcal{B}}_{p_m}[\gamma] = \varpi^*_{p_m} (m')^{p_m}$ if and only if $m'$ is even or $m'\geq m$.
\end{theorem}
%%%%%%%%%%%%%%%%%%%%%%%%%%%%%%%%%%%%%%
\begin{proof}
The existence of closed $m'$-leafed $p_m$-elasticae for even $m'$ and for odd $m'\geq m$ follows from Lemma~\ref{lem:even-leafed} and Proposition~\ref{prop:odd-leafed}, respectively, while Proposition~\ref{prop:odd-leafed} also implies that there is no closed $m'$-leafed $p_m$-elastica if $m'<m$ and $m'$ is odd.
This fact with Theorem~\ref{thm:LiYau-Bp} completes the proof.
\end{proof}

We close this section by completing the proof of Theorem~\ref{thm:LiYau-full-optimal}, i.e., showing that an exponent which can recover optimality for \emph{every} multiplicity is unique.

\begin{proof}[Proof of Theorem~\ref{thm:LiYau-full-optimal}]
If an integer $m\geq2$ is even (resp.\ odd), then by Lemma~\ref{lem:even-leafed} (resp.\ Proposition~\ref{prop:odd-leafed}) there exists a closed $m$-leafed $p_3$-elastica. 
Hence the assertion holds for $p=p_3$.
We turn to uniqueness.
Assume that an exponent $p\in(1,\infty)$ satisfies the assertion of the theorem.
Then, by the rigidity part of Theorem~\ref{thm:LiYau-Bp}, there exists a closed $m$-leafed $p$-elastica for any odd $m\geq3$, in particular for $m=3$.
Therefore, Proposition~\ref{prop:3-leafed} implies $p=p_3$. 
\end{proof}

\begin{proof}[Proof of Corollary~\ref{cor:inf-network-sym}]
This immediately follows from Theorem~\ref{thm:LiYau-full-optimal} with the choice of $\alpha=\frac{2}{3}\pi$, combined with Theorem~\ref{thm:phi*-decrease} and the fact that $\phi^*(p_3)=\frac{\pi}{3}$.
\end{proof}

\section{Existence of minimal $p$-elastic networks}\label{sect:network}

In this final section we prove Theorem~\ref{thm:inf-network-alpha}.
Note that, up to rescaling, our problem \eqref{eq:inf-network} is equivalent to minimizing $\mathcal{B}_p+\mathcal{L}$ (as is done in \cite{DNP20,Miura_LiYau}). 
Indeed, for $\Lambda>0$ and a curve $\gamma$, if we set $\gamma_{\Lambda}(t):=\Lambda\gamma(t)$, then $\mathcal{L}[\gamma_{\Lambda}]=\Lambda \mathcal{L}[\gamma]$ and $\mathcal{B}_p[\gamma_\Lambda]=\Lambda^{1-p}\mathcal{B}_p[\gamma]$.
Hence, with the choice of $\Lambda=\mathcal{B}_p[\gamma]^{\frac{1}{p}}\mathcal{L}[\gamma]^{-\frac{1}{p}}$, we find that
\begin{align}\label{eq:funct-equivalent}
    %\mathcal{E}_p[\Gamma_\Lambda] =
    \mathcal{B}_p[\gamma_\Lambda]+\mathcal{L}[\gamma_\Lambda] = 2\left(\mathcal{L}[\gamma]^{p-1}\mathcal{B}_p[\gamma] \right)^{\frac{1}{p}} = 2 \overline{\mathcal{B}}_p[\gamma]^{\frac{1}{p}}.
\end{align}
Clearly, the same is true for $\gamma$ replaced by a network $\Gamma\in \Theta(p,\alpha)$, and the set $\Theta(p,\alpha)$ is closed under rescaling, so that the desired equivalency holds.

The main concern here is the lack of compactness of the set $\Theta(p,\alpha)$ (cf.\ \cite{DNP20, DPP21, Miura_LiYau} for the case $p=2$). 
More precisely, the length of one component-curve of a minimizing sequence may vanish; in other words, a family of networks composed by three curves may converge to ``two-component network'' (cf.\ Lemma~\ref{lem:dichotomy}).
In order to prevent such a phenomenon, in this paper, we construct a certain  $\Theta$-network whose energy is less than the minimal energy among the class of ``two-component network''.
The assumption \eqref{eq:angle-network} is needed to construct such a $\Theta$-network which consists of exactly three curves (cf.\ Lemma~\ref{lem:test-network}).

We will use Fenchel's theorem for piecewise $W^{2,1}$ closed curves obtained in \cite[Theorem A.1]{DNP20} or \cite[Lemma 5.2]{Miura_LiYau}.
%%%%%%%%%%%%%%%%%%%%%%%%%%%%%%%%%%%%%%
\begin{lemma}[{\cite[Theorem A.1]{DNP20}, \cite[Lemma 5.2]{Miura_LiYau}}] \label{lem:M_LY-Lem5.2}
Let $\gamma_1, \ldots, \gamma_N \in W^{2,1}(0,1;\R^2)$ be immersed curves such that $\gamma_j(1)=\gamma_{j+1}(0)=:V_j$ for $j=1,\ldots, N$, where we interpret $\gamma_{N+1}=\gamma_1$. 
For all $j=1,\ldots, N$, let $\theta_j \in [0,\pi]$ denote the external angle at the vertex $V_j$, i.e., 
\[
\cos{\theta_j}=\left\langle \frac{\gamma_j'(1)}{|\gamma_j'(1)|}, \frac{\gamma_{j+1}'(0)}{|\gamma_{j+1}'(0)|} \right\rangle. 
\]
Then 
\[
\sum_{j=1}^N TC[\gamma_j]:=\sum_{j=1}^N \int_{\gamma_j}|k|\,ds \geq 2\pi-\sum_{j=1}^N\theta_j.
\]
\end{lemma}
%%%%%%%%%%%%%%%%%%%%%%%%%%%%%%%%%%%%%%

Using this lemma, we prove the following key dichotomy result as in \cite{DNP20, Miura_LiYau}.

%%%%%%%%%%%%%%%%%%%%%%%%%%%%%%%%%%%%%%
\begin{lemma}\label{lem:dichotomy}
Let $\{\Gamma_j\} = \{ (\gamma_{1,j}, \gamma_{2,j}, \gamma_{3,j}) \} \subset \Theta(p,\alpha)$ be a sequence such that $\sup_{j}
\overline{\mathcal{B}}_p[\Gamma_j] < \infty$. 
Then, up to reparameterization, translation, dilation, and taking a subsequence (all without relabeling), either of the following two assertions holds: 
\begin{itemize}
\item[(i)] There exists a $\Theta$-network $\Gamma=(\gamma_1, \gamma_2, \gamma_3)\in \Theta(p,\alpha)$ such that for each $i=1,2,3$, the sequence $\{\gamma_{i, j}\}_j$ converges to $\gamma_i$ as $j\to\infty$ in the $W^{2,p}$-weak and $C^1$ topology. 
In particular, $\liminf_{j\to \infty} \overline{\mathcal{B}}_p[\Gamma_j] \geq \overline{\mathcal{B}}_p[\Gamma]$.
\item[(ii)]  Up to permutations of the index $i$, $\mathcal{L}[\gamma_{3,j}]\to 0$ as $j\to\infty$. 
In addition, there are $\gamma_1, \gamma_2 \in W^{2,p}_{\rm imm}(0,1; \R^2)$ such that $\gamma_1(0)=\gamma_1(1)=\gamma_2(0)=\gamma_2(1)$ and such that the sequence $\{\gamma_{1,j}\}_j$ (resp.\ $\{\gamma_{2,j}\}_j$) converges to $\gamma_1$ (resp.\ $\gamma_2$) as $j\to \infty$ in the $W^{2,p}$-weak and $C^1$ topology.
In particular, 
\begin{align}\label{eq:degenerate-lsc}
    \liminf_{j\to\infty}\left(\sum_{i=1}^2\mathcal{L}[\gamma_{i,j}]\right)^{p-1}\left(\sum_{i=1}^2\mathcal{B}_p[\gamma_{i,j}]\right) 
    \geq \left(\sum_{i=1}^2\mathcal{L}[\gamma_{i}]\right)^{p-1}\left(\sum_{i=1}^2\mathcal{B}_p[\gamma_{i}]\right)  .
\end{align}
\end{itemize}
\end{lemma}
%%%%%%%%%%%%%%%%%%%%%%%%%%%%%%%%%%%%%%
\begin{proof}
Throughout the proof we may suppose that, after translation, $\gamma_{i,j}(0)=0$, 
and also, after rescaling as in \eqref{eq:funct-equivalent} (i.e., replacing $\Gamma_j$ with $\mathcal{B}_p[\Gamma_j]^{\frac{1}{p}} \mathcal{L}[\Gamma_j]^{-\frac{1}{p}}  \Gamma_j$),
\begin{align}\label{eq:Gamma-LBestimate}
\sup_j \left(\mathcal{B}_p[\Gamma_j] + \mathcal{L}[\Gamma_j] \right) = 2\sup_j \overline{\mathcal{B}}_p[\Gamma_j]^{\frac{1}{p}} <\infty. 
\end{align}
In addition, we may suppose that after reparameterization, each curve $\gamma_{i,j}$ is of constant speed.
Furthermore, all of these transformations do not change $\overline{\mathcal{B}}_p$.
In what follows we adopt ideas from \cite{DNP20, Miura_LiYau}.

We first consider the case where $\inf_j \min_{i=1,2,3}\mathcal{L}[\gamma_{i,j}]>0$. 
Then, as discussed in \eqref{eq:normalized-B_p}, the assumption of constant-speed implies that 
\[
\sum_{i=1,2,3} \| \gamma_{i,j}'' \|_{L^p}^p = \sum_{i=1,2,3} \mathcal{L}[\gamma_{i,j}]^{2p-1}\mathcal{B}_p[\gamma_{i,j}].
\]
This together with \eqref{eq:Gamma-LBestimate} yields the $L^p$-boundedness of $\{\gamma_{i,j}''\}_j$ for each $i=1,2,3$.
By the assumption of the constant-speed and $\gamma_{i,j}(0)=0$, the $W^{1,\infty}$-boundedness of $\{\gamma_{i,j}\}_j$ also follows.
Therefore, for each $i=1,2,3$, there exists $\gamma_i\in W^{2,p}(0,1;\R^2)$ such that, up to a subsequence, the sequence $\{\gamma_{i,j}\}_j$ converges to $\gamma_i$ in the $W^{2,p}$-weak and $C^1$ topology.
In particular, the $C^1$-convergence and the assumption that $\inf_j \min_{i=1,2,3}\mathcal{L}[\gamma_{i,j}]>0$ ensure that $\Gamma:=(\gamma_1, \gamma_2, \gamma_3)$ satisfies the angle condition to be a $\Theta$-network with angles $(\alpha, \alpha, 2\pi-2\alpha)$.
Since $\mathcal{B}_p[\gamma_{i,j}] =\mathcal{L}[\gamma_{i,j}]^{1-2p} \|\gamma_{i,j}''\|_{L^p}^p$ (cf.\ \eqref{eq:normalized-B_p}), since $\mathcal{L}[\gamma_{i,j}]$ converges to a positive value, and since the weak lower semicontinuity holds for $\|\gamma_{i,j}''\|_{L^p}^p$, we deduce that $\liminf_{j\to \infty} \overline{\mathcal{B}}_p[\Gamma_j] \geq \overline{\mathcal{B}}_p[\Gamma]$.

Next, we consider the case that, after permutations, $\inf_{j}\mathcal{L}[\gamma_{3,j}]=0$ holds but we still have $\inf_{j}\min_{i=1,2}\mathcal{L}[\gamma_{i,j}]>0$. 
In this case, the same argument as above ensures the desired convergence of $\gamma_{1,j}$ and $\gamma_{2,j}$ so that the assertion (ii) holds. 

We finally prove that only the above cases are possible to occur.
For each pair of $i, i'\in\{1,2,3\}$ with $i\neq i'$, the angle condition $(\alpha, \alpha, 2\pi-2\alpha)$ implies that the curves $\gamma_{i}$ and $\gamma_{i'}$ form a piecewise closed curve with exactly two jumps of angle $\theta_1=\theta_2=\pi-\alpha$ or $\theta_1=\theta_2=2\alpha-\pi$, and hence by Lemma~\ref{lem:M_LY-Lem5.2} we have $TC[\gamma_{i,j}]+TC[\gamma_{i',j}] \geq \min\{2\alpha, 4\pi-4\alpha \}=:c_\alpha>0$.
Noting that $TC[\gamma_{i,j}]\leq \mathcal{B}_p[\gamma_{i,j}]^{\frac{1}{p}}\mathcal{L}[\gamma_{i,j}]^{\frac{p-1}{p}}$
follows from H\"older's inequality, we obtain 
\[
\sum_{l=i, i'} \mathcal{B}_p[\gamma_{l,j}]^{\frac{1}{p}} 
\geq \sum_{l=i, i'} \mathcal{L}[\gamma_{l,j}]^{ -\frac{p-1}{p}}TC[\gamma_{l,j}] 
\geq \min_{l=i,i'}\mathcal{L}[\gamma_{l,j}]^{ -\frac{p-1}{p}} c_\alpha.
\]
Then energy-boundedness \eqref{eq:Gamma-LBestimate} implies that $\inf_{j}\max_{l=i,i'}\mathcal{L}[\gamma_{l,j}]>0$.
By the arbitrariness of the choice of $i$ and $i'$, up to taking a subsequence, there are at least two indices $i\in\{1,2,3\}$ such that $\inf_{j}\mathcal{L}[\gamma_{i,j}]>0$.
\end{proof}

Next, applying our new Li--Yau type inequality, we give a lower bound on $\overline{\mathcal{B}}_p$ when the assertion (ii) holds in Lemma~\ref{lem:dichotomy}.
%%%%%%%%%%%%%%%%%%%%%%%%%%%%%%%%%%%%%%
\begin{lemma}\label{lem:nise-LiYau}
Let $\gamma_1, \gamma_2\in W^{2,p}_{\rm imm}(0,1;\R^2)$ satisfy $\gamma_1(0)=\gamma_1(1)=\gamma_2(0)=\gamma_2(1)$. 
Then
\begin{equation}\label{eq:0902-1}
   \left(\sum_{i=1}^2\mathcal{L}[\gamma_{i}]\right)^{p-1}\left(\sum_{i=1}^2\mathcal{B}_p[\gamma_{i}]\right) 
\geq 2^p \varpi_p^*, 
\end{equation}
where equality is attained if and only if the curves $\gamma_1, \gamma_2$ are half-fold figure-eight $p$-elasticae of same length.
\end{lemma}
%%%%%%%%%%%%%%%%%%%%%%%%%%%%%%%%%%%%%%
\begin{proof}
Up to a rigid motion for $\gamma_2$, we may assume that $\gamma_1'(1)$ and $\gamma_2'(0)$ are parallel.
Then up to reparametrization we may regard the concatenation $\gamma$ of $\gamma_1$ and $\gamma_2$ as a single $W^{2,p}$ open curve with a point of multiplicity $3$, whose normalized $p$-bending energy is equal to the right-hand side of \eqref{eq:0902-1}.
The assertion follows by applying
Theorem~\ref{thm:LiYau-Bp-open} to $\gamma$.
\end{proof}

Now the main issue is reduced to showing that there exists a $\Theta$-network of less energy than the minimal energy $2^p\varpi_p^*$ of degenerate networks. 
More precisely, we prove

%%%%%%%%%%%%%%%%%%%%%%%%%%%%%%%%%%%%%%
\begin{lemma}\label{lem:test-network}
Suppose \eqref{eq:angle-network} holds. 
Then there exists a $\Theta$-network $\Gamma\in \Theta(p,\alpha)$ such that $\overline{\mathcal{B}}_p[\Gamma] < 2^p\varpi_p^*$.
\end{lemma}
%%%%%%%%%%%%%%%%%%%%%%%%%%%%%%%%%%%%%%
\begin{proof}
Let $\gamma_{\rm half}^{p,q}$ be a half-period of a wavelike $p$-elastica given by $\gamma_{w}(s,q)$ in \eqref{eq:EP2} with $s\in[-\K_{1,p}(q), \K_{1,p}(q)]$.
Then it follows that
\begin{align}\label{eq:L-gamma_half}
    \mathcal{L}[\gamma_{\rm half}^{p,q}]=2\K_{1,p}(q).
\end{align}
In addition, since the curvature $k$ of $\gamma_{\rm half}^{p,q}$ is $k(s)=2q\cn_p(s,q)$, we infer from Lemma~\ref{lem:int-cn_p} that 
\begin{align}\label{eq:B_p-gamma_half}
\begin{split}
    \mathcal{B}_p[\gamma_{\rm half}^{p,q}]&=\int_{-\K_{1,p}(q)}^{\K_{1,p}(q)}2^pq^p|\cn_p(s,q)|^p\,ds\\
    &=2^{p+1}q^p\left( \tfrac{1}{q^2}\E_{1,p}(q)+\big(1-\tfrac{1}{q^2}\big) \K_{1,p}(q)\right)=2^{p+1}b(q),
\end{split}
\end{align}
where $b(q)$ is given by \eqref{eq:def-varphi}.

We define (up to reparameterization) a triplet of curves $\Gamma^{p,q}_{\rm wave}$ for $q\in(0,1)$ by 
\[
\Gamma^{p,q}_{\rm wave}:= \left(\gamma_{\rm half}^{p,q}, \gamma_{\rm seg}^{p,q},  R\gamma_{\rm half}^{p,q} \right), 
\]
where $R$ denotes the refection with respect to the $e^1$-axis and $\gamma_{\rm seg}^{p,q}(x):=\big((2\E_{1,p}(q)-\K_{1,p}(q))x,0\big){}^\top$ for $x\in(-1,1)$.
It is clear that $\mathcal{L}[\gamma_{\rm seg}^{p,q}]=2(2\E_{1,p}(q)-\K_{1,p}(q))$ and $\mathcal{B}_p[\gamma_{\rm seg}^{p,q}]=0$.
Combining this with \eqref{eq:L-gamma_half} and \eqref{eq:B_p-gamma_half}, we see that
\begin{align*}
    \mathcal{L}[\Gamma^{p,q}_{\rm wave}]
    &=2\mathcal{L}[\gamma_{\rm half}^{p,q}]+\mathcal{L}[\gamma_{\rm seg}^{p,q}]
    =2\big(2\E_{1,p}(q)+\K_{1,p}(q)\big), \\
    \mathcal{B}_p[\Gamma^{p,q}_{\rm wave}]
    &=2\mathcal{B}_p[\gamma_{\rm half}^{p,q}]+\mathcal{B}_p[\gamma_{\rm seg}^{p,q}]
    =2^{p+2}b(q),
\end{align*}
which yields
\begin{align}\label{eq:Bp-test-network}
   \overline{\mathcal{B}}_p[\Gamma^{p,q}_{\rm wave}]=2^{2p+1}\big(2\E_{1,p}(q)+\K_{1,p}(q)\big)^{p-1}b(q).
\end{align}
In particular, by this representation, and by definition of $\varpi_p^*$ and $q^*(p)$, we have 
\begin{align}\notag%\label{eq:Bp-test-limit}
    \lim_{q\to q^*(p)}\overline{\mathcal{B}}_p[\Gamma^{p,q}_{\rm wave}]
    %= 2^{2p+1}\big(4\E_{1,p}(q^*)\big)2^{p-1}b(q^*) 
    = 2^p\varpi_p^*.
\end{align}

Let $q=q_\alpha \in(0,1)$ be a constant satisfying $2\arcsin{q}=\alpha$.
We now prove that $\Gamma^{p,q_\alpha}_{\rm wave}$ gives the desired $\Theta$-network with angles $(\alpha, \alpha, 2\pi-2\alpha)$. 

First we check that $q_\alpha < q^*(p)$. 
By \eqref{eq:phi*} and assumption \eqref{eq:angle-network}, we see that 
\[
2\arcsin{q_\alpha}=\alpha<\pi-\phi^*(p) = 2\arcsin{q^*(p)},
\]
which implies that $q_\alpha < q^*(p)$.

Next, we show that $\Gamma^{p,q_\alpha}_{\rm wave} \in \Theta(p,\alpha)$.
In view of Case II in Proposition~\ref{prop:MY2203-thm1.1}, the tangential angles of $\gamma_{\rm half}^{p,q}$ at two endpoints are $-2\arcsin{q}$ and $2\arcsin{q}$.
This together with $2\arcsin{q_\alpha}=\alpha$ ensures the angle condition. 

It remains to show that
\begin{align*}
    \overline{\mathcal{B}}_p\big[\Gamma^{p,q_\alpha}_{\rm wave}\big] < 2^p \varpi_p^*. 
\end{align*}
To this end, by the fact that $q_\alpha < q^*(p)$, it is sufficient to prove that the energy in \eqref{eq:Bp-test-network} is strictly increasing with respect to $q\in(0,1)$.
We calculate
\begin{align}\label{eq:B_p-network}
\begin{split}
   &\frac{d}{dq} \Big( \big(2\E_{1,p}(q) +\K_{1,p}(q)\big)^{p-1}b(q) \Big) \\
   &\quad =\big(2\E_{1,p}(q)+\K_{1,p}(q)\big)^{p-2} \\
   &\quad \qquad\times \Big( (p-1)\big(2\E'_{1,p}(q)+\K'_{1,p}(q)\big)b(q)+(2\E_{1,p}(q)+\K_{1,p}(q)\big)b'(q) \Big).
\end{split}
\end{align}
It follows from Proposition~\ref{prop:EK-monotone} and Lemma~\ref{lem:monotone-w.r.t-q} that
\begin{align}\label{eq:K-varphi}
     (p-1) \K_{1,p}'(q)b(q) + \K_{1,p}(q)b'(q) >0.
\end{align}
We now show that 
\begin{align}\label{eq:E-varphi}
    (p-1)\E'_{1,p}(q)b(q) + \E_{1,p}(q)b'(q) >0. 
\end{align}
In view of \eqref{eq:diff-ellipint} and \eqref{eq:bibun-varphi}, we see that
\begin{align*}
   &\E_{1,p}(q)b'(q) + (p-1)\E'_{1,p}(q)b(q) \\
    =\, & (p-1)\E_{1,p}(q)\Big(q^{p-1}\K_{1,p}(q) +(1-\tfrac{2}{p})q^{p-3}\big(\E_{1,p}(q)-\K_{1,p}(q)\big) \Big) \\
    &\quad \, + (p-1)\frac{\E_{1,p}(q)-\K_{1,p}(q)}{q}\Big(q^{p-2}\E_{1,p}(q)+(q^p -q^{p-2})\K_{1,p}(q) \Big) \\
    =\, & (p-1)q^{p-3}\Big( (2-\tfrac{2}{p})\E_{1,p}(q)^2 + (2q^2+\tfrac{2}{p}-3)\E_{1,p}(q)\K_{1,p}(q)+(1-q^2)\K_{1,p}(q)^2 \Big).
\end{align*}
Here, we infer from Proposition~\ref{prop:EK-monotone} that 
\[
(2-\tfrac{2}{p})\E_{1,p}(q) - (2-\tfrac{2}{p}-q^2)\K_{1,p}(q)
=q(1-q^2)\mathrm{K}_{1,p}'(q) >0.
\]
Therefore we obtain
\begin{align*}
    & \E_{1,p}(q)b'(q) + (p-1)\E'_{1,p}(q)b(q) \\
    >\,& (p-1)q^{p-3}\Big( \E_{1,p}(q)(2-\tfrac{2}{p}-q^2)\K_{1,p}(q) 
    \\ 
    &\quad\quad\quad\quad\quad\quad\quad
    +(2q^2+\tfrac{2}{p}-3)\E_{1,p}(q)\K_{1,p}(q)+(1-q^2)\K_{1,p}(q)^2 \Big)\\
    =\,& (p-1)q^{p-3}(1-q^2)\big(\K_{1,p}(q)-\E_{1,p}(q) \big)\K_{1,p}(q) >0,
\end{align*}
which ensures \eqref{eq:E-varphi}.
Combining \eqref{eq:B_p-network} with \eqref{eq:K-varphi} and \eqref{eq:E-varphi}, we see that $\overline{\mathcal{B}}_p[\Gamma^{p,q}_{\rm wave}]$ is strictly increasing in $q\in(0,1)$.
\end{proof} 

We are now ready to complete the proof of Theorem~\ref{thm:inf-network-alpha}. 

\begin{proof}[Proof of Theorem~\ref{thm:inf-network-alpha}]
Let $\{\Gamma_j\} = \{ (\gamma_{1,j}, \gamma_{2,j}, \gamma_{3,j}) \} \subset \Theta(p,\alpha)$ be a minimizing sequence such that $\lim_{j\to \infty}\overline{ \mathcal{B}}_p[\Gamma_j] = \inf_{\Gamma \in \Theta(p,\alpha)}\overline{\mathcal{B}}_p[\Gamma]$. 
Then, possibly after rescaling, translation, reparameterization, and
taking a subsequence, either assertion (i) or (ii) in Lemma~\ref{lem:dichotomy} holds.
However, assertion (ii) does not occur in view of \eqref{eq:degenerate-lsc}.
Indeed, since the curves $\gamma_1$ and $\gamma_2$ in (ii) satisfy the assumption of Lemma~\ref{lem:nise-LiYau}, we have $\lim_{j\to\infty}(\sum_{i=1}^2\mathcal{L}[\gamma_{i,j}])^{p-1}(\sum_{i=1}^2\mathcal{B}_p[\gamma_{i,j}])\geq 2^p \varpi^*_p$; 
on the other hand, Lemma~\ref{lem:test-network} implies that
\[
\lim_{j\to\infty}\overline{\mathcal{B}}_p[\Gamma_j] =  \inf_{\Gamma \in \Theta(p,\alpha)}\overline{\mathcal{B}}_p[\Gamma] < 2^p \varpi^*_p.
\]
Therefore we have assertion (i).
Then the limit network $\Gamma$ is a desired minimizer in view of the fact that $\overline{\mathcal{B}}_p[\Gamma] \leq \lim_{j\to\infty}\overline{\mathcal{B}}_p[\Gamma_j] =  \inf_{\Gamma \in \Theta(p,\alpha)}\overline{\mathcal{B}}_p[\Gamma]$.
\end{proof}

%%%%%%%%%%%%%%%%%%%%%%%%%%%%%%%%%%%%%%
%%%%%%%%%%%%%%%%%%%%%%%%%%%%%%%%%%%%%%
%%%%%%%%%%%%%%%%%%%%%%%%%%%%%%%%%%%%%%
%%%%%%%%%%%%%%%%%%%%%%%%%%%%%%%%%%%%%%
%%%%%%%%%%%%%%%%%%%%%%%%%%%%%%%%%%%%%%
%%%%%%%%%%%%%%%%%%%%%%%%%%%%%%%%%%%%%%

\appendix

\section{First variation and multiplier method}

%In what follows, the letter $C$ denotes generic positive constants and it may have different values also within the same line.

\subsection{Preliminary estimates}

We collect elementary inequalities of power type.

\label{sect:A.1}
%%%%%%%%%%%%%%%%%%%%%%%%%%%%%%%%%%%%%%
\begin{lemma}  \label{lem:0728-1}
Let $p\in(1,\infty)$ and $a, b \in \R^n$ $(n\in\N)$. 
Then
\begin{align}
&\big| |a|^{p} - |b|^{p} \big| \leq p |a-b| \Big( |a|^{p-1} + |b|^{p-1}  \Big),  \label{eq:0729-1} \\
%\big| |a|^{p-2}a - |b|^{p-2}b \big|  \leq C |a-b| \Big( |a|^{p-2} + |b|^{p-2}  \Big) \ \ &\text{for} \ \ p>2, \\
%\big| |a|^{p-2}a - |b|^{p-2}b \big|   \leq |a-b|^{p-1}  \hspace{75pt} & \text{for} \ \ 1<p \leq 2. 
&\big| |a|^{p-2}a - |b|^{p-2}b \big|  \leq 
\begin{cases}
C |a-b|^{p-1} & (1<p \leq 2), \\
C |a-b| \Big( |a|^{p-2} + |b|^{p-2}  \Big) \ \ &(p>2), 
\end{cases}
\label{eq:0729-2}
\end{align}
where $C>0$ is a constant depending only on $p$. 
\end{lemma}
%%%%%%%%%%%%%%%%%%%%%%%%%%%%%%%%%%%%%%
\begin{proof}
By the convexity of the function $\R^n \ni x \mapsto |x|^p $, it follows that
\[
 |b|^{p} \geq |a|^{p}+ p \big( |a|^{p-2}a,  b-a \big), 
\]
which implies that
\begin{align*}
\big| |b|^{p} - |a|^{p} \big| 
&\leq  p \max\Big\{ \big| ( |a|^{p-2}a,  b-a ) \big| , \big| ( |b|^{p-2}b,  a-b) \big|  \Big\} \\
&\leq p \Big(   |a|^{p-1} | b-a |  + |b|^{p-1} | b-a |  \Big).
\end{align*}
Therefore we obtain \eqref{eq:0729-1}.

Next we prove \eqref{eq:0729-2} for $1<p\leq 2$. 
The idea comes from \cite[Remark 2.3]{CFZ}.
We assume, without loss of generality, that $0<|b| \leq |a| $ and $a\ne b$.
We see that
\begin{align}\label{eq:CFZ-spirit}
\begin{split}
\big| |a|^{p-2}a - |b|^{p-2}b \big|
& \leq \left| |a|^{p-2}a - |a|^{p-1} \frac{b}{|b|}  +|a|^{p-1} \frac{b}{|b|}  - |b|^{p-2}b \right| \\
& \leq |a|^{p-1} \left|  \frac{a}{|a|} -  \frac{b}{|b|} \right|   + \left| |a|^{p-1} -|b|^{p-1} \right|.
\end{split}
\end{align}
Since $x\mapsto |x|^{p-1}$ is $(p-1)$-H\"older continuous, we have
\[
\left| |a|^{p-1} -|b|^{p-1} \right| \leq C\big| |a| -|b| \big|^{p-1} \leq C \left|  a -b  \right|^{p-1}.
\]
Next, 
\begin{align*}
 |a|^{p-1} \left|  \frac{a}{|a|} -  \frac{b}{|b|} \right| 
& = |a|^{p-2} \left|  a -\frac{ |a| }{|b|}b +\frac{ |b| }{|b|}b -\frac{ |b| }{|b|}b \right| \\
& \leq   |a|^{p-2} \bigg( |a-b| + \left|  \frac{|b|}{|b|}b -  \frac{|a|}{|b|}b \right|  \bigg) 
\leq  2  |a|^{p-2}|a-b|.
\end{align*}
In the case $|a-b| \leq |a|$, we infer from $1<p<2$ that $|a|^{p-2} \leq |a-b|^{p-2}$, and hence
 \[ |a|^{p-2}|a-b| \leq |a-b|^{p-1}. \] 
In the case $|a-b| \geq |a|$, since $|a-b| \leq |a|+|b| \leq 2|a|$ holds, we see that
 \[ |a|^{p-2}|a-b| \leq 2 |a-b|^{p-1}. \] 
Therefore, we infer from \eqref{eq:CFZ-spirit} that 
\begin{align*}
\big| |a|^{p-2}a - |b|^{p-2}b \big|  \leq C|a-b|^{p-1}. 
\end{align*}

We turn to \eqref{eq:0729-2} for the case $p>2$.
According to \cite[page 73]{Lind}, it follows that for $p>2$
\[
\big| |a|^{p-2}a - |b|^{p-2}b \big|  \leq (p-1)|b-a| \int_0^1 |a+t(b-a)|^{p-2} dt, 
\]
from which the desired estimate follows immediately.
\end{proof}

%%%%%%%%%%%%%%%%%%%%%%%%%%%%%%%%%%%%%%
\begin{remark} \label{rem:0819-1}
Let $f, g, h \in L^p(0,1 ; \R^n)$.
Thanks to \eqref{eq:0729-1}, it follows from H\"older's inequality that
\begin{align*}
\int_0^1 \Big| |f(t)|^{p} -  |g(t)|^{p} \Big| \, dt  
\leq p \| f-g\|_{L^p(0,1)} \Big( \| f\|^{p-1}_{L^p(0,1)} +\| g \|^{p-1}_{L^p(0,1)}   \Big).
\end{align*}
In the case $1<p \leq 2$, it follows from H\"older's inequality and \eqref{eq:0729-2} that
\begin{align*}
\int_0^1 \Big| |f(t)|^{p-2}f(t) &-  |g(t)|^{p-2}g(t) \Big| |h(t)|\, dt  
\leq C  \| f-g\|^{p-1}_{L^p(0,1)}\| h\|_{L^p(0,1)}.
\end{align*}
In the case $p>2$, noting that $1/p + 1/p + (p-2)/p =1$, we observe from  H\"older's inequality and \eqref{eq:0729-2} that 
\begin{align*}
\int_0^1 \Big| |f(t)|^{p-2}f(t) &-  |g(t)|^{p-2}g(t) \Big| |h(t)|\, dt  \\
&\leq C  \| f-g\|_{L^p(0,1)}\| h\|_{L^p(0,1)}  \Big( \| f\|^{p-2}_{L^p(0,1)} +\| g \|^{p-2}_{L^p(0,1)}   \Big).
\end{align*}
\end{remark}
%%%%%%%%%%%%%%%%%%%%%%%%%%%%%%%%%%%%%%

\subsection{First variation of the $p$-bending energy}

In this subsection we compute the Fr\'echet derivative $D\B_p$ of $\B_p$.
Note that the Fr\'echet differentiability is nontrivial due to the strong nonlinearity of $\mathcal{B}_p$.
First, let us mention a known formula for the G\^ateaux derivative $d\mathcal{B}_p$ of $\B_p$ (see \cite[Lemma A.1]{MYarXiv2203} for a rigorous derivation).
Let $I:=(0,1)$.
For an immersed curve $\gamma:I\to\R^n$ ($t\in I$), let $ds:=|\gamma'|\,dt$ be the line element in the sense of a weighted measure on $I$.
Let $\partial_s$ denote the arclength derivative along $\gamma$, i.e., $\partial_s\psi=\frac{1}{|\gamma'|}\psi'$.
In particular, the curvature vector is then represented by $$\kappa:=\partial_s^2\gamma=\frac{\gamma''}{|\gamma'|^2} - \frac{( \gamma', \gamma'' )\gamma'}{| \gamma' |^4}.$$

%%%%%%%%%%%%%%%%%%%%%%%%%%%%%%%%%%%%%%
\begin{lemma}\label{lem:G-derivative-Bp}
Let $p\in(1,\infty)$ and $n\in\N$.
For an immersed curve $\gamma\in W_\mathrm{imm}^{2,p}(I;\R^n)$, let $\B_p$ be the $p$-bending energy defined by
$$\mathcal{B}_p[\gamma]:=\int_I|\kappa|^p ds.$$
Then the G\^ateaux derivative $d\mathcal{B}_p$ of $\B_p$ at $\gamma$ is given by, for $h \in W^{2,p}(I;\R^n)$,
\begin{align*}
\langle d\B_p[\gamma] ,  h \rangle 
    &= \int_I  \Big( (1-2p) |\kappa|^p (\partial_s\gamma, \partial_sh) +p|\kappa|^{p-2}(\kappa, \partial_s^2h)\Big) \, ds.
\end{align*}
\end{lemma}
%%%%%%%%%%%%%%%%%%%%%%%%%%%%%%%%%%%%%%

In the following we prove that $\mathcal{B}_p[\gamma]$ is Fr\'echet differentiable by showing the sufficient condition that $d\mathcal{B}_p[\gamma]$ is continuous with respect to $\gamma$
(cf.\ \cite[Proposition 4.8]{Zeid1}).
For a Banach space $X$, let $ \mathscr{L}(X, \R):=\Set{ F: X \to \R | F \text{ is linear and continuous}}$ and $\| \cdot \|_{ \mathscr{L}(X, \R)}$ denote the operator norm 
\[
\| F  \|_{ \mathscr{L}(X, \R)}:= \sup_{\|h\|_X \leq 1} \big| \langle F, h \rangle \big|.
\]
%Recall that the G\^ateaux derivative $dF$ of $F:X\to \R$ belongs to $ \mathscr{L}(X, \R)$.

%%%%%%%%%%%%%%%%%%%%%%%%%%%%%%%%%%%%%%
\begin{lemma}\label{lem:F-derivative-Bp}
Let $p\in(1,\infty)$ and $n\in\N$.
Let $X:=W^{2,p}(I;\R^n)$.
Then the G\^ateaux derivative $d\mathcal{B}_p$ of $\mathcal{B}_p$ is continuous at every %regular curve 
$\gamma\in W^{2,p}_{\rm imm}(I;\R^n)$ in $X$.
\end{lemma}
%%%%%%%%%%%%%%%%%%%%%%%%%%%%%%%%%%%%%%
\begin{proof}
Fix an arbitrary $\{\gamma_j\}_{j\in\N} \subset X$ satisfying ${\gamma}_j \to \gamma$ in $X$.
Then ${\gamma}_j \to \gamma$ in $C^1([0,1])$ also follows, and hence there is $c>0$ independent of $j$ such that
 $|\gamma'_j | >c$ for sufficiently large $j\in \N$.
The proof is completed by showing that
%We show the continuity of $d\mathcal{B}_p$ at $\gamma$:
\begin{align}\label{eq:Frechet-bibun-goal}
\big\|
d\B_p[\gamma] - d\B_p[\gamma_j] 
\big\|_{ \mathscr{L}(X, \R)}  \to 0 \quad \text{if} \quad \gamma_j \to \gamma \ \ \text{in}\ \ X.
\end{align}
In what follows, we use the same notation $C$ for positive constants depending only on $p$ and $\gamma$.
\if0
To this end, we prove that there is $\alpha>0$ such that for any $h\in W^{2,p}(I;\R^n)$ with $\|h\|_X \leq 1$ 
\begin{align} \label{eq:continuity-op}
\Big| \big\langle d\mathcal{B}_p[\gamma_j], h \big\rangle - \big\langle d\mathcal{B}_p[\gamma], h \big\rangle\Big|
= O\big(\| \gamma_j - \gamma \|_X^{\alpha} \big) \quad \text{as} \quad j\to\infty, 
\end{align}
where the right-hand side does not depend on $h$.
\fi

Fix an arbitrary $h \in X$ such that $\|h\|_X \leq 1$.
Let $\partial_{s_j}$ denote the arclength derivative along $\gamma_j$. 
Let $\kappa_j:=\partial_{s_j}^2\gamma_j
=\frac{\gamma_j''}{|\gamma_j'|^2} - \frac{( \gamma_j', \gamma_j'' )\gamma_j'}{| \gamma_j' |^4}.$
First we compute
\begin{align}\label{eq:gateau-estimate}
\begin{split}
\Big| \big\langle d\mathcal{B}_p[\gamma_j], h \big\rangle - \big\langle d\mathcal{B}_p[\gamma], h \big\rangle\Big|
&\leq  (2p-1) \int_I \Big| |\kappa|^p (\partial_s\gamma, h') - |\kappa_j|^p (\partial_{s_j}\gamma_j, h') \Big|\,dt \\
&\quad+ p \int_I \Big| |\kappa|^{p-2}(\kappa, |\gamma'|\partial_s^2h)- |\kappa_j|^{p-2}(\kappa_j, |\gamma_j'| \partial_{s_j}^2h) \Big| \,dt \\
&\ \ =: \mathbf{I} + \mathbf{J}.
\end{split}
\end{align}
We demonstrate the estimate of $\mathbf{J}$ in the right hand side.
It follows that 
\begin{align*}\notag\label{eq:int-J}
\mathbf{J} &\leq  p\int_0^1\big| |\kappa |^{p-2}\kappa - |\kappa_j |^{p-2}\kappa_j \big|  \big| |\gamma'|\partial_s^2h \big|  \, dt 
+ p\int_0^1 |\kappa_j |^{p-1} \big| |\gamma'|\partial_s^2 h - |\gamma_j'|\partial_{s_j}^2 h \big| \, dt \\
&\ \ =:\mathbf{J}_1 + \mathbf{J}_2.
\end{align*}
We apply Lemma~\ref{lem:0728-1} and H\"older's inequality, as in Remark~\ref{rem:0819-1}, to obtain 
\begin{align*}
\mathbf{J}_1 \leq  
\begin{cases}
C\|\kappa -\kappa_j \|^{p-1}_{L^p}  \big\| |\gamma'|\partial_s^2h \big\|_{L^p}  &(1<p\leq 2), \\
C\|\kappa -\kappa_j \|_{L^p}\Big( \|\kappa \|_{L^p}^{p-2} + \|\kappa_j \|_{L^p}^{p-2} \Big) \big\| |\gamma'|\partial_s^2h  \big\|_{L^p}  &(p>2).
\end{cases}
\end{align*}
Since $\|h\|_{X}\leq 1$ and $\gamma \in W^{2,p}_{\rm imm}(I;\R^n)$, we have $\| |\gamma'|\partial_s^2h \|_{L^p} \leq C$.
On the other hand, since $\partial_s^2 h = \frac{1}{|\gamma'|}(-\frac{(\gamma',\gamma'')}{|\gamma'|^3}h'+\frac{1}{|\gamma'|}h'' )$, H\"older's inequality combined with some $L^\infty$-estimates for first derivatives and $\|h\|_X\leq1$ implies that
\[
\mathbf{J}_2 \leq C \|\kappa_j\|_{L^p}^{p-1}\Big(\|\gamma''-\gamma_j''\|_{L^p} + \|\gamma'-\gamma_j'\|_{L^\infty} \Big).
\]
Similarly, we obtain 
\[
\mathbf{I} \leq C \Big( \|\kappa\|_{L^p}\|\gamma'-\gamma_j'\|_{L^\infty} + \|\kappa-\kappa_j\|_{L^p} \big(\|\kappa\|_{L^p}^{p-1} +\|\kappa_j\|_{L^p}^{p-1} \big) \Big).
\]
From the facts that $\gamma_j\to\gamma$ in $W^{2,p}$ and that $\|\kappa\|_{L^p},\|\kappa_j\|_{L^p}\leq C$ and $\|\kappa_j-\kappa\|_{L^p}\leq C\|\gamma''_j-\gamma''\|_{L^p}$, it follows that 
\[\mathbf{I}+\mathbf{J} \leq C\Big(\|\gamma'-\gamma_j'\|_{L^\infty} + \|\gamma''-\gamma_j''\|_{L^p} \Big).\] 
Combining this with \eqref{eq:gateau-estimate} and a well-known embedding $W^{1,\infty} \hookrightarrow W^{2,p}$, we obtain 
\begin{align}\notag %\label{eq:continuity-op}
\sup_{\|h\|_X \leq 1} \Big| \big\langle d\mathcal{B}_p[\gamma_j], h \big\rangle - \big\langle d\mathcal{B}_p[\gamma], h \big\rangle\Big|
\leq 
\begin{cases}
C\|\gamma -\gamma_j \|^{p-1}_X  &(1<p\leq 2), \\
C\|\gamma -\gamma_j \|_X &(p>2), 
\end{cases}
\end{align}
which implies \eqref{eq:Frechet-bibun-goal}. 
\end{proof}

By Lemma~\ref{lem:F-derivative-Bp}, the Fr\'echet derivative $D\mathcal{B}_p[\gamma]$ coincides with $d\mathcal{B}_p[\gamma]$.

\subsection{Lagrange multiplier theorem}
\label{sect:A.2}

Since our problem involves a non-local constraint, we use a Lagrange multiplier method (see e.g.\ \cite[Proposition 43.21]{Zeid3}).
In the following, for $P_0, P_1\in\R^n$ let 
\begin{align*}
    \mathcal{A}_{P_0, P_1}:=&\Set{ u \in W^{2,p}_{\rm imm}(0,1;\R^n) | u(0)=P_0, \ \  u(1)=P_1 }, \\
    X:=&\Set{ h \in W^{2,p}(0,1;\R^n) | h(0)=h(1)=0 }.
%\begin{array}{l}
%u_0 + t h \in A \text{ \ for \ } |t| \leq \varepsilon\\
%\end{array}
\end{align*}

%%%%%%%%%%%%%%%%%%%%%%%%%%%%%%%%%%%%%%
\begin{proposition} \label{prop:Lmultiplier}
Let $u_0\in \mathcal{A}_{P_0,P_1}$.
Let $U(u_0)$ be an open neighborhood of $u_0$ and
$F, G:U(u_0)\to \R$ be Fr\'echet differentiable at $u_0$, and denote
\[
M:=\Set{u\in \mathcal{A}_{P_0,P_1} | G(u)=0}.
%, \quad 
%\mathcal{A}^p_0 := \Set{ u \in W^{2,p} \big(0,1;\R^n \big) | u(0)= u(1)= 0}.
\]
Let $u_0 \in M$.
Then either of the following holds:
\begin{itemize}
\item[(a)] $\big\langle DG(u_0), h \big\rangle=0$ \  for all $h \in X$, 
\item[(b)] $u_0$ is a critical point of $F$ in $M$ if and only if there exists $\lambda\in \R$ such that
\begin{align}\label{eq:appendixLM}
\big\langle DF(u_0),  h \big\rangle + \lambda  \big\langle DG(u_0), h \big\rangle =0 \quad \text{for all} \ \ h\in X.
\end{align}
\end{itemize}
\end{proposition}
%%%%%%%%%%%%%%%%%%%%%%%%%%%%%%%%%%%%%%

%%%%%%%%%%%%%%%%%%%%%%%%%%%%%%%%%%%%%%
\begin{remark}\label{rem:linesegment-OK}
In the case of $F(\gamma)=\mathcal{B}_p[\gamma]$ and $G(\gamma)=\mathcal{L}[\gamma]-L$, when (a) is satisfied, i.e., $D\mathcal{L}[\gamma_0]=0$, then $\gamma_0$ is nothing but a line segment, which also satisfies $D\mathcal{B}_p[\gamma_0]=0$.
Therefore in our setting it suffices to consider (b) only.
\end{remark}
%%%%%%%%%%%%%%%%%%%%%%%%%%%%%%%%%%%%%%
%%%%%%%%%%%%%%%%%%%%%%%%%%%%%%%%%%%%%%
\begin{remark}[Clamped boundary condition]
Given $P_0, P_1, T_0, T_1 \in \R^n$ with $|T_0|=|T_1|=1$, let 
\begin{gather*}
\mathcal{A}_{P_0, P_1, T_0, T_1} := \Set{ \gamma \in \mathcal{A}_{P_0, P_1} |
\partial_s\gamma(0)=T_0,\ \partial_s\gamma(1)=T_1
}.
\end{gather*}
When we consider $\mathcal{A}_{P_0, P_1, T_0, T_1}$ instead of $\mathcal{A}_{P_0,P_1}$, 
then the assertion in Proposition~\ref{prop:Lmultiplier} holds by replacing $X$ in \eqref{eq:appendixLM} with 
\[
\Set{ h \in W^{2,p}(0,1;\R^n) | h(0)=h(1)=0, \ \ h'(0)=h'(1)=0 }.
\]
\end{remark}
%%%%%%%%%%%%%%%%%%%%%%%%%%%%%%%%%%%%%%

Thanks to Lemmata~\ref{lem:G-derivative-Bp} and \ref{lem:F-derivative-Bp}, we can explicitly calculate \eqref{eq:appendixLM} for the case $F(\gamma)=\mathcal{B}_p[\gamma]$ and $G(\gamma)=\mathcal{L}[\gamma]-L$. 
%Moreover, if $\gamma \in W^{2,p}_{\rm imm}(0,1; \R^2)$, then the arclength parameterization $\tilde{\gamma}$ of $\gamma$ belongs to $W^{2,p}(0,L;\R^2)$ (see \cite[Remark  A.3]{MYarXiv2203}). 
%Therefore by the change of variables as in \cite[Theorem A.2]{MYarXiv2203}, we obtain the representation of \eqref{eq:appendixLM} in terms of the arclength parameterization: 
Note that for any $\gamma\in W^{2,p}_\mathrm{imm}(0,1;\R^n)$ the arclength function $\sigma(t):=\int_0^t|\gamma'|$ is of class $W^{2,p}$, has strictly positive derivative, and maps $[0,1]$ to $[0,\mathcal{L}[\gamma]]$.
Hence the arclength reparameterization $\tilde{\gamma}:=\gamma\circ\sigma^{-1}$ is an element of $W^{2,p}_\mathrm{imm}(0,\mathcal{L}[\gamma];\R^n)$ with $|\tilde{\gamma}'|\equiv1$.
Applying the change of variables $s=\sigma(t)$ to Lemma~\ref{lem:G-derivative-Bp} combined with Lemma~\ref{lem:F-derivative-Bp}, and setting $\eta=h\circ\sigma^{-1}\in W^{2,p}(0,L;\R^n)$ for $h\in X$, we obtain the representation of \eqref{eq:appendixLM} in terms of the arclength parameterization:

%%%%%%%%%%%%%%%%%%%%%%%%%%%%%%%%%%%%%%
\begin{proposition}\label{prop:arclengthLM-pin}
Let $\lambda\in \R$.
A curve $\gamma \in W^{2,p}_{\rm imm}(0,1;\R^n)$ satisfies
\begin{align}\notag%\label{eq:multiplier-F}
\big\langle D\mathcal{B}_p[\gamma]+\lambda D\mathcal{L}[\gamma], h \big\rangle=0 
\end{align}
for any $h\in W^{2,p}(0,1;\R^n)$ with $h(0)=h(1)=0$ 
if and only if the arclength parametrization $\tilde{\gamma}\in W^{2,p}(0,L;\R^n)$ of $\gamma$ satisfies 
\begin{align} \notag%\label{eq:0525-100}
\int_0^L\Big( (1-2p) |\tilde{\vc}''|^p (\tilde{\vc}', \eta') +p|\tilde{\vc}''|^{p-2}(\tilde{\vc}'', \eta'') +\lambda (\tilde{\vc}', \eta')\Big) ds = 0
\end{align}
for any $\eta \in W^{2,p}(0,L;\R^n)$ with $\eta(0)=\eta(L)=0$.
\end{proposition}
%%%%%%%%%%%%%%%%%%%%%%%%%%%%%%%%%%%%%%

%%%%%%%%%%%%%%%%%%%%%%%%%%%%%%%%%%%%%%
%%%%%%%%%%%%%%%%%%%%%%%%%%%%%%%%%%%%%%
%%%%%%%%%%%%%%%%%%%%%%%%%%%%%%%%%%%%%%
%%%%%%%%%%%%%%%%%%%%%%%%%%%%%%%%%%%%%%
%%%%%%%%%%%%%%%%%%%%%%%%%%%%%%%%%%%%%%
%%%%%%%%%%%%%%%%%%%%%%%%%%%%%%%%%%%%%%

\bibliographystyle{abbrv}
\bibliography{ref_Miura-Yoshizawa-ver5}

\begin{thebibliography}{10}

\bibitem{AGP20}
J.~J. Arroyo, O.~J. Garay, and A.~P\'{a}mpano.
\newblock Boundary value problems for {E}uler-{B}ernoulli planar elastica. {A}
  solution construction procedure.
\newblock {\em J. Elasticity}, 139(2):359--388, 2020.

\bibitem{BGN12}
J.~W. Barrett, H.~Garcke, and R.~N\"{u}rnberg.
\newblock Elastic flow with junctions: variational approximation and
  applications to nonlinear splines.
\newblock {\em Math. Models Methods Appl. Sci.}, 22(11):1250037, 57, 2012.

\bibitem{BHV}
S.~Blatt, C.~Hopper, and N.~Vorderobermeier.
\newblock A regularized gradient flow for the {$p$}-elastic energy.
\newblock {\em Adv. Nonlinear Anal.}, 11(1):1383--1411, 2022.

\bibitem{BVH}
S.~Blatt, C.~P. Hopper, and N.~Vorderobermeier.
\newblock A minimising movement scheme for the {$p$}-elastic energy of curves.
\newblock {\em J. Evol. Equ.}, 22(2):Paper No. 41, 25, 2022.

\bibitem{CFZ}
T.~Cazenave, D.~Fang, and Z.~Han.
\newblock Continuous dependence for {NLS} in fractional order spaces.
\newblock {\em Ann. Inst. H. Poincar\'{e} C Anal. Non Lin\'{e}aire},
  28(1):135--147, 2011.

\bibitem{DLP19}
A.~Dall'Acqua, C.-C. Lin, and P.~Pozzi.
\newblock Elastic flow of networks: long-time existence result.
\newblock {\em Geom. Flows}, 4(1):83--136, 2019.

\bibitem{DNP20}
A.~Dall'Acqua, M.~Novaga, and A.~Pluda.
\newblock Minimal elastic networks.
\newblock {\em Indiana Univ. Math. J.}, 69(6):1909--1932, 2020.

\bibitem{DP17}
A.~Dall'Acqua and A.~Pluda.
\newblock Some minimization problems for planar networks of elastic curves.
\newblock {\em Geom. Flows}, 2(1):105--124, 2017.

\bibitem{DPP21}
G.~Del~Nin, A.~Pluda, and M.~Pozzetta.
\newblock Degenerate elastic networks.
\newblock {\em J. Geom. Anal.}, 31(6):6128--6170, 2021.

\bibitem{GMP19}
H.~Garcke, J.~Menzel, and A.~Pluda.
\newblock Willmore flow of planar networks.
\newblock {\em J. Differential Equations}, 266(4):2019--2051, 2019.

\bibitem{GMP20}
H.~Garcke, J.~Menzel, and A.~Pluda.
\newblock Long time existence of solutions to an elastic flow of networks.
\newblock {\em Comm. Partial Differential Equations}, 45(10):1253--1305, 2020.

\bibitem{LiYau}
P.~Li and S.~T. Yau.
\newblock A new conformal invariant and its applications to the {W}illmore
  conjecture and the first eigenvalue of compact surfaces.
\newblock {\em Invent. Math.}, 69(2):269--291, 1982.

\bibitem{Lind}
P.~Lindqvist.
\newblock {\em Notes on the {$p$}-{L}aplace equation (second edition)}.
\newblock University Jyv\"{a}skyl\"{a}, Department of Mathematics and
  Statistics, Report 161, 2017.

\bibitem{Lin98}
A.~Linn\'{e}r.
\newblock Explicit elastic curves.
\newblock {\em Ann. Global Anal. Geom.}, 16(5):445--475, 1998.

\bibitem{Love}
A.~E.~H. Love.
\newblock {\em A treatise on the {M}athematical {T}heory of {E}lasticity}.
\newblock Dover Publications, New York, 1944.
\newblock Fourth Ed.

\bibitem{Miura20}
T.~Miura.
\newblock Elastic curves and phase transitions.
\newblock {\em Math. Ann.}, 376(3-4):1629--1674, 2020.

\bibitem{Miura_LiYau}
T.~Miura.
\newblock Li--{Y}au type inequality for curves in any codimension,
  arXiv:2102.06597.

\bibitem{MMR23}
T.~Miura, M.~M\"{u}ller, and F.~Rupp.
\newblock Optimal thresholds for preserving embeddedness of elastic flows.
\newblock {\em to appear in Amer. J. Math.}, arXiv:2106.09549.

\bibitem{MYarXiv2203}
T.~Miura and K.~Yoshizawa.
\newblock Complete classification of planar $p$-elasticae, arXiv:2203.08535v2.

\bibitem{MR23}
M.~M{\"u}ller and F.~Rupp.
\newblock A {Li}-{Yau} inequality for the 1-dimensional {Willmore} energy.
\newblock {\em Adv. Calc. Var.}, 16(2):337--362, 2023.

\bibitem{NP20}
M.~Novaga and P.~Pozzi.
\newblock A second order gradient flow of {$p$}-elastic planar networks.
\newblock {\em SIAM J. Math. Anal.}, 52(1):682--708, 2020.

\bibitem{OPW20}
S.~Okabe, P.~Pozzi, and G.~Wheeler.
\newblock A gradient flow for the {$p$}-elastic energy defined on closed planar
  curves.
\newblock {\em Math. Ann.}, 378(1-2):777--828, 2020.

\bibitem{OW23}
S.~Okabe and G.~Wheeler.
\newblock The {{\(p\)}}-elastic flow for planar closed curves with constant
  parametrization.
\newblock {\em J. Math. Pures Appl. (9)}, 173:1--42, 2023.

\bibitem{PolPhD}
A.~Polden.
\newblock {\em Curves and surfaces of least total curvature and fouth-order
  flows}.
\newblock PhD thesis, Universit\"at {T}\"ubingen, 1996.

\bibitem{SW20}
N.~Shioji and K.~Watanabe.
\newblock Total {$p$}-powered curvature of closed curves and flat-core closed
  {$p$}-curves in {${\bf S}^2(G)$}.
\newblock {\em Comm. Anal. Geom.}, 28(6):1451--1487, 2020.

\bibitem{Takeuchi16}
S.~Takeuchi.
\newblock Legendre-type relations for generalized complete elliptic integrals.
\newblock {\em J. Class. Anal.}, 9(1):35--42, 2016.

\bibitem{vdM98}
H.~von~der Mosel.
\newblock Minimizing the elastic energy of knots.
\newblock {\em Asymptot. Anal.}, 18(1-2):49--65, 1998.

\bibitem{nabe14}
K.~Watanabe.
\newblock Planar {$p$}-elastic curves and related generalized complete elliptic
  integrals.
\newblock {\em Kodai Math. J.}, 37(2):453--474, 2014.

\bibitem{Glen13}
G.~Wheeler.
\newblock On the curve diffusion flow of closed plane curves.
\newblock {\em Ann. Mat. Pura Appl. (4)}, 192(5):931--950, 2013.

\bibitem{Ydcds}
K.~Yoshizawa.
\newblock The critical points of the elastic energy among curves pinned at
  endpoints.
\newblock {\em Discrete Contin. Dyn. Syst.}, 42(1):403--423, 2022.

\bibitem{Zeid3}
E.~Zeidler.
\newblock {\em Nonlinear functional analysis and its applications. {III}}.
\newblock Springer-Verlag, New York, 1985.

\bibitem{Zeid1}
E.~Zeidler.
\newblock {\em Nonlinear functional analysis and its applications. {I}}.
\newblock Springer-Verlag, New York, 1986.

\end{thebibliography}

\end{document}